\documentclass{article}

\usepackage{mathtools}
\usepackage{amsthm}
\usepackage{nicefrac}

\usepackage[smallerops]{newtx}
\usepackage[scale=1]{tx-ds}

\usepackage[english]{babel}
\usepackage{csquotes}
\usepackage[
    backend=biber,
    style=numeric,
    citestyle=numeric,
]{biblatex}
\DeclareNameFormat{sortname}{family-given}
\AtEveryBibitem{%
    \clearlist{location}%
    \clearfield{note}%
    \clearfield{urldate}%
    \clearfield{urlyear}%
    \ifentrytype{software}{}{%
        \clearfield{url}%
    }
}
\renewbibmacro{in:}{}
\usepackage[margin=25mm]{geometry}
\usepackage[bottom, stable]{footmisc}
\usepackage{authblk}

\usepackage{enumitem}
\setlist[enumerate]{
    nosep,
    label=\textit{\roman*}),
}
\newlist{draft}{enumerate}{1}
\setlist[draft]{
    nosep,
    label=\textit{\roman*}),
    before=\color{purple},
}
\usepackage{hyperref}

\theoremstyle{plain}
\newtheorem{theorem}{Theorem}
\newtheorem{lemma}[theorem]{Lemma}
\newtheorem{proposition}[theorem]{Proposition}
\newtheorem{corollaryenv}{Corollary}[theorem]
\newenvironment{corollary}[1][]{%
    \if\relax\detokenize{#1}\relax
    \else
        \ifcsname #1-used\endcsname
            \expandafter\xdef\csname #1-used\endcsname{\the\numexpr\csname #1-used\endcsname+1}%
        \else
            \expandafter\gdef\csname #1-used\endcsname{1}%
        \fi
        \renewcommand{\thecorollaryenv}{\ref{#1}.\csname #1-used\endcsname}%
    \fi%
    \begin{corollaryenv}%
}{%
    \end{corollaryenv}%
}

\theoremstyle{definition}
\newtheorem{definition}{Definition}
\theoremstyle{remark}
\newtheorem*{remark}{Remark}

\newcommand*{\beq}[1]{\begin{equation} \label{#1}}
\newcommand*{\eeq}{\end{equation}}
\newcommand*{\quadtext}[1]{\quad \text{#1} \quad}
\newcommand*{\qquadtext}[1]{\qquad \text{#1} \qquad}
\newcommand*{\defeq}{\coloneqq}
\newcommand*{\eqdef}{\eqqcolon}

\let\Re\relax
\let\Im\relax
\DeclareMathOperator{\Re}{Re}
\DeclareMathOperator{\Im}{Im}
\DeclareMathOperator{\supp}{supp}

\DeclarePairedDelimiter{\floor}{\lfloor}{\rfloor}
\DeclarePairedDelimiter{\ceil}{\lceil}{\rceil}
\DeclarePairedDelimiter{\abs}{\lvert}{\rvert}
\DeclarePairedDelimiter{\norm}{\lVert}{\rVert}
\DeclarePairedDelimiterXPP{\maxnorm}[1]{}{\lVert}{\rVert}{_{\mathrm{max}}}{#1}
\DeclarePairedDelimiterX{\inprod}[2]{\langle}{\rangle}{#1, #2}
\DeclareMathOperator{\AC}{AC}
\newcommand*{\loc}{_{\mathrm{loc}}}

\NewDocumentCommand{\diff}{s o m}{%
    \mathop{}\!\mathrm{d}%
    \IfNoValueTF{#2}{}{^{#2}}%
    \IfBooleanTF{#1}{\mkern -1mu #3}{#3}%
}

\DeclareMathOperator{\tr}{tr}
\newcommand*{\transpose}{\mathsf{T}}
\newenvironment{smallpmatrix}{%
    \bigl( \begin{smallmatrix}
}{\end{smallmatrix} \bigr)}

\newcommand*{\charf}[1]{\mathbb{1}_{#1}}
\newcommand*{\probto}[1][]{\xrightarrow[#1]{\mathbb{P}}}
\newcommand*{\lawto}[1][]{\xrightarrow[#1]{\mathrm{law}}}
\DeclarePairedDelimiterXPP{\bprob}[1]{\mathbb{P}}{[}{]}{}{#1}
\DeclarePairedDelimiterXPP{\pprob}[1]{\mathbb{P}}{(}{)}{}{#1}
\newcommand*{\expect}{\mathop{\mbox{}\mathbb{E}}}
\DeclarePairedDelimiterXPP{\bexpect}[1]{\mathbb{E}}{[}{]}{}{#1}
\newcommand*{\given}[1][]{\nonscript\>#1\vert\nonscript\>\mathopen{}}
\DeclarePairedDelimiterX{\quadvar}[1]{\langle}{\rangle}{#1}
\DeclarePairedDelimiterX{\crossvar}[2]{\langle}{\rangle}{#1, #2}

\DeclareMathOperator{\erf}{erf}

\newcommand*{\normal}[2]{\mathcal{N}(#1, #2)}
\DeclareMathOperator{\uniform}{Unif}
\newcommand*{\HBM}[1][\beta]{\mathcal{B}_{#1}}
\newcommand*{\ABM}{X_\beta}
\newcommand*{\UHP}{\mathbb{H}}
\newcommand*{\clUHP}{\overline{\mathbb{H}}_\infty}
\newcommand*{\RiemannSphere}{\mathbb{C}_\infty}
\newcommand*{\projection}{\mathscr{P}}

\DeclareMathOperator{\TM}{TM}
\DeclareMathOperator{\TMLP}{TM_{LP}}
\DeclareMathOperator{\TMLC}{TM_{LC}}
\DeclareMathOperator{\Hol}{Hol}
\newcommand*{\testfunc}{\varphi}
\DeclarePairedDelimiterXPP{\wronskian}[2]{\mathcal{W}}{(}{)}{}{#1, #2}
\DeclarePairedDelimiterXPP{\unswronskian}[2]{\tilde{\mathcal{W}}}{(}{)}{}{#1, #2}

\newcommand*{\shift}{E}
\newcommand*{\setshift}[1]{%
    \renewcommand*{\shift}{#1}%
}
\newcommand*{\timechange}{\eta_{\shift}}
\newcommand*{\bcphase}{\phi_{\shift}}

\newcommand*{\sine}{\mathrm{sine}_\beta}
\makeatletter
\newcommand*{\sinemat}{\@ifstar\@sinemat\@@sinemat}
\newcommand*{\@sinemat}{\tilde{R}_\beta}
\newcommand*{\@@sinemat}[1][\beta]{R_{#1}}
\newcommand*{\sineTM}{\@ifstar\@sineTM\@@sineTM}
\newcommand*{\@sineTM}{T_{\sinemat*}}
\newcommand*{\@@sineTM}{T_{\sinemat}}
\makeatother
\newcommand*{\sineWT}{m_{\sinemat}}
\newcommand*{\logtime}{\upsilon}
\newcommand*{\slogtime}{\upsilon_{\shift}}

\newcommand*{\Airy}{\mathrm{Airy}_{\beta}}

\newcommand*{\Bessel}{\mathrm{Bessel}_{\beta,a}}
\newcommand*{\Besselop}{\mathfrak{G}_{\beta,a}}
\newcommand*{\Besselweight}{w_{\beta,a}}
\newcommand*{\Besselcoeff}{p_{\beta,a}}
\newcommand*{\Besselsol}{\varpi_{\beta,a}}
\DeclarePairedDelimiterXPP{\BesselWronskian}[2]{\mathcal{W}_{\beta,a}}{(}{)}{}{#1, #2}
\newcommand*{\sBesselop}{\mathfrak{G}_{\beta,a,\shift}}
\makeatletter
\newcommand*{\sBesselmat}{\@ifstar\@sBesselmat\@@sBesselmat}
\newcommand*{\@sBesselmat}{\tilde{G}_{\beta,a,\shift}}
\newcommand*{\@@sBesselmat}[1][\beta]{G_{#1,a,\shift}}
\newcommand*{\sBesselTM}{\@ifstar\@sBesselTM\@@sBesselTM}
\newcommand*{\@sBesselTM}{T_{\sBesselmat*}}
\newcommand*{\@@sBesselTM}{T_{\sBesselmat}}
\makeatother
\newcommand*{\sBesselWT}{m_{\sBesselmat}}
\newcommand*{\sBesselSLtoCS}[1][\beta]{A_{#1,a,\shift}}
\newcommand*{\sBesself}[1][\beta]{\mathrm{f}_{#1,a,\shift}}
\newcommand*{\sBesselg}[1][\beta]{\mathrm{g}_{#1,a,\shift}}

\newcommand*{\diln}{c_{\shift}}
\newcommand*{\lasttime}{\tau_{\shift}}
\newcommand*{\penultime}{\tau_{\shift,\alpha}}
\newcommand*{\timedom}{\mathcal{I}_{\shift,\alpha}}

\newcommand*{\sBM}{B_{\shift}}

\newcommand*{\newfgcommands}[2]{%
    \expandafter\newcommand\expandafter*\csname #1\endcsname{#2_{\beta, a, \shift}}
    \expandafter\newcommand\expandafter*\csname #1f\endcsname{#2_{\beta, a, \shift}^{\mathrm{f}}}
    \expandafter\newcommand\expandafter*\csname #1g\endcsname{#2_{\beta, a, \shift}^{\mathrm{g}}}
}

\newfgcommands{amp}{\rho}
\newfgcommands{phase}{\xi}
\newcommand*{\diffamps}{\Delta^\rho_{\beta, a, \shift}}
\newcommand*{\diffphases}{\Delta^\xi_{\beta, a, \shift}}
\newcommand*{\sumamps}{\Sigma^\rho_{\beta, a, \shift}}
\newcommand*{\sumphases}{\Sigma^\xi_{\beta, a, \shift}}

\newcommand*{\sumdrift}{R_{\beta, a, \shift}}
\newcommand*{\sumdiff}{S_{\beta, a, \shift}}

\newcommand*{\GBM}{Z_{\beta, \shift}}
\newcommand*{\BesselABM}{Y_{\beta,a, \shift}}

\newcommand*{\unsBesselop}{\tilde{\mathfrak{G}}_{\beta,a}}
\newcommand*{\unsBesselweight}{\tilde{w}_{\beta,a}}
\newcommand*{\unsBesselcoeff}{\tilde{p}_{\beta,a}}
\newcommand*{\unsBesself}{\tilde{\mathrm{f}}_{\beta,a,\shift}}
\newcommand*{\unsBesselg}{\tilde{\mathrm{g}}_{\beta,a,\shift}}
\newcommand*{\unsBM}{\tilde{B}_{\shift}}

\makeatletter
\newcommand*{\reflectvec}[1]{\mathpalette\@reflectvec{#1}}
\newcommand*{\@reflectvec}[2]{\reflectbox{\(#1\vec{\reflectbox{\(#1#2\)}}\)}}
\makeatother
\newcommand*{\rBM}{\reflectvec{B}}
\newcommand*{\rBesselcoeff}{\reflectvec{p}_{\beta,a}}
\newcommand*{\rBesselweight}{\reflectvec{w}_{\beta,a}}
\newcommand*{\ramp}{r_{\beta,a}}
\newcommand*{\rphase}{\xi_{\beta,a}}

\counterwithin{equation}{section}
\title{Operator level hard edge to bulk transition in \(\beta\)-ensembles via canonical systems}
\author[1]{Vincent Painchaud\thanks{vpai@kth.se}}
\affil[1]{Department of Mathematics, KTH Royal Institute of Technology}
\date{\today}

\begin{document}

\maketitle

\begin{abstract}
The hard edge and bulk scaling limits of \(\beta\)-ensembles are described by the stochastic Bessel and sine operators, which are respectively a random Sturm--Liouville operator and a random Dirac operator. By representing both operators as canonical systems, we show that in a suitable high-energy scaling limit, the stochastic Bessel operator converges in law to the stochastic sine operator. This is first done in the vague topology of canonical systems' coefficient matrices, and then extended to the convergence of the associated Weyl--Titchmarsh functions and spectral measures. The proof relies on a coupling between the Brownian motions that drive the two operators, under which the convergence holds in probability. As part of the proof, we also obtain \(t\to-\infty\) asymptotics for solutions to the eigenvalue equation of the stochastic Bessel operator.
\end{abstract}

\section{Introduction}

A \(\beta\)-ensemble is a point process on a domain \(D \subseteq \mathbb{R}\) that admits the joint density
\beq{eq.betaensembles}
(\lambda_1, \hdots, \lambda_N) \mapsto \frac{1}{Z_{N,V,\beta}} \exp\Bigl( - \sum_{j=1}^N \beta N V(\lambda_j) \Bigr) \prod_{1\leq j<k\leq N} \abs{\lambda_j - \lambda_k}^\beta
\eeq
where \(V\colon D \to \mathbb{R}\) is a constraining potential, \(\beta > 0\) is a parameter usually called the inverse temperature, and \(Z_{N,V,\beta} > 0\) is a normalizing constant (see~\cite{anderson_introduction_2009} for background). An important problem in random matrix theory is to describe the local statistics of such a point process when \(N\) is large.

In the classical cases of \(\beta \in \{1,2,4\}\), these point processes enjoy a Pfaffian or determinantal structure, which allows to compute explicitly the correlation functions and therefore to obtain descriptions of scaling limits as \(N\to\infty\). With general \(\beta > 0\), these special structures are lost. Edelman and Sutton~\cite{edelman_random_2007} worked from the tridiagonal matrix models obtained by Dumitriu and Edelman~\cite{dumitriu_matrix_2002} and introduced an important idea: the local behavior of \(\beta\)-ensembles can be described by the spectra of random differential operators. Three differential operators were then defined and shown to be scaling limits of \(\beta\)-ensembles: the stochastic Airy operator for the soft edge limit~\cite{ramirez_beta_2011}, the stochastic Bessel operator for the hard edge limit~\cite{ramirez_diffusion_2009}, and the stochastic sine operator for the bulk limit~\cite{valko_sine_beta_2017}. The \(\Airy\), \(\Bessel\) and \(\sine\) point processes, which are the spectra of the corresponding operators, were shown to be universal for a large class of potentials~\cite{bekerman_transport_2015, bourgade_edge_2014, bourgade_bulk_2012, bourgade_universality_2014, krishnapur_universality_2016, rider_universality_2019}. We also note that while both the \(\Airy\) and \(\Bessel\) processes were first described as the spectra of the associated operators, the \(\sine\) process was in fact constructed before the operator in~\cite{killip_eigenvalue_2009} and in~\cite{valko_continuum_2009} (independently). 

The operators that describe the edge limits (Airy and Bessel) are a priori fundamentally different from the bulk (sine) operator: while the edge operators are random Schrödinger operators, the sine operator is a random Dirac operator. Nevertheless, it turns out that both of these classes of operators can be represented under the more general framework of canonical systems. This allows, for instance, to describe transitions from the edge operators to the bulk operator. The purpose of this paper is to use the canonical system framework to prove a hard edge to bulk transition at the operator level, adapting the strategy used recently for the soft edge to bulk transition in~\cite{painchaud_operator_2026} by E.\@ Paquette and the author. 

\paragraph{Bessel and sine operators.}

In order to state precisely our results, we now introduce the Bessel and sine operators. For \(\beta > 0\), let \(\HBM\) denote a hyperbolic Brownian motion with variance \(\nicefrac{4}{\beta}\) started at \(i\) in the upper half-plane, meaning that \(\HBM\) solves the stochastic differential equation
\beq{eq.HBMSDE}
\diff{\HBM}(t) = \frac{2}{\sqrt{\beta}} \Im\HBM(t) \diff{W}(t)
\qquadtext{with}
\HBM(0) = i
\eeq
where \(W\) is a standard complex Brownian motion, with independent standard real Brownian motions as real and imaginary parts. Then, set
\beq{eq.defsinemat}
\sinemat \defeq \frac{1}{2\Im\HBM} \ABM^\transpose \ABM
\quadtext{with}
\ABM \defeq \begin{pmatrix} 1 & -\Re\HBM \\ 0 & \Im\HBM \end{pmatrix}
\qquadtext{and}
J \defeq \begin{pmatrix} 0 & -1 \\ 1 & 0 \end{pmatrix}.
\eeq
The \emph{stochastic sine operator}, first defined in~\cite{valko_sine_beta_2017}, is the random differential operator sending \(u\colon (0,1) \to \mathbb{C}^2\) to
\beq{eq.defsine}
(\sinemat^{-1}\circ\logtime) Ju'
\qquadtext{and with boundary conditions}
\begin{cases}
    u(0) \parallel (1,0), \\
    u(1) \parallel (\Re\HBM(\infty), 1) \text{ if } \beta > 2
\end{cases}
\eeq
where \(\logtime(t) \defeq -\log(1-t)\) and \(\parallel\) denotes parallel. Under these boundary conditions, the stochastic sine operator is self-adjoint on the appropriate domain and has a discrete spectrum which is the \(\sine\) point process~\cite{valko_sine_beta_2017}. 

Now, for \(\beta > 0\) and \(a > -1\), let
\beq{eq.defBesselpw}
\Besselcoeff(t) \defeq \exp\Bigl( -at - \frac{2}{\sqrt{\beta}} B(t) \Bigr)
\qquadtext{and}
\Besselweight(t) \defeq \exp\Bigl( - (a+1)t - \frac{2}{\sqrt{\beta}} B(t) \Bigr),
\eeq
where \(B\) is a standard Brownian motion. The \emph{stochastic Bessel operator}, first defined in~\cite{ramirez_diffusion_2009}, is the random Sturm--Liouville operator \(\Besselop\) acting on a function \(f\colon (0,\infty) \to \mathbb{R}\) as
\beq{eq.defBessel}
\Besselop f = - \frac{1}{\Besselweight} \bigl( \Besselcoeff f' \bigr)'
\eeq
with Dirichlet boundary condition at \(0\) and Neumann boundary condition at infinity. For any given (continuous) Brownian path, this is a Sturm--Liouville operator, and in particular it is a.s.\@ self-adjoint on the appropriate domain. Its spectrum is the \(\Bessel\) point process~\cite{ramirez_diffusion_2009}. 

We also introduce a shifted and scaled version of the stochastic Bessel operator: \(\sBesselop \defeq \frac{2}{\sqrt{\shift}} (\Besselop - \shift)\), for \(\shift > 0\). Using asymptotics of Bessel functions in order to study the spacing between the eigenvalues of \(\Besselop\), it can be seen in the deterministic case \(\beta = \infty\) that the limit of \(\sBesselop\) as \(\shift\to\infty\) is the scaling limit in which the stochastic Bessel operator should converge to the stochastic sine operator.

\paragraph{Canonical system representation.}

Because the stochastic Bessel and sine operators are two different types of operators, we need a single framework that can encompass both of them, for which we use de Branges' theory of canonical systems~\cite{de_branges_hilbert_1968}. Although we refrain from providing a complete introduction to the theory of canonical systems here, we will still introduce the essential concepts as needed throughout the text. We refer the reader to~\cite[Section~2]{painchaud_operator_2026} or \cite[Chapter~2]{painchaud_canonical_2026} for a short introduction specifically tailored to our purposes, or to Remling's book~\cite{remling_spectral_2018} for a more complete overview. 

\begin{definition}
\label{def.CS}
A \emph{canonical system} on an interval \((a,b) \subseteq \mathbb{R}\) is a differential equation of the form
\[
Ju' = -zHu
\qquadtext{with}
J \defeq \begin{pmatrix} 0 & -1 \\ 1 & 0 \end{pmatrix}
\]
where \(u\colon (a,b) \to \mathbb{C}^2\), \(z \in \mathbb{C}\), and \(H\colon (a,b) \to \mathbb{R}^{2\times 2}\) is called the \emph{coefficient matrix}. Here, we will always assume that coefficient matrices are nonzero a.e., positive semi-definite a.e., and locally integrable entrywise.
\end{definition}

When \(H\) is invertible, the canonical system \(Ju' = -zHu\) can be written as \(-H^{-1}Ju' = zu\), and is therefore the eigenvalue equation for the Dirac differential operator \(u \mapsto - H^{-1}Ju'\). In general, a canonical system should still be thought of as an eigenvalue equation, but for a \emph{relation} on a suitable Hilbert space instead of an operator. Just like in the theory of other types of second-order differential operators, canonical systems are self-adjoint with real spectrum on suitable domains, which might have to be defined from boundary conditions depending on the behavior of the system near endpoints.

By inverting the matrix in the eigenvalue equation for the stochastic sine operator, we obtain a canonical system on \((0,1)\) with coefficient matrix \(\sinemat\circ\logtime\). Like any Sturm--Liouville operator, the stochastic Bessel operator can also be turned into a canonical system by an appropriate change of variables. Doing so, it can be shown that \(\sBesselop\) is equivalent to the canonical system on \((0,\infty)\) with coefficient matrix
\[
\sBesselmat \defeq \frac{\Besselweight\sqrt{\shift}}{2} \begin{pmatrix}
    \sqrt{\shift} \sBesselg^2 & \sBesself\sBesselg \\
    \sBesself\sBesselg & \frac{1}{\sqrt{\shift}} \sBesself^2
\end{pmatrix}
\]
where \(\sBesself\) and \(\sBesselg\) are fundamental solutions to \(\sBesselop h = 0\) with \(\sBesself(0) = \sBesselg'(0) = 1\) and \(\sBesself'(0) = \sBesselg(0) = 0\). Note that this coefficient matrix (like the one of any canonical system obtained from a Sturm--Liouville operator) is not invertible, which explains why we cannot use the theory of Dirac operators and we must stick with canonical systems theory.

\paragraph{Main results.}

Our first result describes the convergence of the canonical systems' coefficient matrices. The space of coefficient matrices on an interval \(\mathcal{I}\) can be given what we call the \emph{vague topology}, which is obtained by thinking of coefficient matrices as matrix-valued measures and testing them against compactly supported continuous functions on \(\mathcal{I}\). This construction results in a separable metric space. 

\begin{theorem}
\label{thm.vagueconv}
\setshift{E_n}
Let \(\mathcal{I} \defeq [0,1)\) if \(\beta \leq 2\) and \(\mathcal{I} \defeq [0,1]\) if \(\beta > 2\). For any diverging sequence \(\{\shift\}_{n\in\mathbb{N}} \subset (0,\infty)\), there are \(\mathscr{C}^1\) bijections \(\timechange\colon [0, 1 + \varepsilon_{\shift}) \to [0,\infty)\) where \(\varepsilon_{\shift} \to 0\) such that
\[
\sBesselmat* \defeq \timechange' (\sBesselmat\circ\timechange) \lawto[n\to\infty] \sinemat* \defeq \sinemat\circ\logtime
\]
in the vague topology of coefficient matrices on \(\mathcal{I}\).
\end{theorem}

\begin{remark}
The number \(\varepsilon_{\shift}\) appears here for technical reasons related to the possible discrepancy between the behavior of the Bessel and sine systems near the right endpoint. It always vanishes as \(\shift\to\infty\) (we will in fact take it to be exactly zero for some values of \(\beta\) and \(a\)) and \(\timechange\) should essentially be thought of as a time change between \((0,1)\) and \((0,\infty)\). The precise definition of \(\timechange\) will be given in~\eqref{eq.timechange}.
\end{remark}

From this result, we can then show that the solutions of the canonical systems also converge. More precisely, we can deduce the convergence of their \emph{transfer matrices}. Here, by the transfer matrix of a canonical system on \((a,b)\), we mean a function \(T\colon [a,b) \times \mathbb{C} \to \mathbb{C}^{2\times 2}\) such that for each \(z \in \mathbb{C}\), \(T(\cdot,z)\) is a (matrix) solution to the canonical system with initial condition \(T(a,z) = I_2\).

\begin{corollary}
\label{cor.TMconv}
\setshift{E_n}
Let \(\sBesselTM*, \sineTM*\colon \mathcal{I} \times \mathbb{C} \to \mathbb{C}^{2\times 2}\) be the transfer matrices of the canonical systems with coefficient matrices \(\sBesselmat*\) and \(\sinemat*\) respectively. Then \(\sBesselTM* \to \sineTM*\) in law compactly on \(\mathcal{I} \times \mathbb{C}\) as \(n\to\infty\).
\end{corollary}

The vague topology, however, is usually not strong enough to capture the behavior of the canonical systems' spectra. To extend the convergence to the spectrum, we use Weyl theory. A canonical system always has a \emph{Weyl--Titchmarsh function}, which is a generalized Herglotz function (i.e., a holomorphic map from the upper half-plane \(\UHP\) to its closure \(\clUHP\) in the Riemann sphere) and essentially the Sieltjes transform of its spectral measure (we will come back to the precise definition in Section~\ref{sec.spectralconv}). What is missing from the convergence of transfer matrices to get that of Weyl--Titchmarsh functions is the convergence of the systems' boundary conditions. From there, we obtain the following result.

\begin{theorem}
\label{thm.WTconv}
\setshift{E_n}
Let \(\{\shift\}_{n\in\mathbb{N}} \subset (0,\infty)\) be a diverging sequence, and let \(\sBesselWT, \sineWT\colon \UHP \to \clUHP\) be the Weyl--Titchmarsh functions of the canonical systems with coefficient matrices \(\sBesselmat\) and \(\sinemat\) respectively. Then \(\sBesselWT \to \sineWT\) in law compactly on \(\UHP\) as \(n\to\infty\), and this holds jointly with the convergence of transfer matrices in Corollary~\ref{cor.TMconv}. In particular, the spectral measures of the corresponding systems converge vaguely in law.
\end{theorem}

\begin{remark}
See the discussion in Section~\ref{sec.intro.ppconv} for the link between this result and the convergence of the associated point processes.
\end{remark}

As part of the proof of Theorem~\ref{thm.WTconv}, we obtain the following asymptotics of solutions to \(\Besselop f = \lambda f\) towards \(-\infty\) for \(\lambda > 0\). 

\begin{theorem}
\label{thm.asymptotics}
Let \(\Besselop\) be defined on the full real line from a two-sided Brownian motion \(B\). If \(f\) solves \(\Besselop f = \lambda f\) for \(\lambda > 0\) and if \(f(0)\) and \(f'(0)\) are independent of the restriction of \(B\) to \((-\infty,0)\), then for \(t \geq 0\),
\[
f(-t) = C_f \lambda^{\nicefrac{-1}{4}} e^{(\frac{1}{2\beta} - \frac{a}{2} - \frac{1}{4})t + X(t)} \cos\rphase(t)
\quadtext{and}
f'(-t) = C_f \lambda^{\nicefrac{1}{4}} e^{(\frac{1}{2\beta} - \frac{a}{2} + \frac{1}{4})t + X(t)}\sin\rphase(t)
\]
where \(C_f^2 \defeq f^2(0) + {f'}^2(0)\), \(\rphase\) is a process such that \(\rphase(t) - 2\pi \floor[\big]{\frac{\rphase(t)}{2\pi}} \to U \sim \uniform[0,2\pi)\) in law as \(t\to\infty\), and \(X\) is a process such that for any \(\varepsilon, \delta > 0\), there is a \(C > 0\) for which
\[
\bprob[\Big]{\forall t \geq 0, \abs{X(t)} \leq C(1 + t^{\nicefrac{1}{2} + \delta})} \geq 1 - \varepsilon.
\]
\end{theorem}

\paragraph{Related work on transitions in \(\beta\)-ensembles.}

The transition from the hard edge to the bulk was studied previously at the level of point processes by Holcomb in~\cite{holcomb_random_2018}. Specifically, she proved that for any \(\beta > 0\) and \(a > 0\), the square root of the \(\Bessel\) point process converges to the \(\sine\) point process in a suitable scaling limit. The results from Theorems~\ref{thm.vagueconv} and~\ref{thm.WTconv} can be seen as an extension of this result to the operator level, and allowing \(a \in (-1,0]\) as well.

There are other similar transitions that appear in the theory of \(\beta\)-ensembles. An important example is the transition from the soft edge to the bulk, which was first proven by Valkó and Virág at the level of the point processes in~\cite{valko_continuum_2009}. More recently, in~\cite{painchaud_operator_2026}, Paquette and the author used canonical systems theory to prove that this transition also occurs at the operator level, and the present work leverages the same general strategy. Part of the technical analysis involved in the proofs of Theorems~\ref{thm.vagueconv} and~\ref{thm.WTconv} and their counterparts in the soft edge to bulk transition can be adapted from one problem to the other, illustrating the applicability of the canonical system approach to the study of scaling limits of \(\beta\)-ensembles and relations between them. We emphasize that there are important aspects of the hard edge to bulk transition that are not present in the soft edge case. In particular, at \(a=1\) the Bessel operator undergoes a transition---inexistent in the Airy case---that has an effect on its behavior at infinity and must be accounted for in the proof of Theorem~\ref{thm.WTconv}. 

Another example of transition in \(\beta\)-ensembles theory is that from the hard edge to soft edge, which occurs when taking \(a \to \infty\) in the stochastic Bessel operator or point process. As the stochastic Airy and Bessel operators are both random (generalized) Sturm--Liouville operators, one can exploit their resolvents to describe a transition between them. This was done by Dumaz, Li and Valkó in~\cite{dumaz_operator_2021}, where it is shown that this transition occurs at the operator level, in the norm resolvent sense. For comparison with this type of result, we note that a resolvent \((\mathcal{S}_H - z)^{-1}\) can be defined for a canonical system with coefficient matrix \(H\) given appropriate boundary conditions, and in fact Theorems~\ref{thm.vagueconv} and~\ref{thm.WTconv} imply that for any \(z \in \mathbb{C} \setminus \mathbb{R}\) and any compactly supported continuous function \(\varphi\colon \mathcal{I} \to \mathbb{C}^2\),
\beq{eq.resolventconvergence}
\setshift{E_n}
(\mathcal{S}_{\sBesselmat*} - z)^{-1} \varphi \lawto[n\to\infty] (\mathcal{S}_{\sinemat*} - z)^{-1} \varphi
\eeq
compactly in \(z\). While this is reminiscent of a convergence of resolvents in a strong operator topology, it should be noted that the resolvents that appear here are not defined on the same Hilbert space, so~\eqref{eq.resolventconvergence} does not (a priori) describe an actual operator level convergence. We won't go into details about resolvent convergence here, and we rather refer the interested reader to \cite[Appendix~B]{painchaud_operator_2026} for precise definitions and for links with other types of convergence of canonical systems.

\paragraph{Remarks on the point process convergence.}
\label{sec.intro.ppconv}

Theorem~\ref{thm.WTconv} shows the vague convergence in law of the spectral measures of the shifted Bessel operator \(\sBesselop\) to that of the sine operator. These spectral measures are pure point and have positive masses precisely at the (simple) eigenvalues of the two operators, but their vague convergence is not strong enough to guarantee the vague convergence of the associated eigenvalue point processes, since in principle spectral masses could vanish or merge with others in the limit. 

Extending the result to the convergence of the eigenvalue point processes could be done through an analysis of the spectral weights. The spectral weights of the classical discrete ensembles follow a Dirichlet distribution, and those of the stochastic sine operator are known to be independent of each other and of the eigenvalues, and to follow a \(\Gamma(\nicefrac{\beta}{2}, \nicefrac{4}{\beta})\) distribution (in the shape-scale parametrization)~\cite[Proposition 3]{valko_palm_2023}. The spectral weights of the stochastic Bessel operator are not known, but they are conjectured to have the following form: the weight at an eigenvalue \(\lambda\) should be \(\lambda \Gamma_\lambda\) where \(\{\Gamma_\lambda\}\) is a collection of independent \(\Gamma(\nicefrac{\beta}{2}, \nicefrac{2}{\beta})\) random variables also independent of the eigenvalues~\cite[Conjecture~3.1]{painchaud_canonical_2026}. This would imply that the weights of the shifted and scaled operator \(\sBesselop\) are \(2(1 + \nicefrac{\lambda}{2\sqrt{\shift}}) \Gamma_\lambda\), and if this holds then the convergence of the eigenvalue point processes should not be hard to deduce from Theorem~\ref{thm.WTconv}---see~\cite[Section~4.4.4]{painchaud_canonical_2026} for a more detailed discussion.

\paragraph*{Organization of the paper.}

The rest of the paper is organized as follows. In Section~\ref{sec.setup}, we introduce basic properties of the stochastic Bessel operator, and we build a canonical system version of the shifted and scaled operator \(\sBesselop\). From this, we give an intuitive overview of the proof of Theorem~\ref{thm.vagueconv}. Working towards the full proof, we build in Section~\ref{sec.coupling} a coupling between the Bessel and sine canonical systems, which we use to derive the asymptotic behavior of processes that appear in the entries of the Bessel system's coefficient matrix. From this, we finish the proof of Theorem~\ref{thm.vagueconv} in Section~\ref{sec.vagueconv}. Finally, Section~\ref{sec.spectralconv} is dedicated to the proofs of Theorems~\ref{thm.WTconv} and~\ref{thm.asymptotics}. 

\paragraph*{Acknowledgements.}

Most of this work was done when the author was a PhD student at McGill University, supported by an NSERC CGS-D scholarship as well as an FRQ doctoral scholarship (doi:~\href{https://doi.org/10.69777/319962}{\texttt{10.69777/319962}}). Part of this work was also conducted while the author was in residence at Institut Mittag-Leffler in Djursholm, Sweden in Fall 2024, so we acknowledge support from the Swedish Research Council under grant no.~2021-06594. The author would also like to thank Elliot Paquette for lots of helpful discussions and comments on the manuscript.

\section{The Bessel canonical system and the setup for the convergence}
\label{sec.setup}

In this section, we introduce basic properties of the stochastic Bessel operator \(\Besselop\), and we build a canonical system version of the shifted and scaled operator \(\sBesselop \defeq \frac{2}{\sqrt{\shift}} (\Besselop - E)\). We will see that the coefficient matrix of this canonical system involves solutions to \(\sBesselop f = 0\). From a change of variables of these solutions into polar coordinates, we then give a heuristic argument for the convergence of \(\sBesselop\) to the stochastic sine operator at the level of canonical systems, which will serve as a plan for the rigorous proofs presented in Sections~\ref{sec.coupling} and~\ref{sec.vagueconv}.

\subsection{The Bessel operator and its basic properties}

\subsubsection{Definition and domain}

Recall from the introduction that for \(\beta > 0\) and \(a > -1\), the stochastic Bessel operator acts on \(f\colon (0,\infty) \to \mathbb{R}\) as \(\Besselop f = - \frac{1}{\Besselweight} (\Besselcoeff f')'\) where
\[
\Besselcoeff(t) \defeq \exp\Bigl( - at - \frac{2}{\sqrt{\beta}} B(t) \Bigr)
\qquadtext{and}
\Besselweight(t) \defeq \exp\Bigl( - (a+1) t - \frac{2}{\sqrt{\beta}} B(t) \Bigr)
\]
for a standard Brownian motion \(B\). For almost every Brownian path, this is a well-defined Sturm--Liouville operator on the domain \(\mathfrak{D}_{\Besselop} \defeq \bigl\{ f \in \AC\loc(0,\infty) : \Besselcoeff f' \in \AC\loc(0,\infty) \bigr\}\), where \(\AC\loc(a,b)\) denotes the set of locally absolutely continuous functions on \((a,b)\), that is, functions \(f\colon (a,b) \to \mathbb{C}\) such that \(f(t) = f(t_0) + \int_{t_0}^t g(t) \diff{t}\) for some \(t_0 \in (a,b)\) and some \(g \in L^1\loc(a,b)\). To describe the spectral properties of this operator, we will use Weyl theory. We recall some basic results and terminology here, but we refer the reader to~\cite{eckhardt_weyl-titchmarsh_2013} for a complete introduction.

As a Sturm--Liouville operator, \(\Besselop\) is said to be \emph{limit circle} at \(0\) (or \(\infty\)) if for all \(z \in \mathbb{C}\), all solutions to \(\Besselop f = zf\) lie in \(L^2\bigl( (0,\infty), \Besselweight(t) \diff{t} \bigr)\) near \(0\) (or \(\infty\)), and \emph{limit point} at \(0\) (or \(\infty\)) otherwise. By the Weyl alternative theorem (see e.g.~\cite[Lemma~4.1]{eckhardt_weyl-titchmarsh_2013}), this alternative does not depend on \(z\). A Sturm--Liouville operator becomes self-adjoint when its domain is restricted by boundary conditions at limit circle endpoints, but not at limit point endpoints (see e.g.~\cite[Sections~5 and~6]{eckhardt_weyl-titchmarsh_2013}). 

The stochastic Bessel operator is always limit circle at \(0\), and we will see in Proposition~\ref{prop.Weyl} that it is limit circle at infinity when \(\abs{a} < 1\) but limit point at infinity when \(a \geq 1\). Hence, a boundary condition is always needed at \(0\) to make it self-adjoint, but a boundary condition at infinity is only needed when \(\abs{a} < 1\). From~\cite{ramirez_diffusion_2009}, the boundary conditions that give \(\Besselop\) the \(\Bessel\) point process as its spectrum are a Dirichlet boundary condition \(f(0) = 0\) on the left and a Neumann boundary condition \(\lim_{t\to\infty} \Besselcoeff(t) f'(t) = 0\) on the right.

\subsubsection{Some properties of solutions to \texorpdfstring{\(\Besselop f = z f\)}{Gf = zf}}
\label{sec.solprops}

The solutions to \(\Besselop f = zf\) can be described as solutions to a stochastic differential equation. Indeed, if \(f \in \mathfrak{D}_{\Besselop}\), then by definition \(\Besselcoeff f'\) is absolutely continuous, so
\[
\diff{\bigl( \Besselcoeff f' \bigr)}(t)
    = \bigl( \Besselcoeff f' \bigr)'(t) \diff{t}
    = - z \Besselweight(t) f(t) \diff{t}.
\]
As long as \(f(0)\) and \(f'(0)\) are independent of the Brownian motion, we can also apply Itô's formula to obtain
\[
\diff{\bigl( \Besselcoeff f' \bigr)}(t)
    = f'(t) \diff{\Besselcoeff}(t) + \Besselcoeff(t) \diff{f'}(t) + \diff{\crossvar{\Besselcoeff}{f'}}(t).
\]
Comparing the two expressions, we get
\[
\diff{f'}(t) = - \frac{1}{\Besselcoeff(t)} \Bigl( z \Besselweight(t) f(t) \diff{t} + f'(t) \diff{\Besselcoeff}(t) + \diff{\crossvar{\Besselcoeff}{f'}}(t) \Bigr).
\]
Now,
\[
\diff{\Besselcoeff}(t)
    = - \Bigl( a - \frac{2}{\beta} \Bigr) \Besselcoeff(t) \diff{t} - \frac{2}{\sqrt{\beta}} \Besselcoeff(t) \diff{B}(t)
\qquadtext{so}
\diff{\crossvar{\Besselcoeff}{f'}}(t)
    = - \frac{4}{\beta} \Besselcoeff(t) f'(t) \diff{t},
\]
and we find that \(f\) solves
\beq{eq.BesselSDE}
\begin{aligned}
    \diff{f}(t) & = f'(t) \diff{t}, \\
    \diff{f'}(t) & = - z e^{-t} f(t) \diff{t} + \Bigl( a + \frac{2}{\beta} \Bigr) f'(t) \diff{t} + \frac{2}{\sqrt{\beta}} f'(t) \diff{B}(t).
\end{aligned}
\eeq

There is another characterization of these solutions that will be useful in the sequel, and that is obtained from the fact that the equation \(\Besselop f = 0\) can in fact be solved explicitly. Indeed, this equation directly reduces to \((\Besselcoeff f')' = 0\), which forces \(\Besselcoeff f'\) to be constant. If that constant is zero, then \(f\) itself must be constant, and otherwise \(f\) must be a constant plus a multiple of
\beq{eq.defBesselsol}
\Besselsol(t)
    \defeq \int_0^t \frac{1}{\Besselcoeff(s)} \diff{s}
    = \int_0^t \exp\Bigl( as + \frac{2}{\sqrt{\beta}} B(s) \Bigr) \diff{s}.
\eeq
Hence, \(1\) and \(\Besselsol\) are a pair of fundamental solutions to \(\Besselop f = 0\), and the general solution is \(f = f(0) + f'(0) \Besselsol\). 

This leads to the following.

\begin{proposition}
\label{prop.Weyl}
\(\Besselop\) is a.s.\@ limit circle at infinity if \(\abs{a} < 1\), and a.s.\@ limit point at infinity if \(a \geq 1\).
\end{proposition}

\begin{proof}
By the Weyl alternative, it suffices to check the behavior of solutions to \(\Besselop f = zf\) at a single value of \(z \in \mathbb{C}\). With \(z = 0\), we have seen above that \(1\) and \(\Besselsol\) are a fundamental pair of solutions, and by definition of \(\Besselweight\) it is clear that \(1 \in L^2\bigl( (0,\infty), \Besselweight(t) \diff{t} \bigr)\) a.s.\@ in any case. Therefore, \(\Besselop\) is a.s.\@ limit circle at infinity if \(\Besselsol \in L^2\bigl( (0,\infty), \Besselweight(t) \diff{t} \bigr)\) a.s., and it is a.s.\@ limit point at infinity if \(\Besselsol \notin L^2\bigl( (0,\infty), \Besselweight(t) \diff{t} \bigr)\) a.s.

In order to analyze the behavior of \(\Besselsol\), we use the following result, which follows from basic properties of Brownian motion (see e.g.~\cite[Proposition~19]{painchaud_operator_2026} for a proof): for any \(\varepsilon > 0\), there is a \(C > 0\) such that
\beq{eq.Weyl.goodevent}
\bprob[\Big]{\forall t \geq 0, \frac{2}{\sqrt{\beta}} \abs{B(t)} < C\bigl( 1 + t^{\nicefrac{3}{4}} \bigr)} \geq 1 - \varepsilon. 
\eeq
Note that for any \(\delta > 0\) there is a \(T > 0\) such that \(C(1 + t^{\nicefrac{3}{4}}) \leq \delta t\) for all \(t \geq T\). The rest of the proof is split in three cases: \(\abs{a} < 1\), \(a > 1\) and \(a = 1\). 

Suppose \(\abs{a} < 1\) and take \(\delta < \frac{1-a}{3} \wedge \frac{1}{2}\). Then on the good event from~\eqref{eq.Weyl.goodevent}, if \(t \geq T\) with \(T\) as above,
\[
\Besselsol(t)
    \leq \int_0^T e^{as + C(1 + s^{\nicefrac{3}{4}})} \diff{s} + \int_T^t e^{(a+\delta)s} \diff{s}
    = C_T + \frac{e^{(a+\delta)t}}{a+\delta}
\quadtext{where}
C_T \defeq \int_0^T e^{as + C(1+s^{\nicefrac{3}{4}})} \diff{s} - \frac{e^{(a+\delta)T}}{a+\delta}.
\]
Therefore
\begin{align*}
\norm{\Besselsol}_{\Besselweight}^2
    & = \int_0^\infty \Besselsol^2(t) \Besselweight(t) \diff{t} \\
    & \leq \int_0^T \biggl( \int_0^t e^{as + C(1+s^{\nicefrac{3}{4}})} \diff{s} \biggr)^2 e^{-(a+1)t + C(1+t^{\nicefrac{3}{4}})} \diff{t} + \int_T^\infty \biggl( C_T + \frac{e^{(a+\delta)t}}{a+\delta} \biggr)^2 e^{-(a+1-\delta)t} \diff{t} \\
    & = C_T' + C_T^2 \int_T^\infty e^{-(a+1-\delta)t} \diff{t} + \frac{2C_T}{a+\delta} \int_T^\infty e^{-(1-2\delta)t} \diff{t} + \frac{1}{(a+\delta)^2} \int_T^\infty e^{-(1-a-3\delta)t} \diff{t}
\end{align*}
where \(C_T'\) is the first term of the previous line. By our choice of \(\delta\), this is finite. So \(\norm{\Besselsol}_{\Besselweight}^2 < \infty\) with probability at least \(1 - \varepsilon\) for any \(\varepsilon > 0\), and it follows that \(\norm{\Besselsol}_{\Besselweight}^2 < \infty\) a.s. This shows that \(\Besselop\) is a.s. limit circle at infinity. 

With \(a > 1\) and \(\delta < \frac{a-1}{3} \wedge \frac{1}{2}\), essentially the same argument as above shows that \(\Besselop\) is a.s.\@ limit point at infinity. Indeed, on the good event from~\eqref{eq.Weyl.goodevent}, we now have for \(t \geq T\) that
\[
\Besselsol(t) \geq \tilde{C}_T + \frac{e^{(a-\delta)t}}{a-\delta}
\qquadtext{where}
\tilde{C}_T \defeq \int_0^T e^{as-C(1+s^{\nicefrac{3}{4}})} \diff{s} - \frac{e^{(a-\delta)T}}{a-\delta},
\]
and therefore
\begin{align*}
\norm{\Besselsol}_{\Besselweight}^2
    & \geq \int_T^\infty \biggl( \tilde{C}_T + \frac{e^{(a-\delta)t}}{a-\delta} \biggr)^2 e^{-(a+1+\delta)t} \diff{t} \\
    & = \tilde{C}_T^2 \int_T^\infty e^{-(a+1+\delta)t} \diff{t} + \frac{2\tilde{C}_T}{a-\delta} \int_T^\infty e^{-(1+2\delta)t} \diff{t} + \frac{1}{(a-\delta)^2} \int_T^\infty e^{(a-1-3\delta)t} \diff{t},
\end{align*}
which is infinite because the last integral diverges while the other two converge. It follows that \(\norm{\Besselsol}_{\Besselweight}^2 = \infty\) a.s., so that \(\Besselop\) is a.s.\@ limit point at infinity.

This only leaves out the case \(a = 1\), which is more delicate as in that case the bound on the Brownian motion from~\eqref{eq.Weyl.goodevent} is not strong enough to allow us to conclude. To replace it, we define a sequence of stopping times by setting \(\tau_0 \defeq \inf\{t \geq 1 : B(t) = 0\}\) and then recursively \(\tau_n \defeq \inf\{t \geq \tau_{n-1} + 2 : B(t) = 0\}\) for \(n \in \mathbb{N}\). Then, let \(A_n \defeq \bigl\{ \forall t \in [\tau_n, \tau_n + 1], \abs{B(t)} \leq 1 \bigr\}\). By the strong Markov property of Brownian motion, the \(A_n\)'s are independent, and for every \(n\)
\[
\pprob{A_n}
    = \bprob[\Big]{\sup_{t\in[\tau_n, \tau_n+1]} \abs{B(t)} \leq 1}
    = \bprob[\Big]{\sup_{t\in[0,1]} \abs{B(t)} \leq 1}
    \geq 1 - 2 \bprob[\Big]{\sup_{t\in[0,1]} B(t) > 1}.
\]
The reflection principle shows that
\(\bprob[\big]{\sup_{t\in[0,1]} B(t) > 1}
    = 2 \bprob{B(1) > 1}
    = 1 - \erf\bigl( \nicefrac{1}{\sqrt{2}} \bigr)\), so
\(\pprob{A_n} \geq 2\erf\bigl( \nicefrac{1}{\sqrt{2}} \bigr) - 1\). This number is strictly positive, so the second Borel--Cantelli lemma implies that \(\pprob[\big]{\bigcap_{n\in\mathbb{N}} \bigcup_{k\geq n} A_k} = 1\). On that event, there is always a subsequence \(\{A_{n_k}\}_{k\in\mathbb{N}}\) in which each event occurs, and this gives us a sequence of intervals of length 1 on which the Brownian motion is bounded by 1. We can use this to estimate the norm of \(\Besselsol\):
\begin{align*}
\norm{\Besselsol}_{\Besselweight}^2
    & = \int_0^\infty \biggl( \int_0^t e^{s + \frac{2}{\sqrt{\beta}} B(s)} \diff{s} \biggr)^2 e^{-2t-\frac{2}{\sqrt{\beta}} B(t)} \diff{t}
    \geq \sum_{k=1}^\infty \int_{\tau_{n_k}+\nicefrac{1}{2}}^{\tau_{n_k}+1} \biggl( \int_{\tau_{n_k}}^{\tau_{n_k}+\nicefrac{1}{2}} e^{s - \nicefrac{2}{\sqrt{\beta}}} \diff{s} \biggr)^2 e^{-2t - \nicefrac{2}{\sqrt{\beta}}} \diff{t} \\
    & \geq e^{\nicefrac{-6}{\sqrt{\beta}}} \sum_{k=1}^\infty \bigl( e^{\tau_{n_k}+\nicefrac{1}{2}} - e^{\tau_{n_k}} \bigr)^2 \frac{e^{-2(\tau_{n_k}+\nicefrac{1}{2})} - e^{-2(\tau_{n_k}+1)}}{2}
    = \frac{1}{2} e^{\nicefrac{-6}{\sqrt{\beta}}} \sum_{k=1}^\infty (e^{\nicefrac{1}{2}} - 1)^2 (e^{-1} - e^{-2}).
\end{align*}
The summand here is a positive constant, so this obviously diverges. Therefore, \(\norm{\Besselsol}_{\Besselweight}^2 = \infty\) a.s.\@ when \(a = 1\), which shows that \(\Besselop\) is limit point at infinity and concludes the proof.
\end{proof}

We record here for future reference the following result, which is just a more explicit version of something we have seen in the proof of the last proposition.

\begin{lemma}
\label{lem.Weylbound}
If \(\abs{a} < 1\) and \(\delta < \frac{1-a}{3} \wedge \frac{1}{2}\), then for every \(\varepsilon > 0\) there is a \(C_\varepsilon > 0\) such that
\[
\bprob[\big]{\forall t \geq 0, \Besselweight(t) \leq C_\varepsilon e^{-(1+a-\delta)t}} \geq 1 - \varepsilon
\qquadtext{and}
\bprob[\big]{\forall t \geq 0, \Besselsol^2(t) \Besselweight(t) \leq C_\varepsilon e^{-(1-a-3\delta)t}} \geq 1 - \varepsilon.
\]
\end{lemma}

\begin{proof}
Recall the good event from~\eqref{eq.Weyl.goodevent}, and take \(T > 0\) such that \(C(1+t^{\nicefrac{3}{4}}) \leq \delta t\) for \(t \geq T\). On that event, the bound \(C(1+t^{\nicefrac{3}{4}})\) on the Brownian motion is itself bounded on the compact \([0, T]\) by a deterministic constant, so by definition of \(\Besselweight\) and \(\Besselsol\) these are also bounded on \([0, T]\) by deterministic constants. 

Now, for \(t \geq T\), on the good event it also holds that \(\Besselweight(t) \leq e^{-(a+1-\delta)t}\), and therefore there is a \(C_\varepsilon > 0\) such that \(\Besselweight(t) \leq C_\varepsilon e^{-(a+1-\delta)t}\) for all \(t \geq 0\). Similarly, we have seen in the proof of Proposition~\ref{prop.Weyl} that on the good event from~\eqref{eq.Weyl.goodevent}, there is a \(C_T \in \mathbb{R}\) such that \(\Besselsol(t) \leq C_T + \frac{1}{a+\delta} e^{(a+\delta)t}\) for \(t \geq T\), meaning that there is a \(C > 0\) such that \(\Besselsol(t) \leq C e^{(a+\delta)t}\) for all \(t \geq 0\). Combining these bounds on \(\Besselweight\) and \(\Besselsol\) on the good event, we see that \(\Besselsol^2(t) \Besselweight(t) \leq C^2C_\varepsilon e^{-(1-a-3\delta)t}\) for all \(t \geq 0\). Enlarging \(C_\varepsilon\) if necessary completes the proof.
\end{proof}

\subsubsection{Relationship between solutions to \texorpdfstring{\(\Besselop f = zf\)}{Gf = zf} and those of \texorpdfstring{\(\sBesselop f = zf\)}{GEf = zf}.}

We point out a last property of the stochastic Bessel operator that will be essential in the sequel. 

\begin{proposition}
\label{prop.unshift}
Let \(\Besselop\) be defined from a Brownian motion \(B\), and set \(\sBesselop \defeq \frac{2}{\sqrt{\shift}} (\Besselop - \shift)\). If \(f\) solves \(\sBesselop f = zf\) for some \(z \in \mathbb{C}\), then \(\tilde{f}(t) \defeq f(t + \log\shift)\) solves \(\unsBesselop \tilde{f} = \bigl(1 + \frac{z}{2\sqrt{\shift}} \bigr) \tilde{f}\) where \(\unsBesselop\) has the law of \(\Besselop\), but is defined from the Brownian motion \(B(\cdot + \log\shift) - B(\log\shift)\).
\end{proposition}

\begin{proof}
By definition, \(\unsBesselop\) is the stochastic Bessel operator defined from the Brownian motion \(B(\cdot + \log\shift) - B(\log\shift)\); more explicitly, \(\unsBesselop f = - \frac{1}{\unsBesselweight} (\unsBesselcoeff f')'\) where
\[
\unsBesselcoeff(t)
    = \exp\Bigl( - at - \frac{2}{\sqrt{\beta}} \bigl( B(t + \log\shift) - B(\log\shift) \Bigr)
    = \shift^a e^{\frac{2}{\sqrt{\beta}} B(\log\shift)} \Besselcoeff(t + \log\shift)
\]
and likewise \(\unsBesselweight(t) = \shift^{a+1} e^{\frac{2}{\sqrt{\beta}} B(\log\shift)} \Besselweight(t + \log\shift)\). Therefore if \(t \geq 0\) then
\[
\unsBesselop \tilde{f}(t)
    = - \frac{1}{\unsBesselweight(t)} \bigl( \unsBesselcoeff \tilde{f}' \bigr)'(t)
    = - \frac{1}{\shift \Besselweight(t + \log\shift)} \bigl( \Besselcoeff f' \bigr)'(t + \log\shift)
    = \frac{1}{\shift} \Besselop f(t + \log\shift).
\]
But if \(f\) solves \(\sBesselop f = zf\), then by definition of \(\sBesselop\) it also solves \(\Besselop f = (\shift + \frac{z\sqrt{\shift}}{2}) f\), and combining with the above we see that \(\unsBesselop \tilde{f}(t) = (1 + \frac{z}{2\sqrt{\shift}}) \tilde{f}(t)\).
\end{proof}

\subsection{A canonical system representation of the shifted Bessel operator}

\subsubsection{The canonical system representation and its boundary conditions}
\label{sec.sBesselCSrepresentation}

Like any Sturm--Liouville operator, the shifted and scaled stochastic Bessel operator \(\sBesselop\) can be turned into a canonical system. To do so, we introduce \(\sBesself\) and \(\sBesselg\), a pair of fundamental solutions to \(\sBesselop f = 0\) with initial conditions \(\sBesself(0) = \sBesselg'(0) = 1\) and \(\sBesself'(0) = \sBesselg(0) = 0\), and we set
\beq{eq.sBesselSLtoCS}
\sBesselSLtoCS \defeq \begin{pmatrix}
    \shift^{\nicefrac{1}{4}} \Besselcoeff \sBesselg' & \shift^{\nicefrac{-1}{4}} \Besselcoeff \sBesself' \\
    \shift^{\nicefrac{1}{4}} \sBesselg & \shift^{\nicefrac{-1}{4}} \sBesself
\end{pmatrix}.
\eeq
It is easy to see by direct computations (see~\cite[Section~2.4]{painchaud_operator_2026} for details in the case of a general Sturm--Liouville operator) that \(f\) solves the eigenvalue equation \(\sBesselop f = zf\) if and only if \(u \defeq \sBesselSLtoCS^{-1} \begin{smallpmatrix} \Besselcoeff f' \\ f \end{smallpmatrix}\) solves the canonical system
\beq{eq.sBesselCS}
Ju' = -z\sBesselmat u
\quadtext{on} (0,\infty)
\qquadtext{with}
\sBesselmat \defeq \frac{\Besselweight\sqrt{\shift}}{2} \begin{pmatrix}
    \sqrt{\shift} \sBesselg^2 & \sBesself\sBesselg \\
    \sBesself\sBesselg & \frac{1}{\sqrt{\shift}} \sBesself^2
\end{pmatrix}.
\eeq
Under this change of variables, the boundary condition \(f(0) = 0\) becomes \(e_0^*Ju(0) = 0\) where \(e_\phi \defeq (\cos\phi, \sin\phi)\), meaning that \(u(0)\) must be parallel to \(e_0 = (1,0)\). The boundary condition at infinity transforms in a similar way, although less explicitly. Indeed, it can be shown by standard theory (see e.g.~\cite[Section~6]{eckhardt_weyl-titchmarsh_2013}) that when \(\abs{a} < 1\), there is a \(\bcphase \in [0,\pi)\) such that for \(f\) in the maximal domain of \(\sBesselop\), \(\lim_{t\to\infty} \Besselcoeff(t) f'(t) = 0\) if and only if
\[
\lim_{t\to\infty} \Bigl( \BesselWronskian{\shift^{\nicefrac{1}{4}}\sBesselg}{f}(t) \cos\bcphase + \BesselWronskian{\shift^{\nicefrac{-1}{4}}\sBesself}{f}(t) \sin\bcphase \Bigr) = 0,
\]
where \(\BesselWronskian{f}{g} \defeq \Besselcoeff (fg' - f'g)\) denotes the Wronskian of \(f\) and \(g\). Note that here, because \(\Besselop\) is limit circle at infinity, the limits of both Wronskians above exist (see e.g.~\cite[Lemma~3.2]{eckhardt_weyl-titchmarsh_2013} for a proof). Under the correspondence \(u = \sBesselSLtoCS^{-1} \begin{smallpmatrix} \Besselcoeff f' \\ f \end{smallpmatrix}\), this condition is exactly the condition that \(\lim_{t\to\infty} e_{\bcphase}^* J u(t) = 0\), which is therefore the boundary condition at infinity for the canonical system.

\subsubsection{Time-changing the system}

The canonical system~\eqref{eq.sBesselCS}, under the boundary conditions described above, is almost set up to converge to the sine canonical system, although an important detail is missing: the two systems are not defined on the same time domains. Thus, we will find a \(\mathscr{C}^1\) bijection \(\timechange\colon (0, 1+\varepsilon_{\shift}) \to (0,\infty)\) and rather work with the time-changed system with coefficient matrix \(\timechange' (\sBesselmat\circ\timechange)\) on \((0, 1+\varepsilon_{\shift})\). A direct computation shows that \(u\colon (0,\infty) \to \mathbb{C}^2\) solves~\eqref{eq.sBesselCS} if and only if \(v \defeq u\circ\timechange\) solves \(Jv' = -z\timechange' (\sBesselmat\circ\timechange) v\) on \((0, 1+\varepsilon_{\shift})\), so the two canonical systems are equivalent. 

When analyzing the convergence of a sequence of canonical systems, if the limit system has limit circle endpoints, the convergence is easier to analyze when the corresponding endpoints of the systems in the sequence are also limit circle, as we will see in Section~\ref{sec.spectralconv} (see also~\cite[Section~2.2.3]{painchaud_canonical_2026} or~\cite[Section~2.2]{painchaud_operator_2026}). This is the reason for introducing the extra parameter \(\varepsilon_{\shift}\): when \(\beta > 2\) and \(a \geq 1\), the sine system is limit circle on the right but the Bessel system is limit point, so taking \(\varepsilon_{\shift} > 0\) allows us to have systems which are all limit circle at both endpoints on \([0,1]\), and we will work out the vague convergence on that restricted domain. In that case, the Bessel system won't have a proper boundary condition at \(1\), but the value at \(1\) of the integrable (on \((0, 1+\varepsilon_{\shift})\)) solution will play the same role, as we will see in detail in Section~\ref{sec.spectralconv}. For other choices of parameters, either the sine system is limit point on the right or both systems are limit circle, so we simply take \(\varepsilon_{\shift} = 0\). 

A candidate for the appropriate time change \(\timechange\) can be guessed from the limit deterministic case \(\beta = \infty\). Indeed, in that case the Brownian motion disappears from the problem, and it is easy to see that a solution to \(\sBesselop f = 0\) has the form
\[
f(t) = e^{\nicefrac{at}{2}} \Bigl( C J_a(2\sqrt{\shift} e^{\nicefrac{-t}{2}}) + C' Y_a(2\sqrt{\shift} e^{\nicefrac{-t}{2}}) \Bigr)
\]
for constants \(C,C' \in \mathbb{R}\), where \(J_a\) and \(Y_a\) are Bessel functions of the first and second kind respectively. This yields explicit expressions for \(\sBesself[\infty]\) and \(\sBesselg[\infty]\). Using asymptotic expansions of Bessel functions (see e.g.~\cite{NIST:DLMF}), when \(\shift\) is large we can approximate these expressions by
\[
\sBesself[\infty](t) \approx e^{\nicefrac{at}{2} + \nicefrac{t}{4}} \cos\bigl( 2\sqrt{\shift} (1 - e^{\nicefrac{-t}{2}}) \bigr)
\qquadtext{and}
\sBesselg[\infty](t) \approx \frac{1}{\sqrt{\shift}} e^{\nicefrac{at}{2} + \nicefrac{t}{4}} \sin\bigl( 2\sqrt{\shift} (1 - e^{\nicefrac{-t}{2}}) \bigr).
\]
This gives
\[
\timechange' (\sBesselmat\circ\timechange) \approx
    \timechange' \frac{\sqrt{\shift} e^{-(a+1)\timechange}}{2} \frac{e^{(a+\nicefrac{1}{2})\timechange}}{\sqrt{\shift}} \begin{pmatrix}
        \sin^2(\theta_{\shift}\circ\timechange) & \sin(\theta_{\shift}\circ\timechange) \cos(\theta_{\shift}\circ\timechange) \\
        \sin(\theta_{\shift}\circ\timechange) \cos(\theta_{\shift}\circ\timechange) & \cos^2(\theta_{\shift}\circ\timechange)
    \end{pmatrix}
\]
where \(\theta_{\shift}(t) \defeq 2\sqrt{\shift} (1 - e^{\nicefrac{-t}{2}})\). As \(\shift \to \infty\), the oscillations of the trigonometric functions should make the matrix converge vaguely to \(\frac{1}{2} I_2\), which is exactly the coefficient matrix of the sine system when \(\beta = \infty\). Hence, \(\timechange\) should be chosen so that the prefactor cancels out when \(\shift \to \infty\). This suggests to take \(\timechange\) as a solution of
\begin{subequations}
\label{eq.timechange}
\beq{eq.timechangeODE}
\timechange' = 2\diln e^{\nicefrac{\timechange}{2}}
\eeq
where we allow a dependence on an extra parameter \(\diln \in \mathbb{R}\) chosen so that \(\diln \to 1\) as \(\shift\to\infty\). It is easy to solve this equation explicitly, and the solution with \(\timechange(0) = 0\) is
\beq{eq.timechangesol}
\timechange(t) = -2\log(1 - \diln t).
\eeq
\end{subequations}
This is a \(\mathscr{C}^1\) bijection \((0, \nicefrac{1}{\diln}) \to (0,\infty)\). As explained above, we want a bijection \((0,1) \to (0,\infty)\) when \(\beta \leq 2\) and when \(\beta > 2\) and \(\abs{a} < 1\), so in that case we simply take \(\diln \defeq 1\). When \(\beta > 2\) and \(a \geq 1\), we want to take \(\diln < 1\) in order to extend slightly the time domain; the specific value we take is \(\diln \defeq 1 - \nicefrac{1}{\sqrt{\shift}}\). 
To motivate this choice, note that by Proposition~\ref{prop.unshift}, the time domain \([0,1]\) for the time-changed system \(\timechange' (\sBesselmat\circ\timechange)\) corresponds, upon inverting the time change and reversing the shift, to the time domain \([-\log\shift, -2\log(1-\diln) - \log\shift]\) for the original operator \(\Besselop\). Taking \(\diln = 1 - \nicefrac{1}{\sqrt{\shift}}\) therefore fixes the right boundary of this time domain to always be \(0\) for the original operator, which removes its dependence on the shift \(\shift\). This will play a role in the analysis of the right boundary conditions in Section~\ref{sec.spectralconv}.

In the sequel, we will denote by \(\lasttime \defeq (1 - \nicefrac{1}{\sqrt{\shift}}) / \diln\) the time that satisfies \(\timechange(\lasttime) = \log\shift\), and thus corresponds to time \(0\) for \(\Besselop\), when reversing the time-change and shift.

\subsection{Polar coordinates and intuition on the convergence}
\label{sec.polarcoords}

The first thing we need in order to prove the vague convergence of the canonical systems is to write the solutions \(\sBesself\) and \(\sBesselg\), which appear in the coefficient matrix of the Bessel system, in terms of suitable polar coordinates. The purpose of this change of variables is to separate the oscillating part of the coefficient matrix, which vanishes in the vague limit, from the part that actually converges to the sine system's coefficient matrix. 

\begin{proposition}
\label{prop.polarcoords}
Define \(\sBesselop\) from a standard Brownian motion \(B\) on a filtered probability space \((\Omega, \mathscr{F}, \{\mathscr{F}_t\}_{t\geq 0}, \mathbb{P})\), and let \(\timechange\) and \(\diln\) be as in~\eqref{eq.timechange}. If \(f\) solves \(\sBesselop f = 0\) and if \(f(0)\) and \(f'(0)\) are independent of \(B\), then for \(t \in [0, \nicefrac{1}{\diln})\),
\[
f\circ\timechange
    = \frac{C_{f} \shift^{\nicefrac{-1}{4}} e^{\nicefrac{\timechange}{4}}}{\sqrt{\Besselcoeff\circ\timechange}} e^{\amp} \cos\phase
\qquadtext{and}
f'\circ\timechange
    = \frac{C_{f} \shift^{\nicefrac{1}{4}} e^{\nicefrac{-\timechange}{4}}}{\sqrt{\Besselcoeff\circ\timechange}} e^{\amp} \sin\phase
\]
where \(C_{f}^2 \defeq \sqrt{\shift} f^2(0) + \frac{1}{\sqrt{\shift}} {f'}^2(0)\) and \(\amp\) and \(\phase\) solve the coupled stochastic differential equations
\begin{align*}
\diff{\amp}(t)
    & = \biggl( \frac{1}{\beta} - \Bigl( a + \frac{1}{2} \Bigr) \cos 2\phase(t) - \frac{1}{\beta} \cos 4\phase(t) \biggr) \frac{\diln}{1 - \diln t} \diff{t} - \sqrt{\frac{2}{\beta}} \cos 2\phase(t) \sqrt{\frac{\diln}{1 - \diln t}} \diff{\sBM}(t), \\
\diff{\phase}(t)
    & = -2\diln\sqrt{\shift} \diff{t} + \biggl( \Bigl( a + \frac{1}{2} \Bigr) \sin 2\phase(t) + \frac{1}{\beta} \sin 4\phase(t) \biggr) \frac{\diln}{1 - \diln t} \diff{t} \\
    &\hspace*{88mm} + \sqrt{\frac{2}{\beta}} \sin 2\phase(t) \sqrt{\frac{\diln}{1 - \diln t}} \diff{\sBM}(t)
\end{align*}
with \(\amp(0) = 0\) and \(\phase(0) = \arctan\bigl( \nicefrac{f'(0)}{\sqrt{\shift} f(0)} \bigr)\), and where \(\sBM(t) \defeq \frac{1}{\sqrt{2\diln}} \int_0^{\timechange(t)} e^{\nicefrac{-s}{4}} \diff{B}(s)\) is a Brownian motion and a martingale with respect to the filtration \(\{\mathscr{F}_{\timechange(t)}\}_{t\in [0,\nicefrac{1}{\diln})}\). 
\end{proposition}

\begin{proof}
Let \(f\) solve \(\sBesselop f = 0\). Then \(f\) also solves \(\Besselop f = \shift f\), and therefore it is also a solution to the SDE~\eqref{eq.BesselSDE} with \(z = \shift\). Hence, by the Dambis--Dubins--Schwarz theorem,
\[
\diff{(f'\circ\timechange)}(t)
    = -\shift e^{-\timechange(t)} f\circ\timechange(t) \timechange'(t) \diff{t} + \Bigl( a + \frac{2}{\beta} \Bigr) f'\circ\timechange(t) \timechange'(t) \diff{t} + \frac{2}{\sqrt{\beta}} f'\circ\timechange(t) \sqrt{\timechange'(t)} \diff{\sBM}(t)
\]
where
\[
\sBM(t)
    \defeq \int_0^{\timechange(t)} \frac{1}{\sqrt{\timechange'\circ\timechange^{-1}(s)}} \diff{B}(s)
    = \frac{1}{\sqrt{2\diln}} \int_0^{\timechange(t)} e^{\nicefrac{-s}{4}} \diff{B}(s)
\]
is a standard Brownian motion and a martingale with respect to \(\{\mathscr{F}_{\timechange(t)}\}_{t\in[0,\nicefrac{1}{\diln})}\).

Set \(y \defeq f\circ\timechange\), so that \(y' = \timechange' (f'\circ\timechange)\). From the above expression for \(\diff{(f'\circ\timechange)}\), we deduce that
\[
\diff{y'}(t)
    = \frac{\timechange''(t)}{\timechange'(t)} y'(t) \diff{t}
    - \shift \timechange'(t)^2 e^{-\timechange(t)} y(t) \diff{t} + \Bigl( a + \frac{2}{\beta} \Bigr) \timechange'(t) y'(t) \diff{t} + \frac{2}{\sqrt{\beta}} \sqrt{\timechange'(t)} y'(t) \diff{\sBM}(t).
\]
By definition of \(\timechange\), \({\timechange'}^2 e^{-\timechange} = 4\diln^2\) and \(\nicefrac{\timechange''}{\timechange'} = \nicefrac{\timechange'}{2}\), so this simplifies to
\[
\diff{y'}(t)
    = - 4\diln^2\shift y(t) \diff{t} + \Bigl( a + \frac{2}{\beta} + \frac{1}{2} \Bigr) \timechange'(t) y'(t) \diff{t} + \frac{2}{\sqrt{\beta}} \sqrt{\timechange'(t)} y'(t) \diff{\sBM}(t).
\]

Now, define two real-valued stochastic processes \(r\) and \(\xi\) such that \(e^{r + i\xi} = Sy + \nicefrac{iy'}{S}\) where \(S\) is a (constant) scaling factor that will be determined later. By Itô's formula, omitting the explicit time dependence to simplify notation,
\begin{align*}
\diff{r} + i \diff{\xi}
    & = \diff{\bigl( \log(Sy + \nicefrac{iy'}{S}) \bigr)}
    = \frac{1}{Sy + \nicefrac{iy'}{S}} \Bigl( S y' \diff{t} + \frac{i}{S} \diff{y'} \Bigr) + \frac{1}{2S^2} \frac{1}{(Sy + \nicefrac{iy'}{S})^2} \diff{\quadvar{y'}} \\
    & = e^{-2r} \Bigl( S^2yy' \diff{t} + \frac{1}{S^2} y' \diff{y'} + iy\diff{y'} - i{y'}^2 \diff{t} \Bigr) + \frac{e^{-4r}}{2S^2} \Bigl( S^2y^2 - \frac{{y'}^2}{S^2} - 2iyy' \Bigr) \diff{\quadvar{y'}} \\
    & = e^{-2r} \biggl( S^2yy' - \frac{4\diln^2\shift}{S^2} yy' + \Bigl( a + \frac{2}{\beta} + \frac{1}{2} \Bigr) \frac{{y'}^2}{S^2} \timechange' \biggr) \diff{t} 
    + \frac{2}{\beta} e^{-4r} \frac{{y'}^2}{S^2} \Bigl( S^2y^2 - \frac{{y'}^2}{S^2} \Bigr) \timechange' \diff{t}
    + \frac{2}{\sqrt{\beta}} e^{-2r} \frac{{y'}^2}{S^2} \sqrt{\timechange'} \diff{\sBM}  \\
    &\qquad + ie^{-2r} \biggl( -4\diln^2\shift y^2 - {y'}^2 + \Bigl( a + \frac{2}{\beta} + \frac{1}{2} \Bigr) yy' \timechange' \biggr) \diff{t} 
    - \frac{4i}{\beta} e^{-4r} \frac{y{y'}^3}{S^2} \timechange' \diff{t}
    + \frac{2i}{\sqrt{\beta}} e^{-2r} yy' \sqrt{\timechange'} \diff{\sBM}.
\end{align*}
This simplifies with \(S \defeq \sqrt{2\diln} \shift^{\nicefrac{1}{4}}\). Indeed, it balances the magnitudes of the first two terms of the real part and of the imaginary part, so that they cancel out in the real part and combine as \(4\diln^2\shift y^2 + {y'}^2 = S^2 (S^2y^2 + \nicefrac{{y'}^2}{S^2}) = 2\diln\sqrt{\shift} e^{2r}\) in the imaginary part. Extracting the SDEs for \(r\) and \(\xi\) from the real and imaginary parts of the above and replacing \(y\) and \(y'\) with their expressions in terms of \(r\) and \(\xi\), we get
\begin{align*}
\diff{r}
    & = \Bigl( a + \frac{2}{\beta} + \frac{1}{2} \Bigr) \sin^2\xi \, \timechange' \diff{t}
    + \frac{2}{\beta} \sin^2\xi (\cos^2\xi - \sin^2\xi) \timechange' \diff{t}
    + \frac{2}{\sqrt{\beta}} \sin^2\xi \sqrt{\timechange'} \diff{\sBM} \\
\shortintertext{and}
\diff{\xi}
    & = - 2\diln\sqrt{\shift} \diff{t} + \Bigl( a + \frac{2}{\beta} + \frac{1}{2} \Bigr) \sin\xi \cos\xi \, \timechange' \diff{t}
    - \frac{4}{\beta} \sin^3\xi \cos\xi \, \timechange' \diff{t}
    + \frac{2}{\sqrt{\beta}} \sin\xi \cos\xi \sqrt{\timechange'} \diff{\sBM}.
\end{align*}
It is easy to see by simplifying further with trigonometric identities that \(\phase \defeq \xi\) follows the announced SDE with \(\phase(0) = \arctan\bigl( \frac{y'(0)}{S^2y(0)} \bigr) = \arctan\bigl( \frac{f'(0)}{\sqrt{\shift} f(0)} \bigr)\). Performing similar simplifications in the other SDE, we get
\[
\diff{r}
    = \biggl( \frac{1}{\beta} + a + \frac{1}{2} - \Bigl( a + \frac{1}{2} \Bigr) \cos 2\xi - \frac{1}{\beta} \cos 4\xi \biggr) \frac{\timechange'}{2} \diff{t}
    + \frac{1}{\sqrt{\beta}} (1 - \cos 2\xi) \sqrt{\timechange'} \diff{\sBM}.
\]
With
\[
\tilde{r}(t)
    \defeq r(0) + \Bigl( \frac{a}{2} + \frac{1}{4} \Bigr) \timechange(t) + \frac{1}{\sqrt{\beta}} \int_0^t \sqrt{\timechange'(s)} \diff{\sBM}(s),
\]
we obtain that \(\amp \defeq r - \tilde{r}\) follows the announced SDE with \(\amp(0) = 0\). 

It remains to check the representations of \(f\circ\timechange\) and \(f'\circ\timechange\) in terms of \(\amp\) and \(\phase\). To do so, note that the last term of \(\tilde{r}\) is exactly \(\frac{1}{\sqrt{\beta}} B\bigl(\timechange(t)\bigr)\) where \(B\) is the original Brownian motion, so in fact \(\tilde{r} = r(0) + \frac{\timechange}{4} - \frac{1}{2} \log\Besselcoeff\circ\timechange\). Moreover,
\[
e^{2r(0)}
    = S^2f^2(0) + \frac{{\timechange'}^2(0) {f'^2}(0)}{S^2}
    = 2\diln \Bigl( \sqrt{\shift} f^2(0) + \frac{{f'}^2(0)}{\sqrt{\shift}} \Bigr)
    \eqdef 2\diln C_{f}^2.
\]
It follows that
\begin{align*}
f\circ\timechange
    & = \frac{1}{S} e^r \cos\xi
    = \frac{1}{\sqrt{2\diln}\shift^{\nicefrac{1}{4}}} \frac{e^{r(0) + \amp + \nicefrac{\timechange}{4}}}{\sqrt{\Besselcoeff\circ\timechange}} \cos\phase
    = \frac{C_{f} \shift^{\nicefrac{-1}{4}} e^{\nicefrac{\timechange}{4}}}{\sqrt{\Besselcoeff\circ\timechange}} e^{\amp} \cos\phase
\shortintertext{and likewise that}
f'\circ\timechange
    & = \frac{S}{\timechange'} e^r \sin\xi
    = \frac{\shift^{\nicefrac{1}{4}}}{\sqrt{2\diln} e^{\nicefrac{\timechange}{2}}} \frac{e^{r(0) + \amp + \nicefrac{\timechange}{4}}}{\sqrt{\Besselcoeff\circ\timechange}} \sin\phase
    = \frac{C_{f} \shift^{\nicefrac{1}{4}} e^{\nicefrac{-\timechange}{4}}}{\sqrt{\Besselcoeff\circ\timechange}} e^{\amp} \sin\phase.
\qedhere
\end{align*}
\end{proof}

The polar coordinates from Proposition~\ref{prop.polarcoords} can be used to represent the fundamental solutions \(\sBesself\) and \(\sBesselg\) that appear in the expression~\eqref{eq.sBesselCS} of the coefficient matrix of the Bessel system in terms of pairs \((\ampf, \phasef)\) and \((\ampg, \phaseg)\). In particular, the constants that appear in these representations are \(C_{\sBesself} = \shift^{\nicefrac{1}{4}}\) and \(C_{\sBesselg} = \shift^{\nicefrac{-1}{4}}\), which motivates our choice of scaling in the definitions~\eqref{eq.sBesselSLtoCS} and~\eqref{eq.sBesselCS} of the canonical system representation. These four polar coordinates are not completely independent, since by standard theory the Wronskian of \(\sBesself\) and \(\sBesselg\) is constant (this is also easy to verify by a straightforward computation). This Wronskian is \(1\) at \(0\), so
\[
1 \equiv \Besselcoeff\circ\timechange (\sBesself\sBesselg' - \sBesself'\sBesselg) \circ\timechange
    = e^{\ampf+\ampg} \bigl( \cos\phasef \sin\phaseg - \sin\phasef \cos\phaseg \bigr),
\]
and therefore
\beq{eq.Wronskianidentity}
1 \equiv e^{\ampf+\ampg} \sin(\phaseg - \phasef).
\eeq

Recall that the coefficient matrix of the time-changed, scaled and shifted Bessel system is given by
\[
\timechange'(t)(\sBesselmat\circ\timechange)(t)
    = \frac{\timechange'(t)\Besselweight\circ\timechange(t)\sqrt{\shift}}{2} \begin{pmatrix}
        \sqrt{\shift}\sBesselg^2 & \sBesself\sBesselg \\
        \sBesself\sBesselg & \frac{1}{\sqrt{\shift}} \sBesself^2
    \end{pmatrix}\bigl( \timechange(t) \bigr).
\]
Switching to the polar coordinates from Proposition~\ref{prop.polarcoords}, this becomes
\[
\timechange' (\sBesselmat\circ\timechange)
    = \frac{\timechange' (\Besselweight\circ\timechange) e^{\nicefrac{\timechange}{2}}}{2\Besselcoeff\circ\timechange} \begin{pmatrix}
        e^{2\ampg} \cos^2\phaseg & e^{\ampf+\ampg} \cos\phasef \cos\phaseg \\
        e^{\ampf+\ampg} \cos\phasef \cos\phaseg & e^{2\ampf} \cos^2\phasef
    \end{pmatrix}.
\]
We chose \(\timechange' = 2\diln e^{\nicefrac{\timechange}{2}}\) and by definition \(e^{\timechange} \Besselweight\circ\timechange = \Besselcoeff\circ\timechange\), so the entire prefactor reduces to \(\diln\). To simplify notation in what follows, we set \(\diffamps \defeq \ampg - \ampf\) and \(\sumamps \defeq \ampg + \ampf\), and likewise with the phases \(\phasef\) and \(\phaseg\). Notice that the identity~\eqref{eq.Wronskianidentity} implies that \(e^{-2\ampg} = e^{-\diffamps} \sin\diffphases\). Using this and the trigonometric identity \(\cos\xi\cos\xi' = \frac{1}{2} \cos(\xi-\xi') + \frac{1}{2} \cos(\xi+\xi')\), we can rewrite
\beq{eq.sBesselmatpolarcoords}
\setlength{\multlinegap}{22mm}
\begin{multlined}
\timechange' (\sBesselmat\circ\timechange)
    = \frac{\diln}{2e^{-\diffamps}\sin\diffphases} \begin{pmatrix}
        1 & e^{-\diffamps} \cos\diffphases \\ e^{-\diffamps} \cos\diffphases & e^{-2\diffamps}
    \end{pmatrix} \\
    + \frac{\diln}{2} \begin{pmatrix}
        e^{2\ampg} \cos 2\phaseg & e^{\sumamps} \cos\sumphases \\ e^{\sumamps} \cos\sumphases & e^{2\ampf} \cos 2\phasef
    \end{pmatrix}.
\end{multlined}
\eeq
For comparison, we can expand the definition~\eqref{eq.defsinemat} of the coefficient matrix of the sine system:
\[
\sinemat\circ\logtime
    = \frac{1}{2\Im\HBM\circ\logtime} \begin{pmatrix}
        1 & -\Re\HBM\circ\logtime \\
        -\Re\HBM\circ\logtime & \abs{\HBM\circ\logtime}^2
    \end{pmatrix}.
\]
Comparing the expressions of these two coefficient matrices provides a clear guess on how the convergence happens. Indeed, it is clear from the SDE in Proposition~\ref{prop.polarcoords} that the phases \(\phasef\) and \(\phaseg\) will grow increasingly fast as \(\shift\) becomes large, which should make the second term of~\eqref{eq.sBesselmatpolarcoords} vanish in the vague limit. Then, in the first term, we recognize in place of the real and imaginary parts of the hyperbolic Brownian motion \(\HBM\) those of the process \(-\exp\bigl(-\diffamps-i\diffphases\bigr)\). Hence, we expect that what makes the convergence of possible is that as \(\shift\) becomes large, this process becomes close to a hyperbolic Brownian motion with variance \(\nicefrac{4}{\beta}\), started at \(i\) in the upper half-plane. Notice that because \(\diffamps(0) = 0\) and \(\diffphases(0) = \nicefrac{\pi}{2}\), it does start at \(i\). 

To see how this happens, we derive an SDE for \(-\exp\bigl( -\diffamps-i\diffphases\bigr)\). From Proposition~\ref{prop.polarcoords}, we see that
\begin{align*}
\diff{\diffamps}(t)
    & = \Bigl( (2a+1) \sin\diffphases(t) \sin\sumphases(t) + \frac{2}{\beta} \sin 2\diffphases(t) \sin 2\sumphases(t) \Bigr) \frac{\diln}{1-\diln t} \diff{t} \\
    &\hspace*{66mm} + \frac{2}{\sqrt{\beta}} \sin\diffphases(t) \sin\sumphases(t) \sqrt{\frac{2\diln}{1-\diln t}} \diff{\sBM}(t)
\shortintertext{and}
\diff{\diffphases}(t)
    & = \Bigl( (2a+1) \sin\diffphases(t) \cos\sumphases(t) + \frac{2}{\beta} \sin 2\diffphases(t) \cos 2\sumphases(t) \Bigr) \frac{\diln}{1-\diln t} \diff{t} \\
    &\hspace*{66mm} + \frac{2}{\sqrt{\beta}} \sin\diffphases(t) \cos\sumphases(t) \sqrt{\frac{2\diln}{1-\diln t}} \diff{\sBM}(t).
\end{align*}
From this, applying Itô's formula yields
\beq{eq.approxHBMSDE}
\setlength{\multlinegap}{33mm}
\begin{multlined}
\diff{\bigl( -e^{-\diffamps-i\diffphases} \bigr)}(t)
    = \frac{2i}{\sqrt{\beta}} e^{-\diffamps(t)} \sin\diffphases(t) e^{-2i\phaseg(t)} \sqrt{\frac{2\diln}{1-\diln t}} \diff{\sBM}(t) \\
    + ie^{-\diffamps(t)} \sin\diffphases(t) \Bigl( (2a+1) e^{-2i\phaseg(t)} + \frac{4}{\beta} e^{-4i\phaseg(t)} \Bigr) \frac{\diln}{1-\diln t} \diff{t}.
\end{multlined}
\eeq
Again, because of the increasingly fast oscillations of \(e^{-2i\phaseg}\) and \(e^{-4i\phaseg}\) when \(\shift\) grows, we expect the second line here not to contribute in the limit. Ignoring the second line,~\eqref{eq.approxHBMSDE} has the form of the SDE for a hyperbolic Brownian motion with variance \(\nicefrac{4}{\beta}\) in the upper half-plane, but driven by the process
\beq{eq.approxCBM}
t \mapsto \int_0^{t} ie^{-2i\phaseg(s)} \sqrt{\frac{2\diln}{1-\diln s}} \diff{\sBM}(s).
\eeq
In the next section, we show that indeed this process converges in distribution to a standard complex Brownian motion. Then, we use this to make rigorous the heuristic idea presented above and prove that \(-\exp\bigl( -\diffamps-i\diffphases \bigr)\) does converge to a hyperbolic Brownian motion. This will allow us to prove the vague convergence of the canonical systems in Section~\ref{sec.vagueconv}.

\section{Coupling the Bessel and sine systems}
\label{sec.coupling}

In this section, we build a coupling between a sequence of real Brownian motions and a single complex Brownian motion. This coupling immediately yields a coupling between a sequence of shifted and scaled stochastic Bessel canonical systems and a stochastic sine canonical system, and from this we make rigorous the heuristic presented at the end of the last section and prove the convergence of the process in~\eqref{eq.approxHBMSDE} to a hyperbolic Brownian motion.

Before moving on, we make an important comment about several of the proofs that are found in the sequel. The transition from the soft edge to the bulk, at the canonical system level, was described in~\cite{painchaud_operator_2026} using a setup similar to what we have presented here so far: first finding a suitable time change, then deriving polar coordinates for fundamental solutions, and then identifying the process that must converge to a hyperbolic Brownian motion for the convergence to happen. It turns out that the SDEs that describe the evolution of the polar coordinates have a similar structure in both cases. Moreover, the processes that converge to the hyperbolic Brownian motions have the same expression in terms of the polar coordinates, and the processes that become the driving complex Brownian motions have essentially the same form. Because of this, the ideas needed to build the coupling and eventually to prove the vague convergence of the canonical systems are similar, and several intermediate results can be proven using the same techniques. In the sequel, in order to avoid repeating some computations that can be adapted from the detailed proofs in~\cite{painchaud_operator_2026}, the proofs of the results that have a counterpart in~\cite{painchaud_operator_2026} are often abridged and focus on the important ideas and on the differences. A more general argument which also implies the vague convergence results proven here can also be found in the author's PhD thesis~\cite[Chapter~4]{painchaud_canonical_2026}.

\subsection{Construction of the coupling}

\begin{lemma}
\label{lem.coupling}
\setshift{E_n}
Let \(\{\shift\}_{n\in\mathbb{N}} \subset (0,\infty)\) satisfy \(\shift\to\infty\). There exists a probability space on which are defined a sequence \(\{\sBM\}_{n\in\mathbb{N}}\) of standard real Brownian motions and a standard complex Brownian motion \(W\) such that if \(\phaseg\) is the solution to the SDE from Proposition~\ref{prop.polarcoords} started from \(\phaseg(0) = \nicefrac{\pi}{2}\) and driven by \(\sBM\), then for any \(\alpha \in (0, \nicefrac{1}{2})\) and \(\delta \in (0, \nicefrac{\alpha}{3})\), there is a \(C > 0\) such that
\setshift{E}%
for any \(\shift \in \{\shift_n\}_{n\in\mathbb{N}}\) large enough,
\beq{eq.coupling}
\bprob[\bigg]{\sup_{t\in\timedom} \abs[\Big]{\int_0^t ie^{-2i\phaseg\mkern-3mu(s)} \mkern-1mu \sqrt{\frac{2\diln}{1-\diln s}} \diff{\sBM}(s) - W\circ\slogtime(t)} \geq \shift^{-(\nicefrac{\alpha}{3}-\delta)}}
    \leq 3\shift^{\nicefrac{4\alpha}{3}} \log^2\shift e^{-C\shift^{\nicefrac{2\delta}{3}}}
\eeq
where \(\timedom \defeq [0, (1 - \shift^{\nicefrac{-1}{2}+\alpha})/\diln]\) and \(\slogtime(t) \defeq -\log(1 - \diln t)\). 
\end{lemma}

\begin{remark}
For each Brownian motion \(\sBM\), one can define two pairs \((\ampf, \phasef)\) and \((\ampg, \phaseg)\) of solutions to the SDEs from Proposition~\ref{prop.polarcoords} with \(\ampf(0) = \ampg(0) = \phasef(0) = 0\) and \(\phaseg(0) = \nicefrac{\pi}{2}\), which can be used to define a pair of fundamental solutions \(\sBesself\) and \(\sBesselg\) through their representations in terms of polar coordinates. This yields a realisation of the canonical system~\eqref{eq.sBesselCS} for \(\sBesselop\). In the same way, one can define a hyperbolic Brownian motion driven by \(W\), which yields a realisation of the sine canonical system. 
\setshift{E_n}%
Therefore, the probability space from the lemma actually supports a sequence of canonical systems for the shifted and scaled Bessel operators \(\sBesselop\) as well as a sine canonical system, and these are coupled through their driving Brownian motions. 
\end{remark}

\begin{proof}
This result is analogous to Lemma~15 of~\cite{painchaud_operator_2026}. We provide only the main steps of the proof here, but the computations that we omit can easily be adapted from the detailed proof of Lemma~15 of~\cite{painchaud_operator_2026}. See also~\cite[Lemma~4.4]{painchaud_canonical_2026} for a complete argument.

The idea of the proof is to build, for a given \(\shift > 0\), a coupling between a discretization of the stochastic integral~\eqref{eq.approxCBM} and a random walk, and to extend this coupling to one between~\eqref{eq.approxCBM} and a complex Brownian motion in such a way that the estimate~\eqref{eq.coupling} is satisfied. It is then easy to combine a sequence of such couplings to obtain the announced probability space. 

Fix \(\shift > 0\) and a standard real Brownian motion \(\sBM\), and let \(\phaseg\) solve the SDE from Proposition~\ref{prop.polarcoords} with \(\phaseg(0) = \nicefrac{\pi}{2}\). With \(\lasttime = (1 - \nicefrac{1}{\sqrt{\shift}}) / \diln\), we discretize the interval \([0, \lasttime]\) by setting \(t_0 \defeq 0\), \(t_N \defeq \lasttime\), and in between
\beq{eq.coupling.partition}
\diln t_j \defeq 1 - (1 + \shift^{-p})^{-j}
\qquadtext{so that}
N \defeq \ceil[\bigg]{\frac{\log\shift}{2\log(1 + \shift^{-p})}}
\eeq
where \(p \defeq \nicefrac{2\alpha}{3}\). Then, we set
\beq{eq.coupling.discretized}
\mathscr{B}^\xi_j \defeq \int_{t_{j-1}}^{t_j} ie^{-2i\phaseg(s)} \sqrt{\frac{2\diln}{1-\diln s}} \diff{\sBM}(s)
\qquadtext{and}
\mathscr{B}^\theta_j \defeq \int_{t_{j-1}}^{t_j} ie^{-2i\theta(s)} \sqrt{\frac{2\diln}{1-\diln s}} \diff{\sBM}(s)
\eeq
where \(\theta(t) \defeq \frac{\pi}{2} -2\diln\sqrt{\shift} t\) is the deterministic part of \(\phaseg\). With
\beq{eq.coupling.sigmas}
\sigma^2 \defeq \int_{t_{j-1}}^{t_j} \frac{\diln}{1-\diln t} \diff{t}
    = \log(1 + \shift^{-p})
\qquadtext{and}
\Sigma_j \defeq \begin{pmatrix}
    \expect (\Re\mathscr{B}^\theta_j)^2 & \expect\Re\mathscr{B}^\theta_j \Im\mathscr{B}^\theta_j \\
    \expect \Re\mathscr{B}^\theta_j \Im\mathscr{B}^\theta_j & \expect (\Im\mathscr{B}^\theta_j)^2
\end{pmatrix},
\eeq
we then set \(W_j \defeq e^{-2i(\phaseg(t_{j-1}) - \theta(t_{j-1}))} \sigma\Sigma_j^{\nicefrac{-1}{2}} \mathscr{B}^\theta_j\), where we identify complex numbers with \(\mathbb{R}^2\) vectors for matrix products. Because \(\mathscr{B}^\theta_j \sim \mathbb{C}\normal{0}{\Sigma_j}\), the orthogonal invariance of the normal distribution implies that \(W_j \sim \mathbb{C}\normal{0}{\sigma^2}\) and that the \(W_j\)'s are independent. 

We start by working on a discretization of the problem: we compare the martingale \(\sum_{j=1}^n \mathscr{B}^\xi_j\) for \(n \in \{1, \hdots, N\}\) with the random walk \(\sum_{j=1}^n W_j\). To do so, we rather focus on controlling the discrete processes
\[
\Delta^{\xi\theta}_n \defeq \sum_{j=1}^n \Bigl( \mathscr{B}^\xi_j - e^{-2i(\phaseg(t_{j-1}) - \theta(t_{j-1}))} \mathscr{B}^\theta_j \Bigr)
\qquadtext{and}
\Delta^{\theta W}_n \defeq \sum_{j=1}^n \Bigl( e^{-2i(\phaseg(t_{j-1}) - \theta(t_{j-1})} \mathscr{B}^\theta_j - W_j \Bigr)
\]
for \(n \in \{1, \hdots, N\}\), which are martingales with respect to the filtration \(\{\mathscr{F}_n\}_{n=0}^N\) generated by \(\{\sBM(t_n)\}_{n=0}^N\). 

To control \(\Delta^{\xi\theta}\), one first verifies that its increments are \(8\sigma^2\)-subgaussian. Then, the Itô isometry and Bernstein's inequality for martingales (see e.g.~\cite[Exercise~IV.3.16]{revuz_continuous_1999} for a precise statement of the latter) can be used to show that for \(j \leq N-1\), \(\bexpect[\big]{\abs{\mathscr{B}^\xi_j - e^{-2i(\phaseg(t_{j-1}) - \theta(t_{j-1}))} \mathscr{B}^\theta_j}^2 \given[\big] \mathscr{F}_{j-1}} \leq C_{\beta,a} \sigma^4\) for a \(C_{\beta,a} > 0\). From these estimates, a corollary (Corollary~28.1 in~\cite{painchaud_operator_2026}) of Freedman's inequality~\cite{freedman_tail_1975} shows that for any \(x > 0\) and any \(\shift\) large enough,
\beq{eq.coupling.Deltaxitheta}
\setlength{\multlinegap}{22mm}
\begin{multlined}
\bprob[\bigg]{\sup_{1\leq n\leq N-1} \abs{\Delta^{\xi\theta}_n} > x}
    \leq 2\shift^{2p} \log^2\shift \exp\biggl( -\frac{1}{2} \min\Bigl\{ \Bigl( \frac{\shift^{\nicefrac{p}{2}} x\log 2}{12\sqrt{2}} \Bigr)^{\nicefrac{2}{3}}, \frac{\shift^px^2\log2}{9C_{\beta,a}\log\shift} \Bigr\} \biggr) \\
    + 4\exp\biggl( - \frac{x\log 2}{3\sqrt{2C_{\beta,a}} \log\shift} \exp\biggl( \frac{1}{4}\Bigl( \frac{\shift^{\nicefrac{p}{2}} x\log 2}{12\sqrt{2}} \Bigr)^{\nicefrac{2}{3}} \biggr) \biggr).
\end{multlined}
\eeq

To control the increments of \(\Delta^{\theta W}\), we write them as
\[
e^{-2i(\phaseg(t_{j-1}) - \theta(t_{j-1}))} \mathscr{B}^\theta_j - W_j
    = e^{-2i(\phaseg(t_{j-1}) - \theta(t_{j-1}))} (I_2 - \sigma\Sigma_j^{\nicefrac{-1}{2}}) \mathscr{B}^\theta_j,
\]
from which we deduce that
\beq{eq.coupling.DeltathetaW.variance}
\bexpect[\big]{\abs{e^{-2i(\phaseg(t_{j-1}) - \theta(t_{j-1}))}\mathscr{B}^\theta_j - W_j}^2 \given[\big] \mathscr{F}_{j-1}}
    \leq \norm{I_2 - \sigma\Sigma_j^{\nicefrac{-1}{2}}}_2^2\, \bexpect[\big]{\abs{\mathscr{B}^\theta_j}^2 \given[\big] \mathscr{F}_{j-1}}.
\eeq
Now, the Itô isometry immediately yields \(\bexpect[\big]{\abs{\mathscr{B}^\theta_j}^2 \given[\big] \mathscr{F}_{j-1}} = 2\sigma^2\), and it can be shown by estimating the eigenvalues of the matrix that for all \(\shift\) large enough,
\[
\norm{I_2 - \sigma\Sigma_j^{\nicefrac{-1}{2}}}_2
    \leq \frac{\delta_j}{2\sqrt{2}\sigma^2}
\qquadtext{where}
\delta_j \defeq \frac{1}{\sqrt{\shift}(1-\diln t_j)}.
\]
This yields a bound of \(\nicefrac{\delta_j^2}{4\sigma^2}\) on~\eqref{eq.coupling.DeltathetaW.variance} for \(\diln t_j \in [0, 1-\shift^{\nicefrac{-1}{2}+\alpha}]\) and \(\shift\) large enough. Because the increments of \(\Delta^{\theta W}\) are gaussian, this estimate on the conditional variances can be taken as an estimate on their subgaussian constants. Summing them up, we obtain an estimate of \(\shift^{-2(\alpha-p)}\) on the subgaussian bracket of \(\Delta^{\theta W}\), and it follows from Azuma's inequality for subgaussian martingales (see e.g.~\cite[Theorem~29]{painchaud_operator_2026}) that for \(x > 0\),
\beq{eq.coupling.DeltathetaW}
\bprob[\bigg]{\sup_{n\in S_{\shift,\alpha}} \abs{\Delta^{\theta W}_{n-1}} > x}
    \leq 4\exp\Bigl( - \frac{\shift^{2(\alpha-p)}x^2}{2} \Bigr)
\eeq
where \(S_{\shift,\alpha} \defeq \bigl\{ n \in \mathbb{N} : 0 < \diln t_{n-1} \leq 1 - \shift^{\nicefrac{-1}{2}+\alpha} \bigr\}\). 

The estimates~\eqref{eq.coupling.Deltaxitheta} and~\eqref{eq.coupling.DeltathetaW} give together an estimate on the difference between the discrete martingale \(\sum_{j=1}^n \mathscr{B}^\xi_j\) for \(n \in \{1, \hdots, N\}\) and the random walk \(\sum_{j=1}^n W_j\). This random walk can be extended to a complex Brownian motion run in logarithmic time by setting \(W\bigl( \slogtime(t_j) \bigr) \defeq W_j\) for all \(j\), and defining \(W\) from independent complex Brownian motions on intervals \([\slogtime(t_{j-1}), \slogtime(t_j)]\). Then, our comparison between the two discrete martingales readily extends to a comparison between the process \(t \mapsto \int_0^t ie^{-2i\phaseg(s)} \sqrt{\frac{\diln}{1-\diln s}} \diff{\sBM}(s)\) and \(W\circ\slogtime\) by estimating the growth of the two continuous processes on intervals \([t_{j-1}, t_j]\). Combining all of the tail bounds together then yields~\eqref{eq.coupling} for fixed \(\shift\). Finally, the proof is completed by combining the couplings for a sequence of shifts \(\{\shift_n\}_{n\in\mathbb{N}}\) into a single probability space, which can be done as in the proof of~\cite[Lemma~15]{painchaud_operator_2026}.
\end{proof}

\subsection{Convergence to the hyperbolic Brownian motion}
\label{sec.coupling.HBMconvergence}

On the probability space from Lemma~\ref{lem.coupling}, we can perform pathwise comparisons between processes driven by the limit complex Brownian motion \(W\) and the processes that approximate them. In particular, we can compare the hyperbolic Brownian motion that appears in the coefficient matrix of the sine system with the process \(-\exp\bigl( -\diffamps - i\diffphases \bigr)\) that approximates it, and which solves the SDE~\eqref{eq.approxHBMSDE}. Our proofs rely in part on estimates on integrals whose integrands oscillate quickly, which causes them to average out as \(\shift\to\infty\). After giving these estimates, we prove the convergence of the hyperbolic Brownian motion, considering separately the imaginary and the real part. 

\subsubsection{Averaging of integrals with oscillatory integrands}

We start with the following result. 

\begin{lemma}
\label{lem.averaging}
Let \(\phase\) be a solution to the SDE from Proposition~\ref{prop.polarcoords} driven by a standard Brownian motion \(B\). Let \(X_{\shift}\) solve
\beq{eq.lem.averaging.SDE}
\diff{X_{\shift}}(t) = \mu_{\shift}(t) X_{\shift}(t) \frac{\diln}{1-\diln t} \diff{t} + \sigma_{\shift}(t) X_{\shift}(t) \sqrt{\frac{\diln}{1-\diln t}} \diff{B}(t)
\eeq
on \([0,\nicefrac{1}{\diln})\) for some complex processes \(\mu_{\shift}\) and \(\sigma_{\shift}\). Fix \(T \in (0, \nicefrac{1}{\diln})\) and a nonzero \(k \in \mathbb{R}\), and suppose that \(\mu_{\shift}\) and \(\sigma_{\shift}\) are bounded on \([0, T]\) by constants \(m_\mu, m_\sigma > 0\) independent of \(\shift\). Then, there are constants \(C, C' > 0\) depending only on \(\beta\), \(a\), \(m_\mu\) and \(m_\sigma\) such that for any \(M, x > 0\),
\begin{multline*}
\pprob[\bigg]{\biggl\{ \sup_{t\in [0,T]} \abs{X_{\shift}} \leq M \biggr\} \cup \biggl\{ \sup_{t\in [0,T]} \abs[\Big]{\int_0^t X_{\shift}(s) e^{ki\phase(s)} \frac{\diln}{1-\diln s} \diff{s}} \geq x + \frac{CM}{\sqrt{\shift}(1-\diln T)} \biggr\}} \\
    \leq 4 \exp\Bigl( - \frac{C' \shift(1-\diln T)^2 x^2}{M^2} \Bigr).
\end{multline*}
\end{lemma}

The proof of this lemma relies mainly on integration by parts combined with concentration inequalities for continuous martingales obtained from estimates on their quadratic variations. As the proof can easily be adapted from that of~\cite[Lemma~16]{painchaud_operator_2026}, we omit it. This result also follows from~\cite[Lemma~4.5]{painchaud_canonical_2026}, where a full proof is given.

We give two immediate corollaries. First, taking \(X_{\shift} \equiv 1\) and \(\diln T \defeq 1 - \shift^{\nicefrac{-1}{2}+\alpha}\) yields the following.

\begin{corollary}
\label{cor.averaging.1}
If \(k \neq 0\) and \(\alpha \in (0, \nicefrac{1}{2})\), there are \(C, C' > 0\) depending only on \(\beta\), \(a\) and \(k\) such that for any \(x > 0\),
\[
\bprob[\bigg]{\sup_{t\in\timedom} \abs[\Big]{\int_0^t e^{ki\phase(s)} \frac{\diln}{1-\diln s} \diff{s}} \geq x + C\shift^{-\alpha}}
    \leq 4 e^{-C'\shift^{2\alpha} x^2}
\]
where \(\timedom \defeq [0, (1 - \shift^{\nicefrac{-1}{2}+\alpha}) / \diln]\). 
\end{corollary}

Likewise, taking \(X_{\shift} \equiv 1\) and \(T \defeq \lasttime = (1 - \nicefrac{1}{\sqrt{\shift}}) / \diln\) yields the following.

\begin{corollary}
\label{cor.averaging.2}
If \(k \neq 0\), then for every \(\varepsilon > 0\) there is a \(C_\varepsilon > 0\) depending only on \(\beta\), \(a\), \(k\) and \(\varepsilon\) such that
\[
\bprob[\bigg]{\sup_{t\in [0, \lasttime]} \abs[\Big]{\int_0^t e^{ki\phase(s)} \frac{\diln}{1 - \diln s} \diff{s}} > C_\varepsilon} < \varepsilon.
\]
\end{corollary}

\subsubsection{The geometric Brownian motion}

We now deduce the convergence of \(e^{-\diffamps}\sin\diffphases\) to the imaginary part of the hyperbolic Brownian motion driven by \(W\). Recall that by the Wronskian identity~\eqref{eq.Wronskianidentity}, \(e^{-\diffamps} \sin\diffphases = e^{-2\ampg}\), and by definition the imaginary part of a hyperbolic Brownian motion is a geometric Brownian motion. We start by comparing the logarithms of the two processes.

\begin{proposition}
\label{prop.GBM}
On the probability space from Lemma~\ref{lem.coupling}, if \(\alpha \in (0,\nicefrac{1}{2})\) and \(\delta \in (0, \nicefrac{\alpha}{3})\), there is a \(C > 0\) depending only on \(\beta\), \(a\), \(\alpha\) and \(\delta\) such that for any \(\shift \in \{\shift_n\}_{n\in\mathbb{N}}\) large enough,
\[
\bprob[\bigg]{\sup_{t\in\timedom} \abs[\Big]{ 2\ampg(t) + \log\GBM(t)} \geq \shift^{\nicefrac{-\alpha}{3}+\delta}}
    \leq 4\shift^{\nicefrac{4\alpha}{3}} \log^2\shift e^{-C\shift^{\nicefrac{2\delta}{3}}}.
\]
where \(\timedom \defeq [0, (1-\shift^{\nicefrac{-1}{2}+\alpha}) / \diln]\) and \(\GBM(t) \defeq \exp\bigl( \frac{2}{\sqrt{\beta}} \Im W\circ\slogtime(t) - \frac{2}{\beta} \slogtime(t) \bigr)\). 
\end{proposition}

\begin{proof}
By definition of \(\ampg\),
\begin{align*}
&\sup_{t\in\timedom} \abs[\Big]{2\ampg(t) + \frac{2}{\sqrt{\beta}} \Im W\circ\slogtime(t) - \frac{2}{\beta} \slogtime(t)} \\
    &\hspace*{18mm} \leq \frac{2}{\sqrt{\beta}} \sup_{t\in\timedom} \abs[\Big]{\int_0^t \cos 2\phaseg(s) \sqrt{\frac{2\diln}{1-\diln s}} \diff{\sBM}(s) - \Im W\circ\slogtime(t)} \\
    &\hspace*{36mm} + \abs{2a+1} \sup_{t\in\timedom} \abs[\Big]{\int_0^t \cos 2\phaseg(s) \frac{\diln}{1-\diln s} \diff{s}}
    + \frac{2}{\beta} \sup_{t\in\timedom} \abs[\Big]{\int_0^t \cos 4\phaseg(s) \frac{\diln}{1-\diln s} \diff{s}}.
\end{align*}
Out of the three suprema on the right-hand side, the first one is directly controlled by taking the imaginary part in~\eqref{eq.coupling} in Lemma~\ref{lem.coupling}, and the two others are controlled by Corollary~\ref{cor.averaging.1} with \(x = \shift^{\nicefrac{-\alpha}{3}+\delta}\). The result then follows by combining the tail bounds, which are dominated by the one from Lemma~\ref{lem.coupling}.
\end{proof}

By exponentiating, we can use the above result to compare \(e^{-2\ampg}\) and \(e^{2\ampg}\) to the geometric Brownian motion \(\GBM\) and its reciprocal. 

\begin{corollary}
\label{cor.GBM}
In the setting of the proposition, for any \(\delta \in (0, \nicefrac{\alpha}{3})\) and any \(\shift\) large enough,
\beq{eq.GBM.-}
\bprob[\bigg]{\sup_{t\in\timedom} \abs[\Big]{e^{-2\ampg(t)} - \GBM(t)} \geq \shift^{\nicefrac{-\alpha}{3}+\delta}}
    \leq \frac{2}{\log\shift},
\eeq
and if \(\alpha > \frac{3}{\beta+6}\) and \(\delta < \frac{1}{3\beta} \bigl( \alpha(\beta+6) - 3 \bigr)\), then for any \(\shift\) large enough,
\beq{eq.GBM.+}
\bprob[\bigg]{\sup_{t\in\timedom} \abs[\Big]{e^{2\ampg(t)} - \frac{1}{\GBM(t)}} \geq \exp\Bigl( - \frac{\alpha(\beta+6)-3-3\beta\delta}{3\beta} \log\shift \Bigr)}
    \leq 2\exp\Bigl( - \frac{\beta\delta^2}{36(1-2\alpha)} \log\shift \Bigr).
\eeq
\end{corollary}

\begin{proof}
This result follows from Proposition~\ref{prop.GBM} in the same way as Corollary~17.1 follows from Proposition~17 in~\cite{painchaud_operator_2026}. We give the main ideas here, but the details that are skipped can easily be adapted from~\cite{painchaud_operator_2026} or found in the proof of~\cite[Corollary~4.6.1]{painchaud_canonical_2026}.

To deduce this result from the proposition, note that
\[
\abs[\Big]{e^{\mp 2\ampg} - \GBM^{\pm 1}}
    \leq \GBM^{\pm 1} \abs[\big]{2\ampg + \log\GBM} e^{\abs{2\ampg + \log\GBM}},
\]
so what remains to control here is only the suprema of \(\GBM\) and \(\GBM^{-1}\). Using the joint density of a Brownian motion and its running maximum, an application of Girsanov's theorem allows to recover the cumulative density function of the supremum of \(\pm \log\GBM\), and from this one can deduce that for \(\shift\) large enough,
\beq{eq.supGBM}
\bprob[\bigg]{\sup_{t\in\timedom} \GBM(t) \geq \log\shift}
    \leq \frac{3}{2\log\shift}
\eeq
and
\beq{eq.supGBM-1}
\bprob[\bigg]{\sup_{t\in\timedom} \GBM^{-1}(t) \geq \exp\biggl( \Bigl( \frac{1-2\alpha}{\beta} + \frac{\delta}{3} \Bigr) \log\shift \biggr)}
    \leq \exp\Bigl( - \frac{\beta\delta^2}{36(1-2\alpha)} \log\shift \Bigr).
\eeq
Combining the estimate~\eqref{eq.supGBM} with the tail bound of the proposition directly gives~\eqref{eq.GBM.-}, as the latter bound is dominated by \(\nicefrac{1}{2\log\shift}\) for \(\shift\) large enough. Then, to deduce~\eqref{eq.GBM.+}, note that on the intersection of the complement of the event in~\eqref{eq.supGBM-1} with the complement of the event in the proposition with \(\delta\) replaced by \(\nicefrac{\delta}{3}\), for \(\shift\) large enough,
\[
\sup_{t\in\timedom} \abs[\Big]{e^{2\ampg(t)} - \GBM^{-1}(t)} \leq \exp\biggl( - \Bigl( \frac{\alpha}{3} - \delta - \frac{1-2\alpha}{\beta} \Bigr) \log\shift \biggr).
\]
If \(\alpha > \frac{3}{\beta+6}\) and \(\delta < \frac{1}{3\beta}\bigl( \alpha(\beta+6) - 3 \bigr)\), then the exponent is negative and combining the tail bounds yields~\eqref{eq.GBM.+}.
\end{proof}

\subsubsection{The real part of the hyperbolic Brownian motion}

We finally turn to the comparison between \(-e^{-\diffamps} \cos\diffphases\) and the real part of the hyperbolic Brownian motion driven by \(W\). 

\begin{proposition}
\label{prop.ReHBM}
On the probability space from Lemma~\ref{lem.coupling}, for any \(\alpha \in (0,\nicefrac{1}{2})\) and \(\delta \in (0,\nicefrac{\alpha}{12})\), there is a \(C > 0\) such that 
for any \(\shift \in \{\shift_n\}_{n\in\mathbb{N}}\) large enough,
\[
\bprob[\bigg]{\sup_{t\in\timedom} \abs[\Big]{-e^{-\diffamps(t)} \cos\diffphases(t) - \frac{2}{\sqrt{\beta}} \int_0^t \GBM(s) \diff{(\Re W\circ\slogtime)}(s)} \geq \shift^{\nicefrac{-\alpha}{12}+\delta}}
    \leq \frac{C}{\log\shift}
\]
where \(\timedom \defeq [0, (1 - \shift^{\nicefrac{-1}{2}+\alpha}) / \diln]\).
\end{proposition}

\begin{proof}
This result is analogous to Proposition~18 of~\cite{painchaud_operator_2026}. We follow the same proof method here, but we omit some details which can easily be adapted from~\cite{painchaud_operator_2026} or found in the proof of~\cite[Proposition~4.7]{painchaud_canonical_2026}.

Taking the real part of~\eqref{eq.approxHBMSDE} and simplifying with the Wronskian identity~\eqref{eq.Wronskianidentity}, we get
\beq{eq.approxReHBMSDE}
\setlength{\multlinegap}{33mm}
\begin{multlined}
\diff{\bigl( -e^{-\diffamps} \cos\diffphases \bigr)}(t)
    = \frac{2}{\sqrt{\beta}} e^{-2\ampg(t)} \sin 2\phaseg(t) \sqrt{\frac{2\diln}{1-\diln t}} \diff{\sBM}(t) \\
    + e^{-2\ampg(t)} \Bigl( (2a + 1) \sin 2\phaseg(t) + \frac{4}{\beta} \sin 4\phaseg(t) \Bigr) \frac{\diln}{1-\diln t} \diff{t}.
\end{multlined}
\eeq
The prove the proposition, we start by showing that the second line of~\eqref{eq.approxReHBMSDE} does not contribute, and then we compare the first line with the integral of \(\GBM\) with respect to \(\Re W\circ\slogtime\) by reducing the problem to discretized versions of the two processes and using the results of Lemma~\ref{lem.coupling} and Proposition~\ref{prop.GBM}. Throughout, we work on the event
\[
\mathscr{G}_{\shift} \defeq \biggl\{ \sup_{t\in\timedom} \GBM(t) \leq \log\shift \biggr\} \cup \biggl\{ \sup_{t\in\timedom} \abs[\big]{e^{-2\ampg(t)} - \GBM(t)} \leq \shift^{\nicefrac{-\alpha}{3}+\delta} \biggr\}
\]
on which we can bound \(e^{-2\ampg}\) and \(\GBM\). By Proposition~\ref{prop.GBM} and its proof, we know that \(\pprob{\mathscr{G}_{\shift}^\complement} \leq \nicefrac{4}{\log\shift}\). 

To control the oscillatory terms that appear in the second line of~\eqref{eq.approxReHBMSDE}, note that a simple application of Itô's formula shows that \(e^{-2\ampg}\) satisfies an SDE of the form~\eqref{eq.lem.averaging.SDE} from Lemma~\ref{lem.averaging}, with \(\mu_{\shift} = (2a+1) \cos 2\phaseg + \frac{4}{\beta} \cos 4\phaseg\) and \(\sigma_{\shift} = \frac{2\sqrt{2}}{\sqrt{\beta}} \cos 2\phaseg\). Applying the lemma with \(M = 2\log\shift\) and \(\diln T = 1 - \shift^{\nicefrac{-1}{2}+\alpha}\), we get that for all \(x > 0\), there are \(C, C' > 0\) such that
\[
\pprob[\bigg]{\mathscr{G}_{\shift} \cap \biggl\{ \sup_{t\in\timedom} \abs[\Big]{\int_0^t e^{-2\ampg(s)} e^{ki\phaseg(s)} \frac{\diln}{1-\diln s} \diff{s}} > x + \frac{2C\log\shift}{\shift^\alpha} \biggr\}}
    \leq 4 \exp\Bigl( - \frac{C' \shift^{2\alpha} x^2}{4\log^2\shift} \Bigr)
\]
for both \(k=2\) and \(k=4\). Taking, say, \(x = \shift^{\nicefrac{-\alpha}{2}}\), the tail bound vanishes faster than \(\nicefrac{1}{\log\shift}\) and therefore we can safely neglect these oscillatory terms. 

Now, to compare the first line of~\eqref{eq.approxReHBMSDE} with the integral of \(\GBM\) with respect to \(\Re W\circ\slogtime\), we discretize the time interval \([0, \lasttime]\) by setting \(t_0 \defeq 0\), \(t_N \defeq \lasttime\), and in between
\[
\diln t_j \defeq 1 - (1 + \shift^{\nicefrac{-\alpha}{6}})^{-j}
\qquadtext{so that}
N \defeq \ceil[\Big]{\frac{\log\shift}{2\log(1+\shift^{\nicefrac{-\alpha}{6}})}}
    \leq \shift^{\nicefrac{\alpha}{6}} \log\shift.
\]
Write \(S_{\shift,\alpha} \defeq \bigl\{ j \in \mathbb{N} : 0 < \diln t_{j-1} \leq 1 - \shift^{\nicefrac{-1}{2}+\alpha} \bigr\}\). Then
\begin{subequations}
\begin{align}
& \sup_{t\in\timedom} \abs[\Big]{\int_0^t e^{-2\ampg(s)} \sin 2\phaseg(s) \sqrt{\frac{2\diln}{1-\diln s}} \diff{\sBM}(s) - \int_0^t \GBM(s) \diff{(\Re W\circ\slogtime)}(s)} \notag \\
    &\hspace*{1mm} \leq \sup_{j\in S_{\shift,\alpha}} \abs[\Big]{\int_0^{t_{j-1}} e^{-2\ampg(s)} \sin 2\phaseg(s) \sqrt{\frac{2\diln}{1-\diln s}} \diff{\sBM}(s) - \int_0^{t_{j-1}} \GBM(s) \diff{(\Re W\circ\slogtime)}(s)}
    \label{eq.ReHBM.discretization.a} \\
    &\hspace*{1mm} + \sup_{j\in S_{\shift,\alpha}} \sup_{t\in [t_{j-1}, t_j]} \biggl( \abs[\Big]{\int_{t_{j-1}}^t e^{-2\ampg(s)} \sin 2\phaseg(s) \sqrt{\frac{2\diln}{1-\diln s}} \diff{\sBM}(s)} + \abs[\Big]{\int_{t_{j-1}}^t \GBM(s) \diff{(\Re W\circ\slogtime)}(s)} \biggr). 
    \label{eq.ReHBM.discretization.b}
\end{align}
\end{subequations}
On \(\mathscr{G}_{\shift}\), if \(Y(t) \defeq \int_{t_{j-1}}^t \diff{X}(s)\) denotes any of the two stochastic integrals on~\eqref{eq.ReHBM.discretization.b}, then \(\quadvar{Y}(t) \leq 8\shift^{\nicefrac{-\alpha}{6}} \log^2\shift\) for \(t \in [t_{j-1},t_j]\). Hence, Bernstein's inequality for continuous martingales shows that
\[
\pprob[\bigg]{\mathscr{G}_{\shift} \cap \biggl\{ \sup_{t\in [t_{j-1},t_j]} \abs[\Big]{\int_{t_{j-1}}^t \diff{X}(s)} > x \biggr\}}
    \leq 2 \exp\Bigl( - \frac{\shift^{\nicefrac{\alpha}{6}} x^2}{16\log^2\shift} \Bigr).
\]
With \(x = \shift^{\nicefrac{-\alpha}{12}+\delta}\) for \(\delta > 0\), the tail bound is an exponential decay. Thus, summing up the bounds for all \(j \in S_{\shift,\alpha}\) (of which there are less than \(N \leq \shift^{\nicefrac{\alpha}{6}} \log\shift\)), we see that the whole of~\eqref{eq.ReHBM.discretization.b} exceeds \(\shift^{\nicefrac{-\alpha}{12}+\delta}\) with probability exponentially decreasing in a power of \(\shift\), which is certainly dominated by \(\nicefrac{1}{\log\shift}\) for \(\shift\) large enough. 

It only remains to compare the two discrete martingales on~\eqref{eq.ReHBM.discretization.a}. To do so, we write \(R(X)\) for the process \(X\) but reset on each increment, i.e., \(R(X)(t) \defeq X(t) - X(t_{j-1})\) for \(t \in [t_{j-1}, t_j)\), and we split the comparison as
\begin{subequations}
\label{eq.ReHBM.split}
\begin{align}
&\int_0^{t_{j-1}} e^{-2\ampg(s)} \sin 2\phaseg(s) \sqrt{\frac{2\diln}{1-\diln s}} \diff{\sBM}(s) - \int_0^{t_{j-1}} \GBM(s) \diff{(\Re W\circ\slogtime)}(s) \notag \\
    &\hspace*{11mm} = \int_0^{t_{j-1}} R\bigl( e^{-2\ampg} \bigr)(s) \sin 2\phaseg(s) \sqrt{\frac{2\diln}{1-\diln s}} \diff{\sBM}(s) \label{eq.ReHBM.split.Rexp} \\
    &\hspace*{22mm} + \sum_{k=1}^{j-1} \Bigl( e^{-2\ampg(t_{k-1})} - \GBM(t_{k-1}) \Bigr) \int_{t_{k-1}}^{t_k} \sin 2\phaseg(s) \sqrt{\frac{2\diln}{1-\diln s}} \diff{\sBM}(s) \label{eq.ReHBM.split.diffGBM} \\
    &\hspace*{22mm} + \sum_{k=1}^{j-1} \GBM(t_{k-1}) \biggl( \int_{t_{k-1}}^{t_k} \sin 2\phaseg(s) \sqrt{\frac{2\diln}{1-\diln s}} \diff{\sBM}(s) - \int_{t_{k-1}}^{t_k} \diff{(\Re W\circ\slogtime)}(s) \biggr) \label{eq.ReHBM.split.diffCBM} \\
    &\hspace*{22mm} - \int_0^{t_{j-1}} R(\GBM)(s) \diff{(\Re W\circ\slogtime)}(s) \label{eq.ReHBM.split.RG}.
\end{align}
\end{subequations}
We control each of these four terms independently, our goal being to show that their suprema over \(j \in S_{\shift,\alpha}\) are bounded by \(\shift^{\nicefrac{-\alpha}{12}+\delta}\) with probability at least \(1 - \nicefrac{1}{\log\shift}\) for \(\shift\) large enough.

To control~\eqref{eq.ReHBM.split.diffCBM}, note that on \(\mathscr{G}_{\shift}\),
\(
\frac{\shift^{\nicefrac{-\alpha}{12}+\delta}}{N\GBM(t_{k-1})}
    \geq \frac{\shift^{\nicefrac{-\alpha}{4}+\delta}}{\log^2\shift}
    \geq \shift^{\nicefrac{-\alpha}{4}}
\)
for \(\shift\) large enough, so Lemma~\ref{lem.coupling} implies that
\[
\pprob[\bigg]{\mathscr{G}_{\shift} \cap \biggl\{ \abs[\Big]{\int_{t_{k-1}}^{t_k} \sin 2\phaseg(s) \sqrt{\frac{2\diln}{1-\diln s}} \diff{\sBM}(s) - \int_{t_{k-1}}^{t_k} \diff{(\Re W\circ\slogtime)}(s)} \geq \frac{\shift^{\nicefrac{-\alpha}{12}+\delta}}{N\GBM(t_{k-1})} \biggr\}}
    \leq 3\shift^{\nicefrac{4\alpha}{3}} \log^2\shift e^{-C\shift^{\nicefrac{\alpha}{18}}}
\]
for \(k \in \{1, \hdots, j-1\}\). Summing over \(k\) and taking the supremum over \(j \in S_{\shift,\alpha}\), the bound keeps an exponential decay.

Then, on \(\mathscr{G}_{\shift}\), each summand of~\eqref{eq.ReHBM.split.diffGBM} has quadratic variation bounded by \(2\shift^{\nicefrac{-2\alpha}{3}+2\delta - \nicefrac{\alpha}{6}}\), so Bernstein's inequality shows that for each \(k\) and \(x > 0\),
\[
\pprob[\bigg]{\mathscr{G}_{\shift} \cap \biggl\{ \abs[\Big]{e^{-2\ampg(t_{k-1})} - \GBM(t_{k-1})} \abs[\bigg]{\int_{t_{k-1}}^{t_k} \sin 2\phaseg(s) \sqrt{\frac{2\diln}{1-\diln s}} \diff{\sBM}(s)} > \frac{x}{N} \biggr\}}
    \leq 2 \exp\Bigl( - \frac{\shift^{\nicefrac{\alpha}{2}-2\delta}x^2}{4\log^2\shift} \Bigr).
\]
With \(x = \shift^{\nicefrac{-\alpha}{12}+\delta}\), this tail bound remains expontential when the increments are summed over \(k\) and when the supremum over \(j \in S_{\shift,\alpha}\) is taken.

To control~\eqref{eq.ReHBM.split.RG}, note that it has quadratic variation
\beq{eq.ReHBM.RGbracket}
\sum_{k=1}^{j-1} \int_{t_{k-1}}^{t_k} \bigl( \GBM(t) - \GBM(t_{k-1}) \bigr)^2 \frac{\diln}{1-\diln t} \diff{t}
    = \frac{4}{\beta} \sum_{k=1}^{j-1} \int_{t_{k-1}}^{t_k} \biggl( \int_{t_{k-1}}^t \GBM(s) \diff{(\Im W\circ\slogtime)}(s) \biggr)^2 \frac{\diln}{1-\diln t} \diff{t}.
\eeq
On \(\mathscr{G}_{\shift}\), the quadratic variation of the remaining stochastic integral (for \(t \in [t_{k-1}, t_k]\)) is bounded by \(\shift^{\nicefrac{-\alpha}{6}} \log^2\shift\), so Bernstein's inequality implies that for all \(y > 0\),
\beq{eq.ReHBM.goodeventQV}
\pprob[\bigg]{\mathscr{G}_{\shift} \cap \biggl\{ \max_{1\leq k\leq j-1} \sup_{t\in [t_{k-1},t_k]} \abs[\Big]{\int_{t_{k-1}}^t \GBM(s) \diff{(\Im W\circ\slogtime)}(s)} > y \biggr\}}
    \leq 2(j-1) \exp\Bigl( - \frac{\shift^{\nicefrac{\alpha}{6}} y^2}{2\log^2\shift} \Bigr).
\eeq
On the complementary event, the bound~\eqref{eq.ReHBM.RGbracket} on the quadratic variation of~\eqref{eq.ReHBM.split.RG} is bounded by \(\frac{2y^2}{\beta} \log\shift\), so another application of Bernstein's inequality shows that for all \(x > 0\),
\[
\pprob[\bigg]{\mathscr{G}_{\shift} \cap \biggl\{ \abs[\Big]{\int_0^{t_{j-1}} R(\GBM)(s) \diff{(\Re W\circ\slogtime)}(s)} > x \biggr\}}
    \leq 2(j-1) \exp\Bigl( - \frac{\shift^{\nicefrac{\alpha}{6}} y^2}{2\log^2\shift} \Bigr)
    + 2 \exp\Bigl( - \frac{\beta x^2}{4y^2 \log\shift} \Bigr).
\]
The exponents are matched if \(2\shift^{\nicefrac{\alpha}{6}} y^4 = \beta x^2 \log\shift\), with which the exponents become \(- \shift^{\nicefrac{\alpha}{12}}x \log^{\nicefrac{1}{2}}\shift\). Thus, with \(x = \shift^{\nicefrac{-\alpha}{12}+\delta}\), the tail bound remains exponential when the supremum over \(j \in S_{\shift,\alpha}\) is taken. 

Finally, \eqref{eq.ReHBM.split.Rexp} can be controlled using the same method as this last case. Indeed, on \(\mathscr{G}_{\shift}\), its quadratic variation is bounded by
\[
\sum_{k=1}^{j-1} \int_{t_{k-1}}^{t_k} \Bigl( 2\shift^{\nicefrac{-2\alpha}{3}+2\delta} + \frac{1}{2} \bigl( \GBM(t) - \GBM(t_{k-1}) \bigr)^2 \Bigr) \frac{2\diln}{1-\diln t} \diff{t}.
\]
On the good event in~\eqref{eq.ReHBM.goodeventQV}, this is bounded by \(2\log\shift (\shift^{\nicefrac{-2\alpha}{3}+2\delta} + \nicefrac{y^2}{\beta})\). From this, like in the last case, another application of Bernstein's inequality yields an exponential tail bound that is strong enough to guarantee that the supremum of~\eqref{eq.ReHBM.split.Rexp} over \(j \in S_{\shift,\alpha}\) exceeds \(\shift^{\nicefrac{-\alpha}{12}+\delta}\) with probability exponentially decaying in a power of \(\shift\). 
\end{proof}

\section{Vague convergence of the canonical systems and convergence of solutions}
\label{sec.vagueconv}

Our goal in this section is to prove the vague convergence of the canonical system for \(\sBesselop\) to the stochastic sine canonical system, from which we also deduce the convergence of their transfer matrices. In addition to results from Section~\ref{sec.coupling}, the proofs rely on bounds on the entries of the coefficient matrix, which we derive in the first subsection.

\subsection{Control on the entries of the coefficient matrix}

We start with the following.

\begin{proposition}
\label{prop.tracebound}
Let \(\amp\) solve the SDE from Proposition~\ref{prop.polarcoords} with \(\amp(0) = 0\). For any \(\varepsilon > 0\), there are constants \(C, C' > 0\) depending only on \(\beta\), \(a\) and \(\varepsilon\) such that
\[
\bprob[\bigg]{\forall t \in [0, \lasttime], \abs[\Big]{\amp(t) + \frac{1}{\beta} \log(1 - \diln t)} \leq C + C'\bigl( -\log(1-\diln t) \bigr)^{\nicefrac{3}{4}}}
    \geq 1 - \varepsilon.
\]
\end{proposition}

\begin{remark}
By exponentiating, this allows to control the entries of the coefficient matrix on a good event, uniformly in \(\shift\). Indeed, if \(k, l \in \{0,1,2\}\) sum to \(2\) and \(\delta > 0\), then on the good event from the proposition there are \(C, C', C'' > 0\) such that
\[
e^{k\ampg(t) + l\ampf(t)}
    \leq \frac{e^C}{(1 - \diln t)^{\nicefrac{2}{\beta}}} \exp\Bigl( C'\bigl( -\log(1-\diln t) \bigr)^{\nicefrac{3}{4}} \Bigr)
    \leq \frac{C'' e^C}{(1 - \diln t)^{\nicefrac{2}{\beta} + \delta}}
    \leq \frac{C'' e^C}{(1 - t)^{\nicefrac{2}{\beta} + \delta}}.
\]
In particular, if \(\beta > 2\), then \(\delta\) can be chosen small enough so that \(\nicefrac{2}{\beta} + \delta < 1\), and then the last bound is integrable and does not depend on \(\shift\). 
\end{remark}

\begin{proof}
This result is analogous to Proposition~20 of~\cite{painchaud_operator_2026} in the case of the soft edge to bulk transition, and the proof uses the same ideas. 

From the SDE in Proposition~\ref{prop.polarcoords}, we know that \(\amp\) satisfies
\begin{multline*}
\amp(t) + \frac{1}{\beta} \log(1 - \diln t)
    = - \Bigl( a + \frac{1}{2} \Bigr) \int_0^t \cos 2\phase(s) \frac{\diln}{1 - \diln s} \diff{s}
    - \frac{1}{\beta} \int_0^t \cos 4\phase(s) \frac{\diln}{1 - \diln s} \diff{s} \\
    - \sqrt{\frac{2}{\beta}} \int_0^t \cos 2\phase(s) \sqrt{\frac{\diln}{1-\diln s}} \diff{\sBM}(s).
\end{multline*}
By Corollary~\ref{cor.averaging.2}, each of the first two terms of the right-hand side is bounded on \([0, \lasttime]\) with probability at least \(1 - \nicefrac{\varepsilon}{3}\). Then, the third term is a continuous martingale \(M\) with quadratic variation \(\quadvar{M}(t) \leq - \frac{2}{\beta} \log(1 - \diln t)\). As such, it can be shown from the Dambis--Dubins--Schwarz theorem and the law of the iterated logarithm (see e.g.~\cite[Proposition~19]{painchaud_operator_2026}) that there is a \(C > 0\) such that it is dominated by \(C + C\bigl( - \frac{2}{\beta} \log(1 - \diln t) \bigr)^{\nicefrac{3}{4}}\) with probability at least \(1 - \nicefrac{\varepsilon}{3}\). Combining the bounds then yields the result.
\end{proof}

Proposition~\ref{prop.tracebound} gives a good control on the entries of the coefficient matrix on \([0, \lasttime]\). When \(\beta > 2\) and \(\abs{a} < 1\), however, this will not suffice and we will need to extend this to \((\lasttime, 1)\). So, we strengthen the last result to the following.

\begin{lemma}
\label{lem.tracebound}
Suppose that \(\abs{a} < 1\), and let \(\mathcal{I} \defeq [0,1)\) if \(\beta \leq 2\) or \(\mathcal{I} \defeq [0,1]\) if \(\beta > 2\). For every \(\varepsilon > 0\), there is an \(F_{\varepsilon} \in L^1\loc(\mathcal{I})\) such that \(\bprob{\timechange' (\tr\sBesselmat\circ\timechange) \leq F_{\varepsilon}} \geq 1 - \varepsilon\).
\end{lemma}

\begin{proof}
Recall from the remark following Proposition~\ref{prop.tracebound} that for each \(\varepsilon, \delta > 0\), there is a \(C > 0\) such that with probability at least \(1 - \nicefrac{\varepsilon}{2}\), for all \(t \in [0, \lasttime]\), 
\[
\timechange'(t) \tr\sBesselmat\circ\timechange(t)
    \leq e^{2\ampg(t)} + e^{2\ampf(t)}
    \leq \frac{2C}{(1 - t)^{\nicefrac{2}{\beta}+\delta}}.
\]
When \(\beta > 2\), we can take \(\delta\) small enough so that \(\nicefrac{2}{\beta} + \delta < 1\), and then the map \(t \mapsto 2C(1-t)^{\nicefrac{-2}{\beta}-\delta}\) is indeed in \(L^1[0,1]\). When \(\beta \leq 2\), this map is in \(L^1\loc[0,1)\) for any value of \(\delta\), so in both cases this bound suffices on \([0,\lasttime]\). To complete the proof, we show that with probability at least \(1 - \nicefrac{\varepsilon}{2}\), there is an \(\tilde{F}_{\varepsilon} \in L^1\loc(\mathcal{I})\) such that \(\timechange'(t) \tr\sBesselmat\circ\timechange(t) \leq \tilde{F}_{\varepsilon}(t)\) for all \(t \in (\lasttime, 1)\); then \(F_\varepsilon(t) \defeq \tilde{F}_\varepsilon(t) + 2C(1-t)^{\nicefrac{-2}{\beta}-1}\) satisfies the desired condition.

To obtain a bound on \((\lasttime, 1)\), we can work with \(\unsBesself(t) \defeq \sBesself(t + \log\shift)\) and \(\unsBesselg(t) \defeq \sBesselg(t + \log\shift)\), which solve \(\unsBesselop \tilde{h} = \tilde{h}\) by Proposition~\ref{prop.unshift}, where \(\unsBesselop\) is defined from the Brownian motion \(\unsBM(t) \defeq \sBM(\cdot + \log\shift) - \sBM(\log\shift)\), \(\sBM\) being the Brownian motion from which \(\sBesselop\) is defined. Recall from Section~\ref{sec.solprops} that \(1\) and \(\Besselsol(t) = \int_0^t \frac{1}{\unsBesselcoeff(s)} \diff{s}\) are a pair of fundamental solutions to \(\unsBesselop \tilde{h} = 0\). By standard theory of Sturm--Liouville operators (see e.g.~\cite[Lemma~2.4]{eckhardt_weyl-titchmarsh_2013}), it follows that the functions \(\varphi\) and \(\psi\) that solve
\begin{align*}
\varphi(t) & = 1 + \int_0^t \bigl( \Besselsol(s) - \Besselsol(t) \bigr) \varphi(s) \unsBesselweight(s) \diff{s}
\shortintertext{and}
\psi(t) & = \Besselsol(t) + \int_0^t \bigl( \Besselsol(s) - \Besselsol(t) \bigr) \psi(s) \unsBesselweight(s) \diff{s}
\end{align*}
are a pair of fundamental solutions for \(\unsBesselop \tilde{h} = \tilde{h}\). Because \(\varphi(0) = \psi'(0) = 1\) and \(\varphi'(0) = \psi(0) = 0\), it follows that for \(t \geq 0\),
\[
\unsBesself(t)
    = \unsBesself(0) \varphi(t) + \unsBesself'(0) \psi(t)
    = \sBesself(\log\shift) \varphi(t) + \sBesself'(\log\shift) \psi(t)
\]
and likewise \(\unsBesselg(t) = \sBesselg(\log\shift) \varphi(t) + \sBesselg'(\log\shift) \psi(t)\). Because \(\unsBesselop\) is limit circle at infinity by Proposition~\ref{prop.Weyl}, we know that \(1, \varphi, \psi \in L^2\bigl( [0,\infty), \unsBesselweight(t) \diff{t} \bigr)\) a.s. Hence, the Cauchy--Schwarz inequality implies that
\[
\int_0^t \Besselsol(s) \varphi(s) \unsBesselweight(s) \diff{s}
    \leq \norm{\Besselsol}_{\unsBesselweight} \norm{\varphi}_{\unsBesselweight}
\qquadtext{and that}
\int_0^t \varphi(s) \unsBesselweight(s) \diff{s}
    \leq \norm{1}_{\unsBesselweight} \norm{\varphi}_{\unsBesselweight},
\]
and likewise for \(\psi\) instead of \(\varphi\). These norms are all well-defined random variables, and it follows that there is a \(\tilde{C} > 0\) such that with probability at least \(1 - \nicefrac{\varepsilon}{6}\), for all \(t \geq 0\),
\[
\abs{\unsBesself(t)} \leq \tilde{C} \bigl( \abs{\sBesself(\log\shift)} + \abs{\sBesself'(\log\shift)} \bigr) \bigl( 1 + \Besselsol(t) \bigr),
\]
and likewise for \(\unsBesselg\) instead of \(\unsBesself\). So with probability at least \(1 - \nicefrac{\varepsilon}{6}\), for \(t \in (\lasttime, 1)\),
\begin{align*}
\tr\sBesselmat\circ\timechange(t)
    & = \frac{\Besselweight\bigl( \timechange(t) \bigr)\sqrt{\shift}}{2} \Bigl( \sqrt{\shift} \unsBesselg^2\bigl( \timechange(t) - \log\shift \bigr) + \frac{1}{\sqrt{\shift}} \unsBesself^2\bigl( \timechange(t) - \log\shift \bigr) \Bigr) \\
    & \leq \begin{aligned}[t]
        2\tilde{C}^2\sqrt{\shift} \Besselweight\bigl( \timechange(t) \bigr) \Bigl( & \sqrt{\shift}\abs[\big]{\sBesselg(\log\shift)}^2 + \sqrt{\shift}\abs[\big]{\sBesselg'(\log\shift)}^2 \\
        & + \frac{1}{\sqrt{\shift}} \abs[\big]{\sBesself(\log\shift)}^2 + \frac{1}{\sqrt{\shift}} \abs[\big]{\sBesself'(\log\shift)}^2 \Bigr) \Bigl( 1 + \Besselsol^2\bigl( \timechange(t) - \log\shift \bigr) \Bigr).
    \end{aligned}
\end{align*}
Now, the representations in polar coordinates of \(\sBesself\) and \(\sBesselg\) give
\begin{align*}
\sqrt{\shift} \abs[\big]{\sBesselg(\log\shift)}^2 \vee \sqrt{\shift} \abs[\big]{\sBesselg'(\log\shift)}^2
    & \leq \frac{e^{2\ampg(\lasttime)}}{\Besselcoeff(\log\shift)}
\shortintertext{and}
\frac{1}{\sqrt{\shift}} \abs[\big]{\sBesself(\log\shift)}^2 \vee \frac{1}{\sqrt{\shift}} \abs[\big]{\sBesself'(\log\shift)}^2
    & \leq \frac{e^{2\ampf(\lasttime)}}{\Besselcoeff(\log\shift)},
\end{align*}
and the bounds on \(e^{2\ampf}\) and \(e^{2\ampg}\) that hold on \([0, \lasttime]\) by Proposition~\ref{prop.tracebound} imply that for any \(\delta > 0\), there is a \(\tilde{C}' > 0\) such that \(e^{2\ampf(\lasttime)} \vee e^{2\ampg(\lasttime)} \leq \tilde{C}'\shift^{\nicefrac{1}{\beta}+\nicefrac{\delta}{2}}\) with probability at least \(1 - \nicefrac{\varepsilon}{6}\). Moreover, notice that by definition of \(\Besselweight\) and \(\Besselcoeff\),
\[
\frac{\Besselweight\bigl( \timechange(t) \bigr)}{\Besselcoeff(\log\shift)}
    = \exp\Bigl( -(a+1)\timechange(t) + a\log\shift - \frac{2}{\sqrt{\beta}} \bigl( \sBM\bigl( \timechange(t) \bigr) - \sBM(\log\shift) \bigr) \Bigr)
    = \shift^{-1} \unsBesselweight\bigl( \timechange(t) - \log\shift \bigr).
\]
This reduces the bound on the trace for \(t \in (\lasttime, 1)\) to
\[
\tr\sBesselmat\circ\timechange(t)
    \leq 8\tilde{C}^2\tilde{C}' \shift^{\nicefrac{-1}{2}+\nicefrac{1}{\beta}+\nicefrac{\delta}{2}} \unsBesselweight\bigl( \timechange(t) - \log\shift\bigr) \Bigl( 1 + \Besselsol^2\bigl( \timechange(t) - \log\shift \bigr) \Bigr),
\]
and this holds with probability at least \(1 - \nicefrac{\varepsilon}{3}\). By Lemma~\ref{lem.Weylbound}, if \(\delta\) is small enough, then there is a \(C_\varepsilon > 0\) such that with probability at least \(1 - \nicefrac{\varepsilon}{6}\), both \(\unsBesselweight(t) \leq C_\varepsilon e^{-(1+a-\delta)t}\) and \(\Besselsol^2(t) \unsBesselweight(t) \leq C_\varepsilon e^{-(1-a-\delta)t}\) for all \(t \geq 0\). Since \(\timechange(t) = -2\log(1 - t)\), we deduce that on this event, for any \(t \geq \lasttime\),
\[
\unsBesselweight\bigl( \timechange(t) - \log\shift \bigr) \Bigl( 1 + \Besselsol^2\bigl( \timechange(t) - \log\shift \bigr) \Bigr)
    \leq 2C_\varepsilon\shift^{1-\abs{a}-\delta} (1 - t)^{2(1-\abs{a}-\delta)}.
\]
Finally, as \(\timechange'(t) = 2(1-t)^{-1}\), this shows that with probability at least \(1 - \nicefrac{\varepsilon}{2}\), for all \(t \in (\lasttime,1)\),
\beq{eq.tracebound.temp}
\timechange'(t) \tr\sBesselmat\circ\timechange(t)
    \leq C \shift^{\nicefrac{1}{2}+\nicefrac{1}{\beta}-\abs{a}-\nicefrac{\delta}{2}} (1-t)^{1-2\abs{a}-2\delta}
\eeq
for \(C \defeq 32\tilde{C}^2\tilde{C}'C_\varepsilon\). Notice that for \(t \geq \lasttime\), \(1 - t \leq \nicefrac{1}{\sqrt{\shift}}\) so \(\shift \leq (1-t)^{-2}\). Hence, if \(\nicefrac{1}{2} + \nicefrac{1}{\beta} - \abs{a} > 0\), then for \(\delta\) small enough the above bound is bounded by \(C (1-t)^{\nicefrac{-2}{\beta}-\delta}\). On the other hand, if \(\nicefrac{1}{2} + \nicefrac{1}{\beta} - \abs{a} \leq 0\), then the power of \(\shift\) in the bound~\eqref{eq.tracebound.temp} is negative for \(\delta\) small enough, so~\eqref{eq.tracebound.temp} is bounded by \(C (1-t)^{-(2\abs{a}-1+2\delta)}\), which is integrable on \([0,1)\) since \(0 < 2\abs{a}-1 < 1\) by the assumptions on \(a\). This argument gives the bound we need: it is always \(L^1\loc[0,1)\), and we can make it integrable when \(\beta > 2\) and \(\nicefrac{1}{2}+\nicefrac{1}{\beta}-\abs{a} > 0\) by taking \(\delta\) small enough so that \(\nicefrac{2}{\beta}+\delta < 1\).
\end{proof}

\subsection{Vague convergence of coefficient matrices}

We finally turn to the vague convergence of the coefficient matrices and complete the proof of Theorem~\ref{thm.vagueconv}. In fact, we prove the following stronger result.

\begin{theorem}
\label{thm.vagueconvP}
\setshift{E_n}
Let \(\{\sBM\}_{n\in\mathbb{N}}\) and \(W\) be the Brownian motions from Lemma~\ref{lem.coupling}, and let \(\sBesselmat\) and \(\sinemat\) be the coefficient matrices of the shifted Bessel system and of the sine system built from these Brownian motions as described in the remark following Lemma~\ref{lem.coupling}. Let \(\mathcal{I} \defeq [0,1)\) if \(\beta \leq 2\) and \(\mathcal{I} \defeq [0,1]\) if \(\beta > 2\). Then for any \(\testfunc \in \mathscr{C}_c(\mathcal{I},\mathbb{C}^2)\),
\beq{eq.vagueconvP}
\int_{\mathcal{I}} \testfunc^*(t) \Bigl( \timechange'(t) \sBesselmat\circ\timechange(t) - \sinemat\circ\logtime(t) \Bigr) \testfunc(t) \diff{t} \probto[n\to\infty] 0.
\eeq
In particular, \(\timechange' (\sBesselmat\circ\timechange) \to \sinemat\circ\logtime\) vaguely on \(\mathcal{I}\) in probability and in law.
\end{theorem}

\begin{proof}
This result is analogous to Theorem~21 of~\cite{painchaud_operator_2026}, and the proof uses the same ideas. To simplify notation throughout the proof, we drop the \(n\) subscript from the shift \(\shift\), with the understanding that any limit as \(\shift\to\infty\) is taken along the (arbitrary) diverging sequence \(\{\shift_n\}_{n\in\mathbb{N}}\). 

We start by making a few simplifications to the problem. First, we cut the time interval to integrate only up to time \(\penultime \defeq (1 - \shift^{\nicefrac{-1}{2}+\alpha}) / \diln\). Note that this can trivially be done when \(\beta \leq 2\), since in that case the interval \([\penultime, 1)\) has eventually left the support of \(\testfunc\), this function being compactly supported in \([0,1)\). If \(\beta > 2\), then because \(\tr\sinemat\circ\logtime\) is a.s.\@ integrable on \([0,1)\), \(\int_{\penultime}^1 \testfunc^*(t) \sinemat\circ\logtime(t) \testfunc(t) \diff{t} \to 0\) a.s.\@ as \(\shift\to\infty\). Moreover, for any \(\varepsilon > 0\), Lemma~\ref{lem.tracebound} provides an \(F_\varepsilon \in L^1[0,1]\) such that \(\timechange' \tr\sBesselmat\circ\timechange \leq F_\varepsilon\) with probability at least \(1-\varepsilon\), and therefore
\[
\int_{\penultime}^1 \testfunc^*(t) \timechange'(t) \sBesselmat\circ\timechange(t) \testfunc(t) \diff{t}
    \leq \norm{\testfunc}_\infty^2 \int_{\penultime}^1 F_\varepsilon(t) \diff{t}
    \to 0.
\]
Hence, this integral also converges to \(0\) in probability as \(\shift\to\infty\), and it suffices to prove~\eqref{eq.vagueconvP} with \(\mathcal{I}\) replaced with \(\timedom \defeq [0, \penultime] = [0, (1 - \shift^{\nicefrac{-1}{2}+\alpha}) / \diln]\).

Then, to simplify the comparison between the two systems, we switch the logarithmic time scale of the sine system from \(\logtime(t) = -\log(1-t)\) to \(\slogtime(t) = -\log(1-\diln t)\) when \(\beta > 2\) and \(a \geq 1\) (in other cases, \(\diln = 1\) so nothing changes). A change of variables in the second term below yields
\begin{multline*}
\int_0^{\penultime} \testfunc^*(t) \bigl( \sinemat\circ\logtime(t) - \sinemat\circ\slogtime(t) \bigr) \testfunc(t) \diff{t}
    = \int_{1-\shift^{\nicefrac{-1}{2}+\alpha}}^{\penultime} \testfunc^*(t) \sinemat\circ\logtime(t) \testfunc(t) \diff{t} \\
    + \int_0^{1-\shift^{\nicefrac{-1}{2}+\alpha}} \Bigl( \testfunc^*(t) \sinemat\circ\logtime(t) \testfunc(t) - \frac{1}{\diln} \testfunc^*(\nicefrac{t}{\diln}) \sinemat\circ\logtime(t) \testfunc(\nicefrac{t}{\diln}) \Bigr) \diff{t}.
\end{multline*}
The first term in the right-hand side vanishes a.s.\@ as \(\shift\to\infty\) by integrability of \(\tr\sinemat\circ\logtime\) on \([0,1)\). The second term can be written as
\[
\Re \int_0^{1-\shift^{\nicefrac{-1}{2}+\alpha}} \Bigl( \testfunc^*(t) + \frac{1}{\sqrt{\diln}} \testfunc^*(\nicefrac{t}{\diln}) \Bigl) \sinemat\circ\logtime(t) \Bigl( \testfunc(t) - \frac{1}{\sqrt{\diln}} \testfunc(\nicefrac{t}{\diln}) \Bigr) \diff{t}.
\]
For \(\shift\) large enough, the integrand is dominated by \(8\norm{\testfunc}_\infty^2 \tr\sinemat\circ\logtime\), which is a.s.\@ integrable, so by continuity of \(\testfunc\), the dominated convergence theorem shows that this integral vanishes a.s.\@ as \(\shift\to\infty\).

Combining the above arguments, we have reduced the problem to showing that
\[
\int_0^{\penultime} \testfunc^*(t) \Bigl( \timechange'(t)\sBesselmat\circ\timechange(t) - \sinemat\circ\slogtime(t) \Bigr) \testfunc(t) \diff{t}
    \probto[\shift\to\infty] 0.
\]
Recall that
\[
\sinemat = \frac{1}{2\det\ABM} \ABM^\transpose \ABM
\qquadtext{with}
\ABM = \begin{pmatrix}
    1 & -\Re\HBM \\
    0 & \Im\HBM
\end{pmatrix}
\]
and where \(\HBM\) is a hyperbolic Brownian motion started at \(i\) in \(\UHP\) and driven by \(W\). Then, as we have seen in~\eqref{eq.sBesselmatpolarcoords} in Section~\ref{sec.polarcoords}, \(\timechange' (\sBesselmat\circ\timechange)\) can be written using the polar coordinates from Proposition~\ref{prop.polarcoords} as
\beq{eq.sBesselmatABM}
\timechange' (\sBesselmat\circ\timechange)
    = \frac{\diln}{2\det\BesselABM} \BesselABM^\transpose \BesselABM + \frac{\diln}{2} \begin{pmatrix}
        e^{2\ampg} \cos 2\phaseg & e^{\sumamps} \cos\sumphases \\
        e^{\sumamps} \cos\sumphases & e^{2\ampf} \cos 2\phasef
    \end{pmatrix}
\eeq
where
\[
\BesselABM \defeq \begin{pmatrix}
    1 & e^{-\diffamps} \cos\diffphases \\
    0 & e^{-\diffamps} \sin\diffphases
\end{pmatrix}.
\]

Our results on the convergence of \(-\exp(-\diffamps-i\diffphases)\) to \(\HBM\circ\slogtime\) from Section~\ref{sec.coupling} imply that the first term of~\eqref{eq.sBesselmatABM} converges to the coefficient matrix of the sine system. Indeed, we can decompose
\begin{multline*}
\frac{\diln}{\det\BesselABM} \BesselABM^\transpose \BesselABM - \frac{1}{\det\ABM\circ\slogtime} (\ABM\circ\slogtime)^\transpose \ABM\circ\slogtime \\
    = \frac{1}{\det\BesselABM} \Bigl( (\BesselABM - \ABM\circ\slogtime)^\transpose (\BesselABM - \ABM\circ\slogtime) + (\ABM\circ\slogtime)^\transpose (\BesselABM - \ABM\circ\slogtime) + (\BesselABM - \ABM\circ\slogtime)^\transpose \ABM\circ\slogtime \Bigr) \\
    + \Bigl( \frac{1}{\det\BesselABM} - \frac{1}{\det\ABM\circ\slogtime} \Bigr) (\ABM\circ\slogtime)^\transpose \ABM\circ\slogtime
    + \frac{\diln - 1}{\det\BesselABM} \BesselABM^\transpose \BesselABM.
\end{multline*}
Since for two matrices \(A, A' \in \mathbb{R}^{2\times 2}\), \(\testfunc^*A^\transpose A'\testfunc = \inprod{A'\testfunc}{A\testfunc} \leq \norm{A}_2 \norm{A'}_2 \norm{\testfunc}_2^2 \leq 4 \maxnorm{A} \maxnorm{A'} \norm{\testfunc}_2^2\) where \(\maxnorm{A} \defeq \max_{1\leq j,k \leq 2} \abs{A_{jk}}\), it follows that
\begin{multline}
\abs[\Big]{\testfunc^* \Bigl( \frac{\diln}{2\det\BesselABM} \BesselABM^\transpose \BesselABM - \frac{1}{2\det\ABM\circ\slogtime} (\ABM\circ\slogtime)^\transpose \ABM\circ\slogtime \Bigr) \testfunc} \\
\begin{multlined}
    \leq \frac{2\norm{\testfunc}_2^2}{\det\BesselABM} \maxnorm{\BesselABM - \ABM\circ\slogtime}^2
    + \frac{4\maxnorm{\ABM\circ\slogtime} \norm{\testfunc}_2^2}{\det\BesselABM} \maxnorm{\BesselABM - \ABM\circ\slogtime} \\
    + 2\maxnorm{\ABM\circ\slogtime}^2 \norm{\testfunc}_2^2 \abs[\Big]{\frac{1}{\det\BesselABM} - \frac{1}{\det\ABM\circ\slogtime}}
    + \frac{2\maxnorm{\BesselABM}^2 \norm{\testfunc}_2^2 \abs{\diln-1}}{\det\BesselABM}.
\end{multlined}
\label{eq.vagueconvP.firsttermbound}
\end{multline}
Like in~\eqref{eq.supGBM} in the proof Corollary~\ref{cor.GBM}, we can find an event with probability at least \(1 - \frac{3}{2\log\shift}\) on which \(\Im\HBM\circ\slogtime\) is bounded by \(\log\shift\) on \(\timedom\). On this event, the quadratic variation of \(\Re\HBM\circ\slogtime\) is bounded on \([0, \lasttime]\) by \(\frac{2}{\beta} \log^3\shift\), and thus it follows from Bernstein's inequality that \(\maxnorm{\ABM\circ\slogtime} \leq \log^2\shift\) on \(\timedom\) with probability at least \(1 - \nicefrac{2}{\log\shift}\) for all \(\shift\) large enough. Then, Corollary~\ref{cor.GBM} and Proposition~\ref{prop.ReHBM} show together that for \(\delta > 0\) small enough, there is a \(C > 0\) such that \(\maxnorm{\BesselABM - \ABM\circ\slogtime} \leq \shift^{\nicefrac{-\alpha}{12}+\delta}\) on \(\timedom\) with probability at least \(1 - \nicefrac{C}{\log\shift}\) for all \(\shift\) large enough. Corollary~\ref{cor.GBM} also shows that there is a \(\gamma > 0\) such that \(\abs[\big]{\nicefrac{1}{\det\BesselABM} - \nicefrac{1}{\det\ABM\circ\slogtime}} \leq \shift^{-\gamma}\) on \(\timedom\) with probability at least \(1 - \nicefrac{1}{\log\shift}\) for all \(\shift\) large enough. Finally, by Proposition~\ref{prop.tracebound}, for all \(\varepsilon > 0\) there is a \(C_\varepsilon > 0\) such that with probability at least \(1 - \nicefrac{\varepsilon}{2}\), \(\nicefrac{1}{\det\BesselABM(t)} \leq C_\varepsilon (1-t)^{\nicefrac{-2}{\beta}-\delta}\) for all \(t \in [0,\lasttime]\). Combining all of the above, we see that with probability at least \(1 - \varepsilon\), \eqref{eq.vagueconvP.firsttermbound} is bounded for all \(\shift\) large enough by
\beq{eq.vagueconvP.firsttermbound2}
\frac{2C_\varepsilon \norm{\testfunc(t)}_2^2 \shift^{\nicefrac{-\alpha}{6}+2\delta}}{(1-t)^{\nicefrac{2}{\beta}+\delta}}
    + \frac{4C_\varepsilon \shift^{\nicefrac{-\alpha}{12}+\delta} \log^2\shift \norm{\testfunc(t)}_2^2}{(1-t)^{\nicefrac{2}{\beta}+\delta}}
    + 2\shift^{-\gamma} \log^4\shift \norm{\testfunc(t)}_2^2
    + \frac{2C_\varepsilon (\log^2\shift + \shift^{\nicefrac{-\alpha}{12}+\delta})^2 \norm{\testfunc(t)}_2^2}{\sqrt{\shift} (1-t)^{\nicefrac{2}{\beta}+\delta}}.
\eeq
These four terms are integrable on the support of \(\testfunc\), either (if \(\beta \leq 2\)) because it is compact in \([0,1)\), or (if \(\beta > 2\)) because we can choose \(\delta\) small enough so that \(\nicefrac{2}{\beta} + \delta < 1\). Integrating~\eqref{eq.vagueconvP.firsttermbound2} over \(\timedom\) therefore replaces the time-dependent parts by constants which do not depend on \(\shift\), and since for fixed \(t\) all terms in~\eqref{eq.vagueconvP.firsttermbound2} vanish as \(\shift\to\infty\), it follows that
\[
\int_0^{\penultime} \testfunc^*(t) \Bigl( \frac{\diln}{2\det\BesselABM(t)} \BesselABM^\transpose(t) \BesselABM(t) - \sinemat\circ\slogtime(t) \Bigr) \testfunc(t) \diff{t}
    \probto[\shift\to\infty] 0.
\]

To complete the proof, it remains to show that the integral of the second term of~\eqref{eq.sBesselmatABM} sandwiched between \(\testfunc^*\) and \(\testfunc\) vanishes as \(\shift\to\infty\). By linearity, it suffices to show that if \(\psi \in \mathscr{C}_c\bigl( \mathcal{I}, \mathbb{R} \bigr)\), then
\beq{eq.vagueconvP.vaguely0}
\int_0^{\lasttime} \psi(t) e^{k\ampg(t)+l\ampf(t)} \cos\bigl( k\phaseg(t)+l\phasef(t) \bigr) \diff{t}
    \probto[\shift\to\infty] 0
\eeq
for \(k,l \in \{0,1,2\}\) with \(k+l=2\). For \(\varepsilon > 0\), let
\[
\mathscr{H}_{\shift,\varepsilon} \defeq \Bigl\{ \forall t \in [0,\lasttime], e^{k\ampg(t)+l\ampf(t)} \leq \frac{C_\varepsilon}{(1-\diln t)^{\nicefrac{2}{\beta}+\gamma}} \Bigr\}
\qquadtext{with}
\gamma \defeq \begin{cases}
    \nicefrac{1}{4} & \text{if } \beta \leq 2, \\
    \nicefrac{1}{4} - \nicefrac{1}{2\beta} & \text{if } \beta > 2,
\end{cases}
\]
and where \(C_\varepsilon > 0\) is chosen so that \(\pprob{\mathscr{H}_{\shift,\varepsilon}} \geq 1 - \varepsilon\), which is always possible by Proposition~\ref{prop.tracebound}. The upper bound \((1 - t)^{\nicefrac{2}{\beta}+\gamma}\) is always integrable on the support of \(\psi\), either (if \(\beta \leq 2\)) because it is compact in \([0,1)\) or (if \(\beta > 2\)) because \(\nicefrac{2}{\beta} + \gamma < 1\). 

To prove~\eqref{eq.vagueconvP.vaguely0}, we replace \(\psi\) with a piecewise constant approximation. To do so, we partition \([0, \lasttime]\) by setting
\[
t_j \defeq \frac{j\pi}{2\diln^2} \shift^{-(\nicefrac{1}{2}-p)}
\qquadtext{for}
p \defeq \begin{cases}
    \nicefrac{1}{4} & \text{if } \beta \leq 2, \\
    \nicefrac{1}{4} + \nicefrac{1}{2\beta} & \text{if } \beta > 2
\end{cases}
\]
until we reach an index \(N\) for which this would give \(t_N \geq \lasttime\), and then we set \(t_N \defeq \lasttime\). We can define a discretization \(\hat{\psi} \colon \mathcal{I} \to \mathbb{R}\) of \(\psi\) by setting
\[
\hat{\psi}
    \equiv \hat{\psi}_j
    \defeq \frac{1}{t_j-t_{j-1}} \int_{t_{j-1}}^{t_j} \psi(s) \diff{s}
\quadtext{on}
[t_{j-1}, t_j)
\]
whenever \(j \leq N\), and \(\hat{\psi}(t) \defeq 0\) for \(t \geq \lasttime\). Note that by definition, if \(t \in [t_{j-1}, t_j)\) then \(\abs{\hat{\psi}(t) - \psi(t)} \leq \omega_\psi\bigl( \abs{t_j - t_{j-1}} \bigr) \leq \omega_\psi\bigl( \frac{\pi}{2\diln^2} \shift^{-(\nicefrac{1}{2}-p)} \bigr)\) where \(\omega_\psi\) is a modulus of continuity for \(\psi\), so \(\norm{(\hat{\psi} - \psi) \charf{[0,\lasttime)}}_\infty \leq \omega_\psi\bigl( \frac{\pi}{2\diln^2} \shift^{-(\nicefrac{1}{2}-p)} \bigr) \to 0\). Because the integrand in~\eqref{eq.vagueconvP.vaguely0} can be dominated on \(\mathscr{H}_{\shift,\varepsilon}\) by an integrable function that does not depend on \(\shift\), this shows that it suffices to prove~\eqref{eq.vagueconvP.vaguely0} with \(\psi\) replaced by \(\hat{\psi}\). 

Now, from the SDEs that are satisfied by the polar coordinates, a straightforward application of Itô's formula shows that
\begin{multline*}
e^{-k\ampg(t)-l\ampf(t)} \diff{\Bigl( e^{k\ampg+l\ampf} \sin\bigl( k\phaseg+l\phasef \bigr) \Bigr)}(t) \\
    = - 4\diln\sqrt{\shift} \cos\bigl( k\phaseg(t)+l\phasef(t) \bigr) \diff{t} + \sumdrift(t) \frac{\diln}{1-\diln t} \diff{t} + \sumdiff(t) \sqrt{\frac{\diln}{1-\diln t}} \diff{\sBM}(t)
\end{multline*}
where the processes \(\sumdrift\) and \(\sumdiff\), which can be expressed as polynomials in trigonometric functions of \(\phaseg\) and \smash{\(\phasef\)}, are bounded by constants that depend only on \(\beta\) and \(a\). It follows that
\begin{subequations}
\begin{align}
& \int_0^{\lasttime} \hat{\psi}(t) e^{k\ampg(t)+l\ampf(t)} \cos\bigl( k\phaseg(t) + l\phasef(t) \bigr) \diff{t} \notag \\
    &\hspace*{22mm} = - \frac{1}{4\diln\sqrt{\shift}} \sum_{j=1}^N \hat{\psi}_j e^{k\ampg(t)+l\ampf(t)} \sin\bigl( k\phaseg(t) + l\phasef(t) \bigr) \biggr\rvert_{t_{j-1}}^{t_j} 
    \label{eq.vagueconvP.psihat.parts} \\
    &\hspace*{44mm} + \frac{1}{4\diln\sqrt{\shift}} \int_0^{\lasttime} \hat{\psi}(t) e^{k\ampg(t)+l\ampf(t)} \sumdrift(t) \frac{\diln}{1-\diln t} \diff{t} 
    \label{eq.vagueconvP.psihat.drift} \\
    &\hspace*{44mm} + \frac{1}{4\diln\sqrt{\shift}} \int_0^{\lasttime} \hat{\psi}(t) e^{k\ampg(t)+l\ampf(t)} \sumdiff(t) \sqrt{\frac{\diln}{1-\diln t}} \diff{\sBM}(t). 
    \label{eq.vagueconvP.psihat.diff}
\end{align}
\end{subequations}
Note that on \(\mathscr{H}_{\shift,\varepsilon}\), for all \(\shift\) large enough, \(\exp\bigl( k\ampg(t)+l\ampf(t) \bigr) \leq C_\varepsilon \shift^{p-\nicefrac{\gamma}{2}}\) for all \(t \in \supp\psi\). Indeed, this is clear when \(\beta \leq 2\) as in that case \(\supp\psi\) is contained in \([0,1)\), and when \(\beta > 2\), then \(\exp\bigl( k\ampg(t)+l\ampf(t) \bigr) \leq C_\varepsilon (1-\diln t)^{\nicefrac{-2}{\beta}-\gamma} \leq C_\varepsilon \shift^{\nicefrac{1}{\beta}+\nicefrac{\gamma}{2}} = C_\varepsilon \shift^{p-\nicefrac{\gamma}{2}}\). Since \(t_{N-1} < \lasttime\), we know that \(N < 1 + \frac{2}{\pi} \shift^{\nicefrac{1}{2}-p}\), and it follows that~\eqref{eq.vagueconvP.psihat.parts} is bounded on \(\mathscr{H}_{\shift,\varepsilon}\) by
\[
\frac{C_\varepsilon \norm{\psi}_\infty \shift^{p-\nicefrac{\gamma}{2}}}{4\diln\sqrt{\shift}} \Bigl( 1 + \frac{2\shift^{\nicefrac{1}{2}-p}}{\pi} \Bigr)
    \leq \frac{C_\varepsilon \norm{\psi}_\infty}{\pi} \shift^{\nicefrac{-\gamma}{2}},
\]
and therefore~\eqref{eq.vagueconvP.psihat.parts} converges to \(0\) in probability as \(\shift\to\infty\). Then~\eqref{eq.vagueconvP.psihat.drift} is bounded on \(\mathscr{H}_{\shift,\varepsilon}\) by
\[
\frac{C_\varepsilon \norm{\psi}_\infty \norm{\sumdrift}_\infty \shift^{p-\nicefrac{\gamma}{2}}}{4\diln\sqrt{\shift}} \int_0^{\lasttime} \frac{\diln}{1-\diln t} \diff{t}
    = \frac{C_\varepsilon \norm{\psi}_\infty \norm{\sumdrift}_\infty}{8\diln} \shift^{-(\nicefrac{1}{2}-p+\nicefrac{\gamma}{2})} \log\shift,
\]
so~\eqref{eq.vagueconvP.psihat.drift} also converges to \(0\) in probability. Finally, in the same way as in the last case, the quadratic variation of~\eqref{eq.vagueconvP.psihat.diff} is bounded by a negative power of \(\shift\), so Bernstein's inequality shows that it converges to \(0\) in probability as well. This finishes to prove~\eqref{eq.vagueconvP.vaguely0} and concludes the proof.
\end{proof}

\subsection{Compact convergence of transfer matrices}

The compact convergence of transfer matrices can be deduced from the vague convergence of coefficient matrices that was proven in Theorem~\ref{thm.vagueconvP}, since the mapping from coefficient matrices to transfer matrices is continuous on domains on which the trace of coefficient matrices is dominated by a locally integrable function that is integrable near limit circle endpoints (see~\cite[Theorem~2.12]{painchaud_canonical_2026} for a proof of this result, which is a simple generalization of~\cite[Theorem~5.7(a)]{remling_spectral_2018}). 

The specific version of this continuity result that we use is the following, which is adapted to the convergence in probability of random canonical systems.

\begin{proposition}[Proposition 2.25 of \cite{painchaud_canonical_2026}]
\label{prop.CStoTM}
Let \(\mathcal{I} = [a,b)\) or \([a,b]\). Let \(H_n, H\) be random coefficient matrices, all limit circle at \(a\), and also at \(b\) in case \(\mathcal{I} = [a,b]\). Let \(T_{H_n}, T_H \colon \mathcal{I} \times \mathbb{C} \to \mathbb{C}^{2\times 2}\) be their transfer matrices. Suppose that for any \(\varepsilon > 0\), there are \(f_\varepsilon, g_\varepsilon \in L^1\loc(\mathcal{I})\) such that \(\bprob{\tr H \leq f_\varepsilon} \geq 1 - \varepsilon\) and \(\bprob{\tr H_n \leq g_\varepsilon} \geq 1 - \varepsilon\) for all \(n\) large enough. If \(H_n \to H\) vaguely on \(\mathcal{I}\) in probability, then \(T_{H_n} \to T_H\) compactly on \(\mathcal{I} \times \mathbb{C}\) in probability.
\end{proposition}

From this general result, it is easy to deduce the convergence of the transfer matrix of the shifted Bessel system to that of the sine system from Theorem~\ref{thm.vagueconvP} and Lemma~\ref{lem.tracebound}. The following also implies Corollary~\ref{cor.TMconv}.

\begin{corollary}[thm.vagueconvP]
\label{cor.TMconvP}
\setshift{E_n}
Let \(\mathcal{I} \defeq [0,1)\) if \(\beta \leq 2\) and \(\mathcal{I} \defeq [0,1]\) if \(\beta > 2\). With \(\sBesselmat* \defeq \timechange' (\sBesselmat\circ\timechange)\) and \(\sinemat* \defeq \sinemat\circ\logtime\), let \(\sBesselTM*, \sineTM* \colon \mathcal{I} \times \mathbb{C} \to \mathbb{C}^{2\times 2}\) be the transfer matrices of the associated canonical systems, both defined on the probability space from Lemma~\ref{lem.coupling}. Then \(\sBesselTM* \to \sineTM*\) compactly on \(\mathcal{I} \times \mathbb{C}\) in probability as \(n\to\infty\).
\end{corollary}

\begin{proof}
By Proposition~\ref{prop.CStoTM}, this result follows from the vague convergence \(\timechange' (\sBesselmat\circ\timechange) \to \sinemat\circ\logtime\) proved in Theorem~\ref{thm.vagueconvP} provided there are suitable bounds on the traces of these coefficient matrices. 

Recall that by definition, \(\tr\sinemat = \frac{1}{2\Im\HBM} (1 + \abs{\HBM}^2)\) where \(\HBM\) is a hyperbolic Brownian motion with variance \(\nicefrac{4}{\beta}\), started at \(i\) in the upper half-plane. Because the imaginary part of \(\HBM\) is a geometric Brownian motion, it can be verified from properties of Brownian motion (see the proof of~\cite[Theorem~22]{painchaud_operator_2026} for details) that for all \(\varepsilon, \delta > 0\), there is a \(C_\varepsilon > 0\) such that
\beq{eq.TMconvP.sinebound}
\bprob[\Big]{\forall t \in [0,1), \tr\sinemat\circ\logtime(t) \leq \frac{C_\varepsilon}{(1-t)^{\nicefrac{2}{\beta}+\delta}}} \geq 1 - \varepsilon.
\eeq
The function \(t \mapsto C_\varepsilon (1-t)^{-\nicefrac{2}{\beta}-\delta}\) is always \(L^1\loc[0,1)\), and it is also integrable near \(1\) when \(\beta > 2\) if \(\delta\) is taken small enough. 

When \(\abs{a} < 1\), Lemma~\ref{lem.tracebound} ensures the existence of a dominating function on the trace of \(\timechange' (\sBesselmat\circ\timechange)\), and with~\eqref{eq.TMconvP.sinebound} this suffices to deduce the convergence of the transfer matrices from Proposition~\ref{prop.CStoTM} together with the vague convergence of the coefficient matrices. When \(a \geq 1\) and \(\beta > 2\), we rather have a suitable bound on \(\timechange' (\tr\sBesselmat\circ\timechange)\) from Proposition~\ref{prop.tracebound}, and the convergence of transfer matrices follows as before.

This only leaves out the case \(\beta \leq 2\) and \(a \geq 1\). In that case, we have no bound on \(\timechange' (\tr\sBesselmat\circ\timechange)\) on \((1-\nicefrac{1}{\sqrt{\shift}}, 1)\), but we still have the bound from Proposition~\ref{prop.tracebound} up to \(1 - \nicefrac{1}{\sqrt{\shift}}\). Hence, for any \(b \in (0,1)\), this bound holds on \([0,b]\) for all \(\shift\) large enough, and it follows that \(\sBesselTM* \to \sineTM*\) compactly on \([0,b] \times \mathbb{C}\) in probability. Because \(b\) is arbitrary in \((0,1)\), this implies that \(\sBesselTM* \to \sineTM*\) compactly on \([0,1) \times \mathbb{C}\) in probability. 
\end{proof}

\section{Spectral convergence}
\label{sec.spectralconv}

In this section, we prove the compact convergence in probability of the canonical systems' Weyl--Titchmarsh functions. We start by recalling basic facts about Weyl theory for canonical systems and describe the relationship between the convergence of Weyl--Titchmarsh and that of transfer matrices. We will see that this relationship depends strongly on the behavior of the limit system at the endpoints of the time domain. For this reason, the proof of Theorem~\ref{thm.WTconv} is broken in two parts: first we carry out the proof for \(\beta \leq 2\), and then we do it for \(\beta > 2\). As part of the proof for \(\beta > 2\) we also derive the asymptotics announced in Theorem~\ref{thm.asymptotics}.

\subsection{Convergence of Weyl--Titchmarsh functions of canonical systems}

Here, we recall some basic facts about the Weyl theory of canonical systems, and introduce some convergence results that will be used in the sequel. For a more complete introduction to the Weyl theory of canonical systems, we refer the reader to~\cite[Section~3.4]{remling_spectral_2018}.

Consider a canonical system \(Ju' = -zHu\) on an interval \((a,b)\), and suppose that it is limit circle at \(a\) and subject to the boundary condition \(e_0^*Ju(a) = 0\). Suppose that this system is either limit point at \(b\) or also subject to a boundary condition at \(b\) if it is limit circle, which in both cases yields a self-adjoint realization of the system. The Weyl--Titchmarsh function for this problem is the map \(m\colon \UHP \to \clUHP\) given by \(m(z) \defeq \frac{u_1(a,z)}{u_2(a,z)}\), where \(u\colon [a,b) \times \mathbb{C} \to \mathbb{C}^2\) is defined so that for each \(z \in \mathbb{C}\), \(u(\cdot,z)\) solves the canonical system and is either integrable near \(b\) if it is limit point or satisfies the boundary condition at \(b\) if it is limit circle. Here, by an \emph{integrable} solution \(u\), we mean one such that \(\int_a^b u^*(t) H(t) u(t) \diff{t} < \infty\). Note also that it is clear from the definition that the Weyl--Titchmarsh function is not modified by any time change made to the system: if \(H\) is replaced with \(\eta' (H\circ\eta)\) for some increasing \(\mathscr{C}^1\) bijection \(\eta\), then the solution \(u\) is replaced with \(u\circ\eta\) but this has no effect on the Weyl--Titchmarsh function.

A canonical system's Weyl--Titchmarsh function is always a generalized Herglotz function, that is, a holomorphic function \(\UHP \to \clUHP\). As such, it induces a spectral measure \(\mu\) which can be recovered through Stieltjes inversion. In particular, the singular part of \(\mu\) can be obtained by the relation \(\mu\bigl( \{t\} \bigr) = -i \lim_{\varepsilon \downarrow 0} \varepsilon m(t + i\varepsilon)\), which holds for all \(t \in \mathbb{R}\). We refer the reader to~\cite[Appendix~F]{schmudgen_unbounded_2012} for more details on Herglotz functions. 

The Weyl--Titchmarsh function can be described explicitly in terms of the transfer matrix, and of the boundary condition at \(b\) in the limit circle case. In the latter case, if the system is given the boundary condition \(e_\phi^*Ju(b) = 0\) for some \(\phi \in [0,\pi)\), then the properties of the transfer matrix imply that the Weyl--Titchmarsh function is \(m(z) = \projection T(b,z)^{-1} e_\phi\) where \(\projection(z,w) \defeq \nicefrac{z}{w}\). When the system is limit point at \(b\), a fundamental result of Weyl theory shows that the function \(z \mapsto \lim_{t\to b} \projection T(t,z)^{-1} e_\phi\) does not depend on \(\phi\), and is in fact always equal to the Weyl--Titchmarsh function. 

It turns out that the mapping from a transfer matrix to the corresponding Weyl--Titchmarsh function is continuous on suitable domains. The way this relationship works depends on whether the systems involved are limit point or limit circle at the right endpoint. The simplest case is the limit point case; then we have the following result, which is a simple extension of~\cite[Theorem~5.7(a)]{remling_spectral_2018}.

\begin{theorem}
\label{thm.TMconv.LP}
Let \(\TM[a,b)\) denote the set of transfer matrices of canonical systems on \((a,b)\) that are limit circle at \(a\) with boundary condition \(e_0^*Ju(a) = 0\), and write \(\TMLP[a,b) \subset \TM[a,b)\) for the subset of those which are additionally limit point at \(b\). Let \(T_n \in \TM[a,b)\) and \(T \in \TMLP[a,b)\), and let \(m_n, m \in \Hol(\UHP, \clUHP)\) be the Weyl--Titchmarsh functions of the corresponding systems. If \(T_n \to T\) compactly, then \(m_n \to m\) compactly. In particular, the mapping \(\TMLP[a,b) \to \Hol(\UHP, \clUHP)\) sending a transfer matrix to the corresponding Weyl--Titchmarsh function is continuous with respect to the topology of compact convergence.
\end{theorem}

When the systems are limit circle at the right endpoint, the convergence of boundary conditions is also required for the Weyl--Titchmarsh functions to converge. For that case, we have the following result.

\begin{theorem}[Theorem 11 of \cite{painchaud_operator_2026}]
\label{thm.TMconv.LC}
Let \(\TMLC[a,b]\) denote the set of transfer matrices of canonical systems which are limit circle at both \(a\) and \(b\) with boundary condition \(e_0^*Ju(a) = 0\). Let \(\RiemannSphere^\UHP\) denote the set of mappings from \(\UHP\) to the Riemann sphere \(\RiemannSphere\), and define \(M\colon \TMLC[a,b] \times \mathscr{C}(\UHP, \mathbb{C}^2) \to \RiemannSphere^\UHP\) by setting \(M(T,w) \colon \UHP \to \RiemannSphere\) as \(M(T,w)(z) \defeq \projection T(b,z)^{-1}w(z)\). Then \(M\) is continuous on \(M^{-1}\bigl( \Hol(\UHP, \clUHP) \bigr)\) under the topology of compact convergence.
\end{theorem}

When \(w \in \mathscr{C}(\UHP, \mathbb{C}^2)\) is constant with \(\nicefrac{w_1}{w_2} \equiv \cot\theta \in \mathbb{R} \cup \{\infty\}\) for a \(\theta \in [0,\pi)\), then \(M(T,w)\) is the Weyl--Titchmarsh function of a canonical system on \([a,b]\) with boundary condition \(e_\theta^* Ju(b) = 0\). When \(w\) is not constant, then it should be understood as a boundary condition which is allowed to depend on the spectral parameter. Of course, this does not result in an actual boundary condition for the system, but \(M(T,w)\) can still be a perfectly valid Weyl--Titchmarsh function in that case. 

An important example of this situation is the following. Suppose that, for some \(c > b\), \(T\) is the restriction to \([a,b] \times \mathbb{C}\) of the transfer matrix \(\tilde{T} \colon [a,c) \times \mathbb{C} \to \mathbb{C}^{2\times 2}\) of a system that is limit point at \(c\). The Weyl--Titchmarsh function of this system is then \(m(z) = \projection u(a,z)\) for an appropriate \(u\colon [a,c) \times \mathbb{C} \to \mathbb{C}^2\), and it can be shown from the definition of the transfer matrix that in fact \(u(a,z) = \tilde{T}(t,z)^{-1} u(t,z)\) for any \(t \in [a,c)\). Hence, \(m(z) = \projection T(b,z)^{-1} u(b,z) = M\bigl( T, u(b,\cdot) \bigr)(z)\). In other words, \(u(b,z)\) can be seen as a \(z\)-dependent \enquote{boundary condition} that one must add at \(b\) when restricting the system on \([a,c)\) to \([a,b]\) in order to preserve the spectral information. In the context of the hard edge to bulk transition, this is the idea behind our choice of taking \(\diln < 1\) when \(\beta > 2\) and \(a \geq 1\): it allows to make the Bessel system limit circle at \(1\), at the cost of having a boundary condition that depends on the spectral parameter.

\subsection{\texorpdfstring{\(\beta \leq 2\)}{beta < 2}: Limit point case}
\label{sec.spectralconv.beta<2}

When \(\beta \leq 2\), the stochastic sine canonical system is limit point at its right endpoint \(1\). Therefore, the convergence of the Weyl--Titchmarsh function of \(\sBesselop\) to that of the sine system is a simple consequence of the convergence of their transfer matrices. The following result proves the part of Theorem~\ref{thm.WTconv} about \(\beta \leq 2\). 

\begin{corollary}[thm.vagueconvP]
\label{cor.WTconvP.beta<2}
\setshift{E_n}
Suppose \(\beta \leq 2\), and let \(\sBesselWT, \sineWT \colon \UHP \to \clUHP\) be the Weyl--Titchmarsh functions of the shifted Bessel and sine canonical systems, both defined on the probability space from Lemma~\ref{lem.coupling}. Then \(\sBesselWT \to \sineWT\) compactly on \(\UHP\) in probability as \(n\to\infty\). 
\end{corollary}

\begin{proof}
\setshift{E_n}
Recall that on a separable metric space, convergence in probability is equivalent to each subsequence having a further subsequence that converges a.s.~\cite[Lemma~5.2]{kallenberg_foundations_2021}. This characterization can be used to prove~Corollary~\ref{cor.WTconvP.beta<2} in the same way as it can be used to prove the continuous mapping theorem (as done e.g.\@ in~\cite[Lemma~5.3]{kallenberg_foundations_2021}). 

Pick an arbitrary subsequence \(S \subset \mathbb{N}\). By Corollary~\ref{cor.TMconvP}, as \(n\to\infty\) the transfer matrix \(\sBesselTM*\) of the shifted Bessel system converges compactly on \([0,1) \times \mathbb{C}\) in probability to the transfer matrix \(\sineTM*\) of the sine system, so there is a further subsequence \(S' \subset S\) along which \(\sBesselTM* \to \sineTM*\) a.s. By Theorem~\ref{thm.TMconv.LP}, it follows that \(\sBesselWT \to \sineWT\) compactly on \(\UHP\) a.s.\@ along \(S'\). Because the subsequence \(S\) was arbitrary, it follows that \(\sBesselWT \to \sineWT\) as \(n\to\infty\) compactly on \(\UHP\) in probability.
\end{proof}

\subsection{\texorpdfstring{\(\beta > 2\)}{beta > 2}: Limit circle case}
\label{sec.spectralconv.beta>2}

When \(\beta > 2\), the stochastic sine canonical system is limit circle at its right endpoint \(1\). In that case, the convergence of the Bessel system's transfer matrix to that of the sine system is not sufficient to deduce the convergence of the Weyl--Titchmarsh functions: the boundary conditions must also converge. Our proof of the convergence of the boundary conditions relies on asymptotics of solutions to \(\Besselop f = f\) towards \(-\infty\) when \(\Besselop\) is defined on the whole real line from a two-sided Brownian motion. This argument is similar to the one used to prove the convergence of the boundary conditions in~\cite[Section~7]{painchaud_operator_2026} in the context of the soft edge to bulk transition.

This section is split in two parts. First, in Section~\ref{sec.spectralconv.beta>2.polarcoords}, we find polar coordinates for solutions to \(\Besselop f = f\) on the negative real line, and we analyze their asymptotic behavior towards \(-\infty\). Then, we use this in Section~\ref{sec.spectralconv.beta>2.WT} to prove the convergence of the boundary conditions, and deduce the convergence of the Weyl--Titchmarsh functions.

\subsubsection{Polar coordinates and their asymptotic behavior}
\label{sec.spectralconv.beta>2.polarcoords}

In this section, we derive polar coordinates for solutions to \(\Besselop f = \lambda f\) towards \(-\infty\) for a positive spectral parameter \(\lambda\). Then, we obtain descriptions of the asymptotic behavior of these polar coordinates. The results of this section (in particular Propositions~\ref{prop.asymptotics.amp} and~\ref{prop.asymptotics.phase}), together, prove Theorem~\ref{thm.asymptotics}.

\begin{proposition}
\label{prop.polarcoords2}
Suppose \(\Besselop\) is defined from a two-sided standard Brownian motion \(B\), and write \(\rBM(t) \defeq B(-t)\). If \(f\) solves \(\Besselop f = \lambda f\) with \(\lambda > 0\) and if \(f(0)\) and \(f'(0)\) are independent of the restriction of \(\rBM\) to \((0,\infty)\), then for \(t \geq 0\),
\[
f(-t) = \lambda^{\nicefrac{-1}{4}} e^{\ramp(t) - \nicefrac{t}{4}} \cos\rphase(t)
\qquadtext{and}
f'(-t) = - \lambda^{\nicefrac{1}{4}} e^{\ramp(t) + \nicefrac{t}{4}} \sin\rphase(t)
\]
where \(\ramp\) and \(\rphase\) solve
\begin{align*}
\diff{\ramp}(t) 
    & = \Bigl( \frac{1}{2\beta} - \frac{a}{2} \Bigr) \diff{t} + \biggl( \Bigl( \frac{1}{4} + \frac{a}{2} \Bigr) \cos 2\rphase(t) - \frac{1}{2\beta} \cos 4\rphase(t) \biggr) \diff{t} + \frac{2}{\sqrt{\beta}} \sin^2\rphase(t) \diff{\rBM}(t), \\
\diff{\rphase}(t)
    & = - \sqrt{\lambda} e^{\nicefrac{t}{2}} \diff{t} - \biggl( \Bigl( \frac{1}{4} + \frac{a}{2} \Bigr) \sin 2\rphase(t) - \frac{1}{2\beta} \sin 4\rphase(t) \biggr) \diff{t} + \frac{1}{\sqrt{\beta}} \sin 2\rphase(t) \diff{\rBM}(t)
\end{align*}
with \(\ramp(0) = \frac{1}{2} \log\bigl( f^2(0) + {f'}^2(0) \bigr)\) and \(\rphase(0) = \arctan\bigl( \nicefrac{f'(0)}{f(0)} \bigr)\).
\end{proposition}

\begin{proof}
Let \(y(t) \defeq f(-t)\). Reversing time in the equation \(\Besselop f = \lambda f\) yields \(- \frac{1}{\rBesselweight(t)} \bigl( \rBesselcoeff y' \bigr)'(t) = \lambda y(t)\) where \(\rBesselcoeff(t) \defeq \Besselcoeff(-t) = \exp\bigl( at - \frac{2}{\sqrt{\beta}} \rBM(t) \bigr)\) and \(\rBesselweight(t) \defeq \Besselweight(-t) = e^{t} \rBesselcoeff(t)\). Note that \(\diff{\rBesselcoeff}(t) = (a + \frac{2}{\beta}) \rBesselcoeff(t) - \frac{2}{\sqrt{\beta}} \rBesselcoeff(t) \diff{\rBM}(t)\), so applying Itô's formula in the same way as we did in Section~\ref{sec.solprops} to obtain the SDE~\eqref{eq.BesselSDE} for \(f'\), we get
\beq{eq.rBesselSDE}
\diff{y'}(t)
    = - \lambda e^t y(t) \diff{t} - \Bigl( a - \frac{2}{\beta} \Bigr) y'(t) \diff{t} + \frac{2}{\sqrt{\beta}} y'(t) \diff{\rBM}(t).
\eeq
Now, set \(r\) and \(\xi\) as real processes that satisfy \(e^{r+i\xi} = Sy + \nicefrac{iy'}{S}\) where \(S\) is a (time-dependent) scaling factor to be determined later. Then by Itô's formula, omitting the explicit time dependences to simplify notation,
\begin{align*}
\diff{r} + i\diff{\xi}
    & = \frac{1}{Sy + \nicefrac{iy'}{S}} \Bigl( S'y \diff{t} + Sy' \diff{t} + \frac{i}{S} \diff{y'} - \frac{iS'}{S^2} y' \diff{t} \Bigr) + \frac{1}{2(Sy + \nicefrac{iy'}{S})^2} \frac{1}{S^2} \diff{\quadvar{y'}} \\
    & = e^{-2r} \biggl( \Bigl( SS'y^2 + S^2yy' - \frac{S'{y'}^2}{S^3} \Bigr) \diff{t} + \frac{y'}{S^2} \diff{y'} - i \Bigl( \frac{2S'yy'}{S} + {y'}^2 \Bigr) \diff{t} + iy \diff{y'} \biggr) \\
    &\hspace*{66mm} + \frac{e^{-4r}}{2S^2} \Bigl( S^2y^2 - \frac{{y'}^2}{S^2} - 2iyy' \Bigr) \diff{\quadvar{y'}} \\
    & = e^{-2r} \biggl( \Bigl( SS'y^2 + S^2yy' - \frac{S'{y'}^2}{S^3} - \frac{\lambda e^t yy'}{S^2} - \Bigl( a - \frac{2}{\beta} \Bigr) \frac{{y'}^2}{S^2} \Bigr) \diff{t} + \frac{2}{\sqrt{\beta}} \frac{{y'}^2}{S^2} \diff{\rBM} \\
    &\hspace*{11mm} - i \Bigl( \frac{2S'yy'}{S} + {y'}^2 + \lambda e^t y^2 + \Bigl( a - \frac{2}{\beta} \Bigr) yy' \Bigr) \diff{t} + \frac{2i}{\sqrt{\beta}} yy' \diff{\rBM} \biggr) 
    + \frac{2}{\beta} e^{-4r} \Bigl( y^2{y'}^2 - \frac{{y'}^4}{S^4} - \frac{2iy{y'}^3}{S^2} \Bigr) \diff{t}.
\end{align*}
With \(S(t) \defeq \lambda^{\nicefrac{1}{4}} e^{\nicefrac{t}{4}}\), the second and fourth terms of the real part cancel out, and the imaginary part simplifies with \({y'}^2 + \lambda e^t y^2 = S^2 e^{2r}\). This also gives \(S'(t) = \frac{1}{4} S(t)\), so that taking the real and imaginary parts in the above while substituting \(Sy = e^r \cos\xi\) and \(\nicefrac{y'}{S} = e^r \sin\xi\) yields
\begin{align*}
\diff{r}
    & = \Bigl( \frac{1}{4} \cos^2\xi - \Bigl( \frac{1}{4} + a - \frac{2}{\beta} \bigr) \sin^2\xi + \frac{2}{\beta} \sin^2\xi (\cos^2\xi - \sin^2\xi) \Bigr) \diff{t} + \frac{2}{\sqrt{\beta}} \sin^2\xi \diff{\rBM}
\shortintertext{and}
\diff{\xi}
    & = - \sqrt{\lambda} e^{\nicefrac{t}{2}} \diff{t} - \biggl( \Bigl( \frac{1}{2} + a - \frac{2}{\beta} \Bigr) \cos\xi \sin\xi + \frac{4}{\beta} \cos\xi \sin^3\xi \biggr) \diff{t} + \frac{2}{\sqrt{\beta}} \cos\xi \sin\xi \diff{\rBM}.
\end{align*}
Simplifying with trigonometric identities shows that \(\ramp \defeq r\) and \(\rphase \defeq \xi\) solve the announced SDEs. The representations for \(f\) and \(f'\) are recovered directly using that \(f(-t) = y(t) = \frac{1}{S(t)} e^{\ramp(t)} \cos\rphase(t)\) and that \(f'(-t) = - y'(t) = - S(t) e^{\ramp(t)} \sin\rphase(t)\). 
\end{proof}

We now obtain asymptotic descriptions of these polar coordinates, starting with the radial one.

\begin{proposition}
\label{prop.asymptotics.amp}
Let \(\ramp\) solve the SDE from Proposition~\ref{prop.polarcoords2}. Then for any \(\varepsilon, \delta > 0\), there is a \(C > 0\) such that
\[
\bprob[\bigg]{\forall t \geq 0, \abs[\Big]{\ramp(t) - \ramp(0) - \Bigl( \frac{1}{2\beta} - \frac{a}{2} \Bigr) t} \leq C(1 + t^{\nicefrac{1}{2}+\delta})}
    \geq 1 - \varepsilon.
\]
\end{proposition}

\begin{proof}
Because \(\ramp\) solves the SDE from Proposition~\ref{prop.polarcoords2}, it suffices to show that for each \(\varepsilon, \delta > 0\), there are \(C, C' > 0\) such that for \(k \in \{2,4\}\),
\beq{eq.asymptotics.amp.drift}
\bprob[\bigg]{\sup_{t\geq 0} \abs[\Big]{\int_0^t e^{ki\rphase(s)} \diff{s}} \geq C} \leq \varepsilon
\eeq
and
\beq{eq.asymptotics.amp.martingale}
\bprob[\bigg]{\forall t \geq 0, \abs{M(t)} \leq C'(1 + t^{\nicefrac{1}{2}+\delta})} \geq 1 - \varepsilon
\qquadtext{where}
M(t) \defeq \frac{2}{\sqrt{\beta}} \int_0^t \sin^2\rphase(s) \diff{\rBM}(s).
\eeq
Notice that \(\quadvar{M}(t) \leq \nicefrac{4t}{\beta}\), so~\eqref{eq.asymptotics.amp.martingale} directly follows from properties of continuous martingales (more precisely, from the generalization of~\eqref{eq.Weyl.goodevent} for continuous martingales; see e.g.~\cite[Proposition~19]{painchaud_operator_2026} for a precise statement and a proof). 

To prove~\eqref{eq.asymptotics.amp.drift}, note that by Itô's formula, \(e^{ki\rphase}\) satisfies on \([0,\infty)\) the SDE
\beq{eq.asymptotics.expSDE}
\diff{\bigl( e^{ki\rphase} \bigr)}(t)
    = -ki \sqrt{\lambda} e^{ki\rphase(t) + \nicefrac{t}{2}} \diff{t} + R_k(t) e^{ki\rphase(t)}\diff{t} + S_k(t) e^{ki\rphase(t)} \diff{\rBM}(t)
\eeq
where
\beq{eq.asymptotics.defRS}
R_k \defeq - \frac{k^2}{4\beta} - ki \Bigl( \frac{1}{4} + \frac{a}{2} \Bigr) \sin 2\rphase + \frac{ki}{2\beta} \sin 4\rphase + \frac{k^2}{4\beta} \cos 4\rphase
\qquadtext{and}
S_k \defeq \frac{ki}{\sqrt{\beta}} \sin 2\rphase.
\eeq
It follows that
\begin{align*}
\abs[\Big]{\int_0^t e^{ki\rphase(s)} \diff{s}}
    & = \frac{1}{k\sqrt{\lambda}} \abs[\bigg]{\int_0^t e^{ki\rphase(s)-\nicefrac{s}{2}} R_k(s) \diff{s} + \int_0^t e^{ki\rphase(s)-\nicefrac{s}{2}} S_k(s) \diff{\rBM}(s) - \int_0^t e^{\nicefrac{-s}{2}} \diff{\bigl( e^{ki\rphase} \bigr)}(s)} \\
    & = \frac{1}{k\sqrt{\lambda}} \abs[\bigg]{\int_0^t e^{ki\rphase(s)-\nicefrac{s}{2}} \Bigl( R_k(s) - \frac{1}{2} \Bigr) \diff{s} + \int_0^t e^{ki\rphase(s)-\nicefrac{s}{2}} S_k(s) \diff{\rBM}(s) - e^{ki\rphase(s)-\nicefrac{s}{2}} \Bigr\rvert_0^t}.
\end{align*}
The last term is bounded, and since \(R_k\) is bounded and \(e^{\nicefrac{-s}{2}}\) is integrable the first integral is also bounded. This leaves only the second integral, whose real and imaginary parts have quadratic variations bounded by \(\frac{k^2}{\beta} \int_0^t e^{-s} \diff{s} \leq \frac{k^2}{\beta e}\). By Bernstein's inequality for martingales, it follows that for all \(x > 0\),
\[
\bprob[\bigg]{\sup_{t\geq 0} \frac{1}{k\sqrt{\lambda}} \abs[\Big]{\int_0^t e^{ki\rphase(s)-\nicefrac{s}{2}} S_k(s) \diff{\rBM}(s)} > x} \leq 4 \exp\Bigl( - \frac{\beta\lambda e x^2}{2} \Bigr).
\]
Taking \(x\) large enough so that this exponential tail bound is less than \(\varepsilon\) yields~\eqref{eq.asymptotics.amp.drift} and completes the proof.
\end{proof}

We move on to an asymptotic description of the behavior of the phase \(\rphase\). 

\begin{proposition}
\label{prop.asymptotics.phase}
Let \(\rphase\) solve the SDE from Proposition~\ref{prop.polarcoords2}. If \(U \sim \uniform\mathbb{S}^1\) where \(\mathbb{S}^1 \subset \mathbb{C}\) denotes the unit circle, then \(e^{i\rphase(t)} \to U\) in law as \(t\to\infty\).  
\end{proposition}

\begin{proof}
The idea of this proof is the same as that of the analog result in the context of the soft edge to bulk problem, namely \cite[Proposition~26]{painchaud_operator_2026}. 

Since \(\expect U^m = 0\) for any \(m \in \mathbb{Z} \setminus \{0\}\), by density of trigonometric polynomials in continuous functions on \(\mathbb{S}^1\) it suffices to prove that \(\expect e^{mi\rphase(t)} \to 0\) as \(t\to\infty\) for any \(m \in \mathbb{Z} \setminus \{0\}\). This is equivalent to showing that \(\varphi_m(t) \to 0\) as \(t\to\infty\) for any \(m \in \mathbb{Z} \setminus \{0\}\) if \(\varphi_m(t) \defeq \expect X_m(t)\) for \(X_m(t) \defeq \exp\bigl( mi\rphase(t) + 2mi\sqrt{\lambda} e^{\nicefrac{t}{2}} \bigr)\). 

Now, using the expression~\eqref{eq.asymptotics.expSDE} of the differential of \(e^{mi\rphase}\), we see that
\[
\diff{X_m}(t)
    = R_m(t) X_m(t) \diff{t} + S_m(t) X_m(t) \diff{\rBM}(t).
\]
This shows that \(\varphi_m(t) = \varphi_m(0) + \expect \int_0^t R_m(s) X_m(s) \diff{s}\), so that \(\varphi_m'(t) = \expect R_m(t) X_m(t)\). Setting \(\tilde{R}_m \defeq R_m + \frac{m^2}{4\beta}\) to be the oscillating part of \(R_m\) as a function of \(\rphase\), this yields \(\varphi_m'(t) = - \frac{m^2}{4\beta} \varphi_m(t) + \expect \tilde{R}_m(t) X_m(t)\), which integrates to
\[
\varphi_m(t) = e^{-\frac{m^2t}{4\beta}} \varphi_m(0) + e^{-\frac{m^2t}{4\beta}} \expect \int_0^t e^{\frac{m^2s}{4\beta}} \tilde{R}_m(s) X_m(s) \diff{s}.
\]
It is clear that the first term here vanishes as \(t\to\infty\), so by definition~\eqref{eq.asymptotics.defRS} of \(R_m\), it suffices to show that
\beq{eq.asymptotics.phase.toshow}
e^{-\frac{m^2t}{4\beta}} \expect \int_0^t e^{\frac{m^2s}{4\beta}} X_m(s) e^{ki\rphase(s)} \diff{s}
    \to 0
\quadtext{as}
t\to\infty
\eeq
for \(k \in \{2,4\}\). To prove this, we reuse the expression~\eqref{eq.asymptotics.expSDE} of the differential of \(e^{ki\rphase}\) to write
\begin{multline*}
e^{-\frac{m^2t}{4\beta}} \expect \int_0^t e^{\frac{m^2s}{4\beta}} X_m(s) e^{ki\rphase(s)} \diff{s}
    = \frac{1}{ki\sqrt{\lambda}} e^{-\frac{m^2t}{4\beta}} \expect \int_0^t e^{ki\rphase(s) + (\frac{m^2}{4\beta}-\frac{1}{2})s} R_k(s) X_m(s) \diff{s} \\
    - \frac{1}{ki\sqrt{\lambda}} e^{-\frac{m^2t}{4\beta}} \expect \int_0^t e^{(\frac{m^2}{4\beta}-\frac{1}{2})s} X_m(s) \diff{\bigl( e^{ki\rphase} \bigr)}(s).
\end{multline*}
Because \(X_m\) and \(R_k\) are bounded, the first integral is bounded by a constant times \(e^{\nicefrac{-t}{2}}\), so it vanishes as \(t\to\infty\). Then, integrating by parts, we get
\begin{multline*}
e^{-\frac{m^2t}{4\beta}} \expect \int_0^t e^{(\frac{m^2}{4\beta}-\frac{1}{2})s} X_m(s) \diff{\bigl( e^{ki\rphase} \bigr)}(s)
    = e^{-\frac{m^2t}{4\beta}} \expect e^{ki\rphase(s)+(\frac{m^2}{4\beta}-\frac{1}{2})s} X_m(s) \biggr\rvert_0^t \\
    - e^{-\frac{m^2t}{4\beta}} \expect \int_0^t e^{ki\rphase(s)+(\frac{m^2}{4\beta}-\frac{1}{2})s} \Bigl( R_m(s) + \frac{m^2}{4\beta} - \frac{1}{2} + S_k(s) S_m(s) \Bigr) X_m(s) \diff{s}.
\end{multline*}
Like earlier, because \(X_m\), \(R_m\), \(S_k\) and \(S_m\) are bounded, this is bounded by a constant times \(e^{\nicefrac{-t}{2}}\), so it vanishes as \(t\to\infty\). This finishes to prove~\eqref{eq.asymptotics.phase.toshow} and completes the proof.
\end{proof}

\subsubsection{Convergence of the Weyl--Titchmarsh functions}
\label{sec.spectralconv.beta>2.WT}

We now have all the tools to prove the convergence of the Bessel system's Weyl--Titchmarsh function to that of the sine system. Before carrying out the full proof, we give an overview of the setup and introduce some notation.

When \(\abs{a} < 1\), the shifted Bessel system \(\timechange' (\sBesselmat\circ\timechange)\) is limit circle at \(1\), like the sine system. In that case, for each \(z \in \mathbb{C}\) there is a solution \(h_z\colon [0,\infty) \to \mathbb{R}\) to \(\sBesselop h_z = zh_z\) that satisfies the boundary condition \(\lim_{t\to\infty} \Besselcoeff(t) h_z'(t) = 0\), and (as seen in Section~\ref{sec.sBesselCSrepresentation}) the canonical system's boundary condition at \(1\) takes the form \(\sBesselSLtoCS^{-1} \begin{smallpmatrix} \Besselcoeff h_z' \\ h_z \end{smallpmatrix} (\infty)\) where \(\sBesselSLtoCS\) is the matrix~\eqref{eq.sBesselSLtoCS} maps canonical system solutions to those of \(\sBesselop f = zf\).

When \(a \geq 1\), the Bessel system is limit point at its right endpoint. Recall that in that case we took \(\diln = 1 - \nicefrac{1}{\sqrt{\shift}}\) in the definition of \(\timechange\), so the system is actually defined on \([0, \nicefrac{1}{\diln}) \supsetneq [0,1]\). Following the remark below Theorem~\ref{thm.TMconv.LC}, we restrict it to \([0,1]\) and we add at the new right endpoint \(1\) the \(z\)-dependent \enquote{boundary condition} \(u_z\bigl( \timechange(1) \bigr)\) where \(u_z\) is an integrable solution to the original system on \((0,\infty)\), so that \(u_z\circ\timechange\) solves the system on \((0,\nicefrac{1}{\diln})\). By the standard construction of the canonical system corresponding to a Sturm--Liouville operator (see e.g.~\cite[Section~2.4]{painchaud_operator_2026}), this solution \(u_z\) must have the form \(u_z = \sBesselSLtoCS^{-1} \begin{smallpmatrix} \Besselcoeff h_z' \\ h_z \end{smallpmatrix}\) for an \(L^2\bigl( [0,\infty), \Besselweight(t) \diff{t} \bigr)\) solution \(h_z\) to \(\sBesselop h_z = zh_z\). 

In order to uniformize the analysis of the two cases, we will show that when \(\abs{a} < 1\) we can effectively move the boundary condition from \(1\) to \(1 - \nicefrac{1}{\sqrt{\shift}}\). Hence, in both cases, our goal is to show the convergence of \(u_z\bigl( \timechange(\lasttime) \bigr)\) to the sine system's boundary condition, where \(\lasttime \defeq (1 - \nicefrac{1}{\sqrt{\shift}}) / \diln\) as before and where \(u_z \defeq \sBesselSLtoCS^{-1} \begin{smallpmatrix} \Besselcoeff h_z' \\ h_z \end{smallpmatrix}\) for a specific solution \(h_z\) to \(\sBesselop h_z = zh_z\), either an integrable one (if \(a \geq 1\)) or one that satisfies the boundary condition at infinity (if \(\abs{a} < 1\)). This defines \(h_z\) up to a multiplicative constant, and to see the appropriate way to normalize it we make the boundary condition's representation a bit more explicit.

The matrix \(\sBesselSLtoCS\) is built out of the two solutions \(\sBesself\) and \(\sBesselg\). In the polar coordinates from Proposition~\ref{prop.polarcoords}, with \(\diffamps \defeq \ampg - \ampf\) and \(\diffphases \defeq \phaseg - \phasef\),
\begin{align*}
\shift^{\nicefrac{-1}{4}} \sBesself\circ\timechange
    & = \frac{\shift^{\nicefrac{-1}{4}} e^{\nicefrac{\timechange}{4}}}{\sqrt{\Besselcoeff\circ\timechange}} e^{\ampg-\diffamps} \cos(\phaseg - \diffphases) \\
    & = e^{-\diffamps} \cos\diffphases \, \shift^{\nicefrac{1}{4}} \sBesselg\circ\timechange + e^{-\diffamps} \sin\diffphases \frac{\shift^{\nicefrac{-1}{4}} e^{\nicefrac{\timechange}{4}}}{\sqrt{\Besselcoeff\circ\timechange}} e^{\ampg} \sin\phaseg.
\end{align*}
The second term can be simplified using the identity~\eqref{eq.Wronskianidentity}, which implies that \(e^{-\diffamps} \sin\diffphases = e^{-2\ampg}\). This further simplifies when evaluated at \(\lasttime\), since \(\timechange(\lasttime) = \log\shift\). Thus,
\begin{align*}
\shift^{\nicefrac{-1}{4}} \sBesself(\log\shift)
    & = e^{-\diffamps(\lasttime)} \cos\diffphases(\lasttime) \shift^{\nicefrac{1}{4}} \sBesselg(\log\shift) + \frac{1}{\sqrt{\Besselcoeff(\log\shift)}} e^{-\ampg(\lasttime)} \sin\phaseg(\lasttime),
\intertext{and it can be shown in the same way that}
\shift^{\nicefrac{-1}{4}} \sBesself'(\log\shift)
    & = e^{-\diffamps(\lasttime)} \cos\diffphases(\lasttime) \shift^{\nicefrac{1}{4}} \sBesselg'(\log\shift) - \frac{1}{\sqrt{\Besselcoeff(\log\shift)}} e^{-\ampg(\lasttime)} \cos\phaseg(\lasttime).
\end{align*}
Using these representations, we can write
\begin{align}
u_z(\log\shift)
    & = \begin{pmatrix}
        \shift^{\nicefrac{-1}{4}} \sBesself & -\shift^{\nicefrac{-1}{4}} \Besselcoeff\sBesself' \\
        -\shift^{\nicefrac{1}{4}} \sBesselg & \shift^{\nicefrac{1}{4}} \Besselcoeff\sBesselg'
    \end{pmatrix} \begin{pmatrix}
        \Besselcoeff h_z' \\
        h_z
    \end{pmatrix} (\log\shift)
\label{eq.spectralconv.uz.def} \\
    \begin{split}
    & = \wronskian{h_z}{\shift^{\nicefrac{1}{4}}\sBesselg}(\log\shift) \begin{pmatrix}
        - e^{-\diffamps(\lasttime)} \cos\diffphases(\lasttime) \\
        1
    \end{pmatrix} \\
    &\hspace*{22mm} + e^{-\ampg(\lasttime)} \sqrt{\Besselcoeff(\log\shift)} \begin{pmatrix}
        \sin\phaseg(\lasttime) & \cos\phaseg(\lasttime) \\
        0 & 0 
    \end{pmatrix} \begin{pmatrix}
        h_z'(\log\shift) \\
        h_z(\log\shift)
    \end{pmatrix}
    \end{split}
\label{eq.spectralconv.uz}
\end{align}
where \(\wronskian{h_z}{\shift^{\nicefrac{1}{4}}\sBesselg} = \shift^{\nicefrac{1}{4}} \Besselcoeff (h_z \sBesselg' - h_z' \sBesselg)\) is the Wronskian of \(h_z\) and \(\shift^{\nicefrac{1}{4}} \sBesselg\). Recall that the right boundary condition of the sine system is \((\Re\HBM(\infty), 1)\) and that we have shown in Section~\ref{sec.coupling.HBMconvergence} that \(-e^{-\diffamps} \cos\diffphases\) converges to \(\Re\HBM\) run in logarithmic time. Hence, we expect the vector in the first term of~\eqref{eq.spectralconv.uz} to converge to this boundary condition if the prefactor is removed, and we need to normalize \(h_z\) so that this prefactor is set to \(1\).

To do so, we exploit the following property of \(h_z\): since it solves \(\sBesselop h_z = zh_z\), by Proposition~\ref{prop.unshift} the function \(\tilde{h}_z(t) \defeq h_z(t + \log\shift)\) solves \(\unsBesselop \tilde{h}_z = (1 + \frac{z}{2\sqrt{\shift}}) \tilde{h}_z\) where \(\unsBesselop = \Besselop\) in law, but is defined from a different Bownian motion than \(\sBesselop\) is. This equivalence extends to the whole real line if we make the Brownian motions two-sided. Now, we fix a solution \(\Phi_z\) to \(\unsBesselop\Phi_z = z \Phi_z\) with the right behavior at infinity, so that \(\tilde{h}_z = \gamma\Phi_{1+\nicefrac{z}{2\sqrt{\shift}}}\) for some (possibly random) scaling \(\gamma\). Note that we may (and do) choose \(\Phi_z\) to be analytic in \(z\) by properties of Sturm--Liouville operators (see e.g.~\cite[Theorem~2.7]{eckhardt_weyl-titchmarsh_2013}). Then with \(\unsBesselg(t) \defeq \sBesselg(t + \log\shift)\),
\beq{eq.wronskian.logE}
\wronskian{h_z}{\shift^{\nicefrac{1}{4}}\sBesselg}(\log\shift)
    = \gamma\Besselcoeff(\log\shift) \unswronskian{\Phi_{1+\nicefrac{z}{2\sqrt{\shift}}}}{\shift^{\nicefrac{1}{4}}\unsBesselg}(0),
\eeq
where \(\unswronskian{f}{g} \defeq \unsBesselcoeff (fg' - f'g)\) is the Wronskian for \(\unsBesselop\). Taking \(\gamma \defeq \bigl( \Besselcoeff(\log\shift) \unswronskian{\Phi_{1+\nicefrac{z}{2\sqrt{\shift}}}}{\shift^{\nicefrac{1}{4}}\unsBesselg}(0) \bigr)^{-1}\) thus normalizes \(h_z\) as desired (i.e., it sets \(\wronskian{h_z}{\shift^{\nicefrac{1}{4}}\sBesselg}(\log\shift) = 1\)), but it is not obvious that this can be done. This leads us to Lemma~\ref{lem.spectralconv.beta>2.goodevent} below, which gives a good event where the solution \(h_z\) can be normalized as desired. 

\begin{lemma}
\label{lem.spectralconv.beta>2.goodevent}
If \(\beta > 2\) and \(K \subset \mathbb{C}\) is compact, then \(\bprob[\big]{\forall z \in K, \Besselcoeff(\log\shift) \abs[\big]{\unswronskian{\Phi_{1+\nicefrac{z}{2\sqrt{\shift}}}}{\shift^{\nicefrac{1}{4}} \unsBesselg}(0)} \geq \shift^{\nicefrac{-a}{2}}} \to 1\) as \(\shift\to\infty\), where all symbols are as in~\eqref{eq.wronskian.logE}.
\end{lemma}

This allows us to complete the proof of the convergence of the Weyl--Titchmarsh functions. To avoid interrupting the argument too much we postpone the proof of Lemma~\ref{lem.spectralconv.beta>2.goodevent} until the end of the section.

\begin{theorem}
\setshift{E_n}
Suppose \(\beta > 2\), and let \(\sBesselWT, \sineWT\colon \UHP \to \clUHP\) be the Weyl--Titchmarsh functions of the shifted Bessel and sine canonical systems, both defined on the probability space from Lemma~\ref{lem.coupling}. Then \(\sBesselWT \to \sineWT\) compactly on \(\UHP\) in probability as \(n\to\infty\).
\end{theorem}

\begin{proof}
As stated before, since the transfer matrices of these canonical systems convergence compactly in probability by Corollary~\ref{cor.TMconvP}, it suffices to prove that the Bessel system's boundary condition at \(1\) converges in probability to the boundary condition of the sine system by Theorem~\ref{thm.TMconv.LC}. When \(a \geq 1\), the boundary condition we are talking about here is the \(z\)-dependent one of the truncated system, and the convergence must hold compactly in \(z\). 

The boundary condition of the Bessel system is given by \(u_z\circ\timechange(1)\) where \(u_z \defeq \sBesselSLtoCS^{-1} \begin{smallpmatrix} \Besselcoeff h_z' \\ h_z \end{smallpmatrix}\) for
\beq{eq.spectralconv.hz}
h_z \defeq \frac{1}{\Besselcoeff(\log\shift) \unswronskian{\Phi_{1+\nicefrac{z}{2\sqrt{\shift}}}}{\shift^{\nicefrac{1}{4}}\unsBesselg}(0)}\, \Phi_{1+\nicefrac{z}{2\sqrt{\shift}}}
\eeq
where \(\Phi_z\) is a fixed solution to \(\unsBesselop\Phi_z = z\Phi_z\) with the right behavior at infinity, and chosen to be analytic in \(z\). Fixing a compact \(K \subset \mathbb{C}\), by Lemma~\ref{lem.spectralconv.beta>2.goodevent}, \(h_z\) is well defined at least asymptotically almost surely for \(z \in K\), and we will not worry about what happens outside of the event from the lemma.

In order to uniformize the analysis of the two cases (\(\abs{a} < 1\) and \(a \geq 1\)), we first show that when \(\abs{a} < 1\) we can effectively move the boundary condition from \(1\) to \(\lasttime\), that is, we show that \(u_z\circ\timechange(1) - u_z\circ\timechange(\lasttime) \to 0\) as \(\shift\to\infty\). By definition,
\[
u_z
    = \begin{pmatrix}
        \shift^{\nicefrac{-1}{4}} \sBesself & -\shift^{\nicefrac{-1}{4}} \Besselcoeff\sBesself' \\
        -\shift^{\nicefrac{1}{4}} \sBesselg & \shift^{\nicefrac{1}{4}} \Besselcoeff\sBesselg'
    \end{pmatrix} \begin{pmatrix}
        \Besselcoeff h_z' \\
        h_z
    \end{pmatrix}
    = \begin{pmatrix}
        \wronskian{\shift^{\nicefrac{-1}{4}}\sBesself}{h_z} \\
        \wronskian{h_z}{\shift^{\nicefrac{1}{4}}\sBesselg}
    \end{pmatrix},
\]
so it suffices to show that if \(\psi\) stands for either \(\shift^{\nicefrac{-1}{4}} \sBesself\) or \(\shift^{\nicefrac{1}{4}} \sBesselg\), then \(\wronskian{h_z}{\psi}(\infty) - \wronskian{h_z}{\psi}(\log\shift) \to 0\) in probability as \(\shift\to\infty\) (which is not entirely trivial, as the solutions themselves depend on \(\shift\)). To prove this, note that because \(\sBesselop\psi = 0\) and \(\sBesselop h_z = zh_z\), then \(\bigl( \wronskian{h_z}{\psi} \bigr)' = \frac{z\sqrt{\shift}}{2} h_z \psi \Besselweight\) so
\[
\wronskian{h_z}{\psi}(\infty) - \wronskian{h_z}{\psi}(\log\shift) 
    = \frac{z\sqrt{\shift}}{2} \int_{\log\shift}^\infty h_z(t) \psi(t) \Besselweight(t) \diff{t}.
\]
To estimate this, we can make the dependence on \(\shift\) clearer by shifting time and using \(\tilde{h}_z(t) \defeq h_z(t+\log\shift)\) and \(\tilde{\psi}(t) \defeq \psi(t+\log\shift)\). By Proposition~\ref{prop.unshift}, \(\tilde{h}_z\) and \(\tilde{\psi}\) solve \(\unsBesselop \tilde{h}_z = (1 + \frac{z}{2\sqrt{\shift}}) \tilde{h}_z\) and \(\unsBesselop\tilde{\psi} = \tilde{\psi}\) where \(\unsBesselop\) is defined from the Brownian motion \(\unsBM \defeq \sBM(\cdot+\log\shift) - \sBM(\log\shift)\). In particular, the weight of \(\unsBesselop\) is related to that of \(\Besselop\) by \(\unsBesselweight(t) = \shift^{a+1} e^{\frac{2}{\sqrt{\beta}} \sBM(\log\shift)} \Besselweight(t+\log\shift)\). Hence, shifting time by \(\log\shift\) in the above integral yields
\[
\wronskian{h_z}{\psi}(\infty) - \wronskian{h_z}{\psi}(\log\shift)
    = \frac{z}{2} \shift^{\nicefrac{-1}{2}-a} e^{-\frac{2}{\sqrt{\beta}} \sBM(\log\shift)} \int_0^\infty \tilde{h}_z(t) \tilde{\psi}(t) \unsBesselweight(t) \diff{t}.
\]
Recall from Section~\ref{sec.solprops} that \(1\) and \(\Besselsol(t) \defeq \int_0^t \frac{1}{\unsBesselcoeff(s)} \diff{s}\) are a pair of fundamental solutions to \(\unsBesselop f = 0\). Therefore, the argument we used in the proof of Lemma~\ref{lem.tracebound} shows that for any \(\varepsilon > 0\) there is a \(C > 0\) such that with probability at least \(1 - \nicefrac{\varepsilon}{4}\), for all \(t \geq 0\),
\begin{align*}
\abs{\tilde{\psi}(t)}
    & \leq C\bigl( \abs{\psi(\log\shift)} + \abs{\psi'(\log\shift)} \bigr) \bigl( 1 + \Besselsol(t) \bigr)
\shortintertext{and}
\abs{\tilde{h}_z(t)}
    & \leq C \Bigl( 1 + \frac{\abs{z}}{2\sqrt{\shift}} \Bigr) \bigl( \abs{\tilde{h}_z(0)} + \abs{\tilde{h}_z'(0)} \bigr) \bigl( 1 + \Besselsol(t) \bigr).
\end{align*}
By the representation in polar coordinates of \(\psi\) with the estimates on \(e^{2\ampf}\) and \(e^{2\ampg}\) from Proposition~\ref{prop.tracebound}, for any \(\delta > 0\) we know that there is a \(C' > 0\) such that \(\abs{\psi(\log\shift)} \vee \abs{\psi'(\log\shift)} \leq C' \shift^{\nicefrac{1}{2\beta} + \nicefrac{\delta}{2}} \bigl( \Besselcoeff(\log\shift) \bigr)^{\nicefrac{-1}{2}}\) with probability at least \(1 - \nicefrac{\varepsilon}{4}\). Note that by definition, \(\bigl( \Besselcoeff(\log\shift) \bigr)^{\nicefrac{-1}{2}} =  \shift^{\nicefrac{a}{2}} \exp\bigl( \frac{1}{\sqrt{\beta}} \sBM(\log\shift) \bigr)\). Then, on the event from Lemma~\ref{lem.spectralconv.beta>2.goodevent}, \(\abs{\tilde{h}_z(0)} \leq \shift^{\nicefrac{a}{2}} \abs{\Phi_{1+\nicefrac{z}{2\sqrt{\shift}}}(0)}\) and \(\abs{\tilde{h}_z'(0)} \leq \shift^{\nicefrac{a}{2}} \abs{\Phi_{1+\nicefrac{z}{2\sqrt{\shift}}}'(0)}\). Because \(\Phi_z\) is analytic in \(z\), \(\Phi_{1+\nicefrac{z}{2\sqrt{\shift}}}(0) = \Phi_1(0) + O(\nicefrac{\abs{z}}{\sqrt{\shift}})\) and \(\Phi_{1+\nicefrac{z}{2\sqrt{\shift}}}'(0) = \Phi_1'(0) + O(\nicefrac{\abs{z}}{\sqrt{\shift}})\), and since \(\Phi_1\) is a solution to \(\unsBesselop \Phi_1 = \Phi_1\) defined by its behavior at infinity, its law does not depend on \(\shift\). It follows that there is a \(C'' > 0\) such that \(\abs{\tilde{h}_z(0)} \vee \abs{\tilde{h}_z'(0)} \leq C'' \shift^{\nicefrac{a}{2}}\) for all \(z \in K\) with probability at least \(1 - \nicefrac{\varepsilon}{4}\). Combining these estimates, we see that for a (different) \(C > 0\), with probability at least \(1 - \nicefrac{3\varepsilon}{4}\), for all \(z \in K\),
\[
\abs[\big]{\wronskian{h_z}{\psi}(\infty) - \wronskian{h_z}{\psi}(\log\shift)}
    \leq C \shift^{\nicefrac{1}{2\beta}-\nicefrac{1}{2}+\nicefrac{\delta}{2}} e^{-\frac{1}{\sqrt{\beta}} \sBM(\log\shift)} \int_0^\infty \bigl( 1 + \Besselsol(t) \bigr)^2 \unsBesselweight(t) \diff{t}.
\]
Since \(\abs{a} < 1\) here, \(1, \Besselsol \in L^2\bigl( [0,\infty), \unsBesselweight(t) \diff{t} \bigr)\), so the integral is finite. Then for \(\shift\) large enough \(e^{-\frac{1}{\sqrt{\beta}} \sBM(\log\shift)} \leq \shift^{\nicefrac{\delta}{2}}\) with probability at least \(1 - \nicefrac{\varepsilon}{4}\) by properties of Brownian motion, so finally the difference is bounded by a constant times \(\shift^{- (\nicefrac{1}{2} - \nicefrac{1}{2\beta}) + \delta}\) with probability at least \(1 - \varepsilon\). Since \(\beta > 2\) here, this vanishes as \(\shift\to\infty\) for \(\delta\) small enough. Therefore, the difference of the Wronskians converges to \(0\) in probability, and indeed \(u_z(\infty) - u_z(\log\shift) \to 0\) compactly in probability as \(\shift\to\infty\). 

It remains to show (now for any \(a > -1\)) that \(u_z(\log\shift)\) converges to the boundary condition of the sine system, namely \((\Re\HBM(\infty), 1)\) where \(\HBM\) is the hyperbolic Brownian motion from the sine system's coefficient matrix. Writing \(u_z(\log\shift)\) as in~\eqref{eq.spectralconv.uz} but using the expression~\eqref{eq.spectralconv.hz} of \(h_z\) in terms of \(\Phi_{1+\nicefrac{z}{2\sqrt{\shift}}}\),
\beq{eq.spectralconv.uz2}
\setlength{\multlinegap}{22mm}
\begin{multlined}
u_z(\log\shift)
    = \begin{pmatrix}
        - e^{-\diffamps(\lasttime)} \cos\diffphases(\lasttime) \\
        1
    \end{pmatrix} \\
    + \frac{\shift^{\nicefrac{-a}{2}} e^{-\ampg(\lasttime)-\frac{1}{\sqrt{\beta}} \sBM(\log\shift)}}{\Besselcoeff(\log\shift) \unswronskian{\Phi_{1+\nicefrac{z}{2\sqrt{\shift}}}}{\shift^{\nicefrac{1}{4}}\unsBesselg}(0)} \begin{pmatrix}
        \sin\phaseg(\lasttime) & \cos\phaseg(\lasttime) \\
        0 & 0 
    \end{pmatrix} \begin{pmatrix}
        \Phi_{1+\nicefrac{z}{2\sqrt{\shift}}}'(0) \\
        \Phi_{1+\nicefrac{z}{2\sqrt{\shift}}}(0)
    \end{pmatrix}.
\end{multlined}
\eeq
The second term here vanishes in the limit in probability. Indeed, for any \(\varepsilon > 0\), Proposition~\ref{prop.tracebound} shows that if \(\delta > 0\) then \(e^{-\ampg(\lasttime)} \leq C\shift^{\nicefrac{-1}{2\beta}+\nicefrac{\delta}{2}}\) for some \(C > 0\) with probability at least \(1-\nicefrac{\varepsilon}{3}\), as above \(e^{-\frac{1}{\sqrt{\beta}} \sBM(\log\shift)} \leq \shift^{\nicefrac{\delta}{2}}\) with probability at least \(1-\nicefrac{\varepsilon}{3}\), and the denominator of the prefactor is at least \(\shift^{\nicefrac{-a}{2}}\) on the event from Lemma~\ref{lem.spectralconv.beta>2.goodevent}. As above, \(\Phi_{1+\nicefrac{z}{2\sqrt{\shift}}}(0) = \Phi_1(0) + O(\nicefrac{\abs{z}}{\sqrt{\shift}})\) and the law of \(\Phi_1(0)\) does not depend on \(\shift\), and the same goes for \(\Phi_{1+\nicefrac{z}{2\sqrt{\shift}}}'(0)\). It follows that for any \(\shift\) large enough the second term of~\eqref{eq.spectralconv.uz2} is bounded (uniformly for \(z \in K\)) by a constant times \(\shift^{\nicefrac{-1}{2\beta}+\delta}\) with probability at least \(1 - \varepsilon\). Taking \(\delta\) small enough, since \(\varepsilon\) is arbitrary it follows that the second term of~\eqref{eq.spectralconv.uz2} converges to \(0\) compactly in \(z\) in probability as \(\shift\to\infty\).

The above immediately implies that the second entry of \(u_z(\log\shift)\) converges to \(1\), which is the desired value. To show that \(-e^{-\diffamps(\lasttime)} \cos\diffphases(\lasttime) \to \Re\HBM(\infty)\), we set \(\penultime \defeq (1 - \shift^{\nicefrac{-1}{2}+\alpha}) / \diln\), and we write
\begin{subequations}
\label{eq.spectralconv.bcdiff}
\begin{align}
\abs[\big]{e^{-\diffamps(\lasttime)}\cos\diffphases(\lasttime) + \Re\HBM(\infty)}
    & \leq \abs[\big]{e^{-\diffamps(\lasttime)}\cos\diffphases(\lasttime) - e^{-\diffamps(\penultime)}\cos\diffphases(\penultime)} 
    \label{eq.spectralconv.bcdiff.almostHBEdiff} \\
    &\qquad + \abs[\big]{e^{-\diffamps(\penultime)}\cos\diffphases(\penultime) + \Re\HBM\circ\slogtime(\penultime)} 
    \label{eq.spectralconv.bcdiff.HBMconv} \\
    &\qquad + \abs[\big]{\Re\HBM\circ\slogtime(\penultime) - \Re\HBM(\infty)}.
    \label{eq.spectralconv.bcdiff.HBMdiff} 
\end{align}
\end{subequations}
As \(\slogtime(\penultime) = (\frac{1}{2}-\alpha)\log\shift\), \eqref{eq.spectralconv.bcdiff.HBMdiff} vanishes a.s.\@ in the limit \(\shift\to\infty\). It is also clear that~\eqref{eq.spectralconv.bcdiff.HBMconv} vanishes in probability as \(\shift\to\infty\) by Proposition~\ref{prop.ReHBM}. To show that~\eqref{eq.spectralconv.bcdiff.almostHBEdiff} also vanishes, we return to the SDE
\beq{eq.spectralconv.ReHBM}
\setlength{\multlinegap}{33mm}
\begin{multlined}
\diff{\bigl( -e^{-\diffamps} \cos\diffphases \bigr)}(t)
    = \frac{2}{\sqrt{\beta}} e^{-2\ampg(t)} \sin 2\phaseg(t) \sqrt{\frac{2\diln}{1-\diln t}} \diff{\sBM}(t) \\
    + e^{-2\ampg(t)} \Bigl( (2a + 1) \sin 2\phaseg(t) + \frac{4}{\beta} \sin 4\phaseg(t) \Bigr) \frac{\diln}{1-\diln t} \diff{t},
\end{multlined}
\eeq
which was obtained earlier by taking the real part of~\eqref{eq.approxHBMSDE} and simplifying with the identity~\eqref{eq.Wronskianidentity}. By Proposition~\ref{prop.tracebound}, we know that for any \(\varepsilon, \delta > 0\), there is a \(C > 0\) such that \(e^{-\ampg(t)} \leq C(1-\diln t)^{\nicefrac{1}{\beta}-\nicefrac{\delta}{2}}\) for \(t \in [0,\lasttime]\) with probability at least \(1-\varepsilon\). It follows that
\[
\abs[\bigg]{\int_{\penultime}^{\lasttime} e^{-2\ampg(t)} \Bigl( (2a + 1) \sin 2\phaseg(t) + \frac{4}{\beta} \sin 4\phaseg(t) \Bigr) \frac{\diln}{1-\diln t} \diff{t}}
    \leq C' \int_{\penultime}^{\lasttime} \diln (1 - \diln t)^{-1+\nicefrac{2}{\beta}-\delta} \diff{t}
\]
with probability at least \(1-\varepsilon\) for a different constant \(C' > 0\). This integrates to be bounded by a constant times \(\shift^{-(\nicefrac{1}{2}-\alpha)(\nicefrac{2}{\beta}-\delta)}\), which vanishes as \(\shift\to\infty\) for \(\delta\) small enough. Similarly, with probability at least \(1-\varepsilon\), the quadratic variation of the first line of~\eqref{eq.spectralconv.ReHBM} is bounded on \([\penultime, \lasttime]\) by
\[
C \int_{\penultime}^{\lasttime} \diln (1-\diln t)^{-1+\nicefrac{4}{\beta}-2\delta} \diff{t}
    = \frac{C}{\nicefrac{4}{\beta}-2\delta} \shift^{-(\nicefrac{1}{2}-\alpha)(\nicefrac{4}{\beta}-2\delta)} \bigl( 1 - \shift^{-\alpha(\nicefrac{4}{\beta}-2\delta)} \bigr),
\]
which vanishes as \(\shift\to\infty\). Applying Bernstein's inequality, it follows that the integral of~\eqref{eq.spectralconv.ReHBM} between \(\penultime\) and \(\lasttime\) converges to \(0\) in probability as \(\shift\to\infty\), and this concludes the proof.
\end{proof}

\begin{proof}[Proof of Lemma~\ref{lem.spectralconv.beta>2.goodevent}]
We need to estimate \(\Besselcoeff(\log\shift) \unswronskian{\Phi_{1+\nicefrac{z}{2\sqrt{\shift}}}}{\shift^{\nicefrac{1}{4}}\unsBesselg}(0)\) for \(z \in K\) where \(K \subset \mathbb{C}\) is compact.

We first replace \(\Phi_{1+\nicefrac{z}{2\sqrt{\shift}}}\) with \(\Phi_1\). Fix \(\varepsilon > 0\). As we have seen before, Proposition~\ref{prop.tracebound} implies that for any \(\delta > 0\), there is a \(C > 0\) such that
\(
\shift^{\nicefrac{1}{4}}\unsBesselg(0) \vee \shift^{\nicefrac{1}{4}}\unsBesselg'(0)
    \leq C \shift^{\nicefrac{1}{2\beta}+\nicefrac{\delta}{2}} \bigl( \Besselcoeff(\log\shift) \bigr)^{\nicefrac{-1}{2}}
\)
with probability at least \(1 - \nicefrac{\varepsilon}{6}\). Then, by definition, \(\bigl( \Besselcoeff(\log\shift) \bigr)^{\nicefrac{1}{2}} = \shift^{\nicefrac{-a}{2}} \exp\bigl(- \frac{1}{\sqrt{\beta}} \sBM(\log\shift) \bigr)\), and by properties of Brownian motion this is bounded by \(\shift^{\nicefrac{-a}{2}+\nicefrac{\delta}{2}}\) with probability at least \(1 - \nicefrac{\varepsilon}{6}\) for any \(\shift\) large enough. Finally, as \(\Phi_z\) is analytic in \(z\), \(\Phi_{1+\nicefrac{z}{2\sqrt{\shift}}}(0) - \Phi_1(0) = O(\nicefrac{\abs{z}}{\sqrt{\shift}})\) and \(\Phi'_{1+\nicefrac{z}{2\sqrt{\shift}}}(0) - \Phi'_1(0) = O(\nicefrac{\abs{z}}{\sqrt{\shift}})\) where the implicit constants are well-defined random variables, so with probability at least \(1 - \nicefrac{\varepsilon}{6}\) these differences are bounded by \(C'\shift^{\nicefrac{-1}{2}}\) for some constant \(C' > 0\), uniformly in \(z \in K\). It follows that with probability at least \(1 - \nicefrac{\varepsilon}{2}\), there is a \(C > 0\) such that
\beq{eq.spectralconv.goodevent.zpart}
\abs[\big]{\Besselcoeff(\log\shift) \unswronskian{\Phi_{1+\nicefrac{z}{2\sqrt{\shift}}} - \Phi_1}{\shift^{\nicefrac{1}{4}}\unsBesselg}(0)}
    \leq C\shift^{\nicefrac{-1}{2}+\nicefrac{1}{2\beta}-\nicefrac{a}{2}+\delta}
\eeq
for all \(z \in K\) and all \(\shift\) large enough.

Now, because \(\Phi_1\) and \(\unsBesselg\) both solve \(\unsBesselop f = f\), their Wronskian is constant by standard theory of Sturm--Liouville operators. Therefore, by definition of \(\unsBesselg\), \(\unswronskian{\Phi_1}{\shift^{\nicefrac{1}{4}}\unsBesselg}(0) = \unswronskian{\Phi_1}{\shift^{\nicefrac{1}{4}}\unsBesselg}(-\log\shift) = \shift^{\nicefrac{1}{4}} \unsBesselcoeff(-\log\shift) \Phi_1(-\log\shift)\). In the polar coordinates from Proposition~\ref{prop.polarcoords2}, \(\Phi_1(-t) = e^{\ramp(t) - \nicefrac{t}{4}} \cos\rphase(t)\), so
\[
\Besselcoeff(\log\shift) \unswronskian{\Phi_1}{\shift^{\nicefrac{1}{4}}\unsBesselg}(0)
    = \Besselcoeff(\log\shift) \unsBesselcoeff(-\log\shift) e^{\ramp(\log\shift)} \cos\rphase(\log\shift).
\]
Because \(\unsBesselcoeff\) is driven by the Brownian motion \(\sBM(\cdot + \log\shift) - \sBM(\log\shift)\), by definition \(\Besselcoeff(\log\shift) \unsBesselcoeff(-\log\shift) = 1\). Then, by Proposition~\ref{prop.asymptotics.amp}, we know that \(e^{\ramp(\log\shift)} \geq C\shift^{\nicefrac{1}{2\beta} - \nicefrac{a}{2} - \delta} e^{\ramp(0)}\) with probability at least \(1 - \nicefrac{\varepsilon}{6}\) for some \(C > 0\). As \(\ramp(0)\) is a well-defined random variable, we can choose \(\zeta > 0\) small enough so that \(e^{\ramp(0)} > \zeta\) with probability at least \(1 - \nicefrac{\varepsilon}{6}\). Proposition~\ref{prop.asymptotics.phase} shows that \(e^{i\rphase(t)} \to U \sim \uniform\mathbb{S}^1\) in law as \(t\to\infty\), so shrinking \(\zeta\) if necessary we also see that \(\abs{\cos\rphase(\log\shift)} \geq \zeta\) with probability at least \(1 - \nicefrac{\varepsilon}{6}\) for any \(\shift\) large enough. Hence,
\beq{eq.spectralconv.goodevent.1part}
\abs[\big]{\Besselcoeff(\log\shift) \unswronskian{\Phi_1}{\shift^{\nicefrac{1}{4}}\unsBesselg}(0)}
    \geq C\zeta^2 \shift^{\nicefrac{1}{2\beta}-\nicefrac{a}{2}-\delta}
\eeq
with probability at least \(1 - \nicefrac{\varepsilon}{2}\), for all \(\shift\) large enough. 

Putting together~\eqref{eq.spectralconv.goodevent.zpart} and~\eqref{eq.spectralconv.goodevent.1part}, we see that with probability at least \(1 - \varepsilon\),
\[
\shift^{\nicefrac{a}{2}} \abs[\big]{\Besselcoeff(\log\shift) \unswronskian{\Phi_{1+\nicefrac{z}{2\sqrt{\shift}}}}{\shift^{\nicefrac{1}{4}}\unsBesselg}(0)}
    \geq C \zeta^2 \shift^{\nicefrac{1}{2\beta} - \delta} - C \shift^{\nicefrac{-1}{2}+\nicefrac{1}{2\beta}+\delta}
\]
for some \(C > 0\) and all \(\shift\) large enough. Because \(\beta > 2\) here, with \(\delta\) small enough the first term diverges as \(\shift\to\infty\) while the second one vanishes. In particular, the right-hand side is at least \(1\) for any \(\shift\) large enough, hence the result.
\end{proof}

\printbibliography

\end{document}